%% file: main.tex
\documentclass[a4paper,12pt]{article}

\usepackage[english]{babel}
\input{preambule.tex}

\usepackage{listings}
\usepackage{color}

\definecolor{deepblue}{rgb}{0,0,0.5}
\definecolor{deepred}{rgb}{0.6,0,0}
\definecolor{deepgreen}{rgb}{0,0.5,0}
\DeclareFixedFont{\ttb}{T1}{txtt}{bx}{n}{9} 
\DeclareFixedFont{\ttm}{T1}{txtt}{m}{n}{9}  
\newcommand\pythonstyle{\lstset{
language=Python,
basicstyle=\ttm,
morekeywords={self},              
keywordstyle=\ttb\color{deepblue},
emph={MyClass,__init__},          
emphstyle=\ttb\color{deepred},    
stringstyle=\color{deepgreen},
frame=tb,                         
showstringspaces=false
}}
\lstnewenvironment{python}[1][]
{
\pythonstyle
\lstset{#1}
}
{}

\newcommand\pythoninline[1]{{\pythonstyle\lstinline!#1!}}

\usepackage[letterpaper,top=2cm,bottom=2cm,left=3cm,right=3cm,marginparwidth=1.75cm]{geometry}
\usepackage{kpfonts}

\newcommand{\minipdf}[2]{\begin{minipage}{#1\textwidth}\center{\resizebox{0.9\textwidth}{!}{\includegraphics{#2}}}\end{minipage}}
\usepackage{graphicx}
\usepackage[colorlinks=true, allcolors=blue]{hyperref}
\usepackage{algpseudocode}
\usepackage{algorithm}

\newenvironment{myproof}[2] {\paragraph{Proof of {#1} {#2} :}}{\hfill$\square$}

\newtheorem{theorem}{Theorem}
\newtheorem{example}{Example}
\newtheorem{remark}[theorem]{Remark}
\newtheorem{definition}[theorem]{Definition}
\newcommand{\TT}{\mathbin{\rotatebox[origin=c]{90}{$\vdash$}}}

\title{Higher dimensional floorplans and Baxter $d$-permutations}
\author{Nicolas Bonichon and Thomas Muller and Adrian Tanasa}

\begin{document}
\sloppy
\maketitle

\begin{abstract}
    A $2-$dimensional mosaic floorplan is a partition of a rectangle by other rectangles with no empty rooms. These partitions (considered up to some deformations) are known to be in bijection with Baxter permutations.
    A $d$-floorplan is the generalisation of mosaic floorplans in higher dimensions, and a $d$-permutation is a $(d-1)$-tuple of permutations. Recently, in
    N. Bonichon and P.-J. Morel, {\it J. Integer Sequences} 25 (2022), Baxter $d$-permutations generalising the usual Baxter permutations were introduced.

    In this paper, we consider mosaic floorplans in arbitrary dimensions, and we construct a generating tree for $d$-floorplans, which generalises  the known generating tree structure for $2$-floorplans. The corresponding labels and rewriting rules appear to be significantly more involved in higher dimensions.
    Moreover we give a bijection between the $2^{d-1}$-floorplans and $d$-permutations characterized by forbidden vincular patterns. Surprisingly, this set of $d$-permutations is strictly contained within the set of Baxter $d$-permutations.
    
\end{abstract}

\tableofcontents


\section{Introduction}

A \emph{floorplan} of size $n$ is a partition of a rectangle using  $n$ interior-disjoint rectangles. These combinatorial objects have been studied in various fields of computer science, architecture or discrete geometry.

The boundaries of the rectangles of a floorplan define a set of horizontal and vertical straight lines. A floorplan is \emph{generic} if the union of horizontal (resp. vertical) boundary straight lines  that share the same $y$-coordinate (resp. $x$-coordinate) is a single line. A \emph{segment} is a maximal (horizontal or vertical) straight line in the inner boundaries.
A \emph{mosaic floorplan} is a generic floorplan with no segment crossing. This condition is called the \emph{tatami condition}.

In this paper we investigate a natural generalization of floorplans to higher
dimensions. A \emph{$d$-dimensional floorplan} (or a \emph{box partition}) is a partition
of a $d$-dimensional hyperrectangle with $n$ interior-disjoint $d$-dimensional
hyperrectangles (called blocks).
The boundaries of the hyperrectangles of a $d$-dimensional floorplan define a set of $(d-1)$-hyperrectangles. A \emph{$d$-dimensional floorplan} is \emph{generic} if the set of boundaries that share the same $i$-th coordinate is a single $(d-1)$-hyperrectangle. A \emph{border} is a maximal $(d-1)$-hyperrectangle of the interior of the bounding $d$-hyperrectangle.
A \emph{$d$-floorplan} is a generic $d$-dimensional floorplan that has no border crossing (this being the \emph{tatami condition} in higher dimensions).
These objects have already been considered in the literature. In the case $d=3$, they are called generic boxed Plattenbau in~\cite{felsner2020plattenbauten}.

\begin{figure}[!htb]
    \center{\minipdf{0.30}{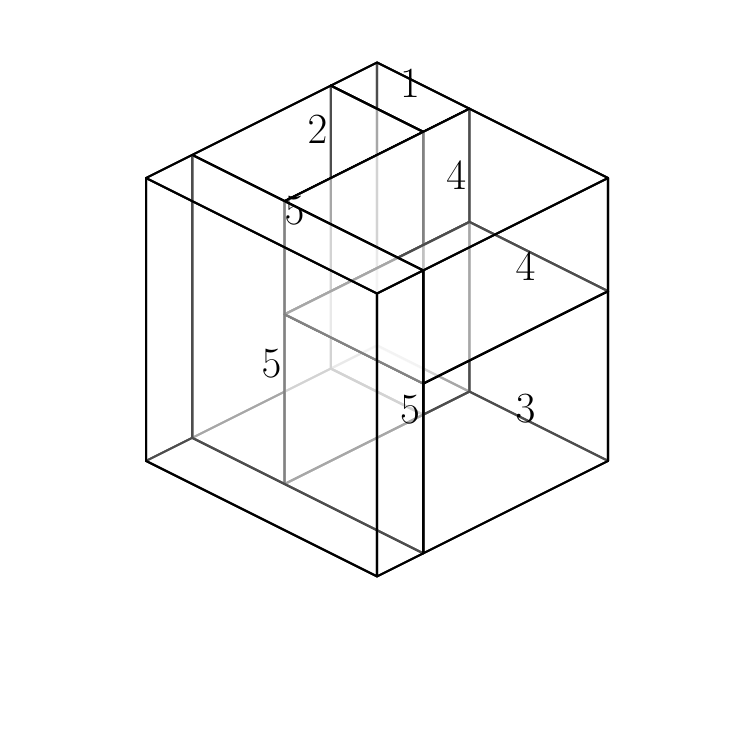}\minipdf{0.30}{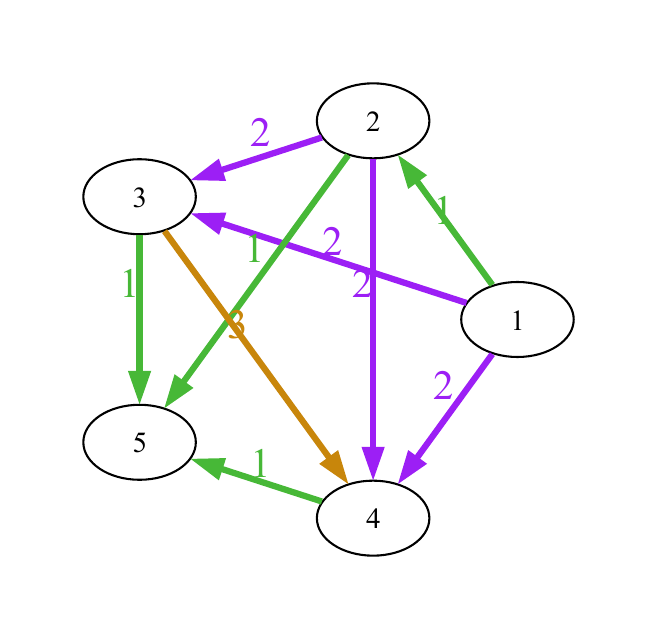}\minipdf{0.30}{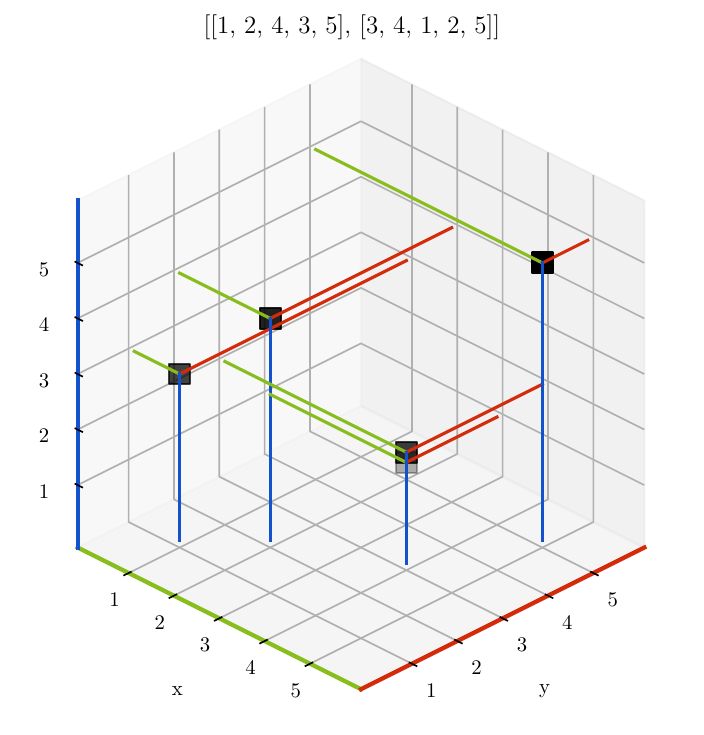}}
    \caption{On the left an example of a $3-$floorplan. In the middle the
        relative order of each blocks with respect to each direction ($x,y,z$). On
        the right the corresponding $3$-permutation (considering the $3$-floorplan
        as a $4$-floorplan).}
    \label{fig:main-example}
\end{figure}

Two $d$-floorplans are equivalent if the relative positions (up, down, left, right etc...) of their boxes are the same.
In the two dimensional case, these objects are known to be closely related to pattern-avoiding permutations such as Baxter permutations or separable permutations~\cite{ackerman2006bijection,asinowski2013orders} and other combinatorial  objects~\cite{aval2021baxter,Young2003}. The study of the generating functions and the enumeration of families of pattern-avoiding mosaic floorplans is also an active field of research, see for example ~\cite{asinowski20242}.

Additionally, in~\cite{asinowski2010separable},  a generalization to arbitrary
dimensions of a restricted family of mosaic floorplans, called {\it guillotine
        partitions}, was considered. The authors find a bijection between
$2^{d-1}-$dimensional guillotine partitions and separable $d$-permutations,
which are a higher dimensional generalisation of separable permutations.

The main contribution of this paper is twofold.
First, we exhibit a generation tree for $d$-floorplans. Given a floorplan, we can remove its \emph{top} box and unequivocally fill the resulting empty space in order to obtain a smaller floorplan (using the tatami condition). This gives a natural definition for a generating tree. In the $d=2$ case, the children of a given floorplan is determined by the number of boxes that touch the top  and left boundaries. Moreover, determining these parameters for the children is straightforward~\cite{bonichon2010baxter}. However, for $d\geq 3$, we need to manage more involved parameters. This allows us to enumerate efficiently the set of all $d$-floorplans for $n \leq 10$.

The second main contribution is a generalization of the bijection between mosaic floorplans and Baxter permutations. A \emph{$d$-permutation} (or \emph{multipermutation}) is a tuple of $d-1$ permutations of size $n$. The presented generalization of the mapping from $d$-permutation to $2^{d-1}$-floorplan is straightforward, but the characterization of the corresponding multipermutation is more involved. These multipermutations are defined by the avoidance of the vincular patterns of Baxter permutations and the dimension $3$ patterns of the separable $d$-permutations. Surprisingly, this set is strictly included in the set of Baxter $d$-permutations defined in~\cite{bonichon2022baxter}. A summary of the objects and the corresponding pattern avoiding permutations is given in Table \ref{tab:seqFlor}. Moreover, this bijection generalizes the bijection between guillotine $2^{d-1}$-floorplans and separable $d$-permutations of ~\cite{asinowski2010separable}.

\begin{table}[!htb]
    \center{\begin{tabular}{|c|c|c|}
            \hline
            Objects & Permutations& Pattern avoidance                                         \\
            \hline
            \hline

 Slicing floorplans & Separable & $\Sym(2413)$  \\ 
 \hline 

 Mosaic floorplans& Baxter  & $ \sbaxpa$ \\ 
 \hline 

 $2^{d-1}$-guillotine floorplans & $d$-Separable & $\Sym(2413), \Sym(\perm{312}{213})$  \\ 
 \hline 

 $2^{d-1}$-Floorplans  & sub $d$-Baxter & $ \sbaxpa$,  $\Sym(\perm{312}{213})$ \\ 
 \hline 

   & $d-$Baxter  & $ \sbaxpa$, $ \sbaxpb$,  \\  &  & $ \sbaxpc$,\\ & & $ \sbaxpd$  \\ 
 \hline  
  \end{tabular}}
    \label{tab:seqFlor}
    \caption{Table of the different class of floorplans and their corresponding permutation classes}
\end{table}


\section{Preliminaries and definitions}
\label{sec:prelim}

In this section, we give various definitions leading to the one of
$d$-floorplans. We then introduce a notion of equivalence and some operations
on these objects. The definitions of this
section
are natural generalisations of the $2-$dimensional case (see~\cite{ackerman2006bijection,hong2000corner,asinowski2024}).

\subsection{Generalisation of rectangles in dimension \texorpdfstring{$d$}{d}}

Let us consider a $d$-dimensional Euclidean space with coordinates
$(x_1,\ldots,x_d)$. A  \emph{hyperrectangle} is defined by the Cartesian
product of $d$ intervals of the form $[x_{i,min}, x_{i,max}]$, where $i$
denotes the $i$th coordinate. An interval $[a,b]$ is \emph{punctual} if $a=b$.
The \emph{dimension} of a hyperrectangle is the number of non-punctual
intervals. The \emph{interior} of a hyperrectangle $R$ is the Cartesian
product of the $d$ intervals of $R$ where each non-punctual interval
$[x_{i,min}, x_{i,max}]$ is replaced by the open interval $]x_{i,min},
    x_{i,max}[$. The \emph{boundary} of a hyperrectangle is the difference between
the hyperrectangle and its interior.

A \emph{box} (resp. \emph{facet}, \emph{edge}) is a hyperrectangle of
dimension $d$ (resp. $d-1$, $d-2$). We say that a facet is of axis $i$ if its
$i$-th coordinate is punctual.

We say that a point $q = (x_1,...,x_d)$ is a \emph{corner of a
    hyperrectangle}
$R= \Pi_{i=1..d}
    [x_{i,min},x_{i,max}]$ if $x_i \in \{x_{i,min},x_{i,max}\}$ for all $i$. We
denote by $q_{\min}(R)$ (resp. $q_{\max}(R)$) the corner of the hyperrectangle with the minimal (resp. maximal) coordinates, we call these corners the \emph{minimal} and the \emph{maximal} corner of $R£$. Note that $R$ is fully determined by $q_{\min}(R)$ and $q_{\max}(R)$. We denote by $\bar{q}$ the opposite corner of $q$ in $R$. By extension a \emph{corner} is a point that is a corner of at least one block.

The boundary of a box possesses $2d$ facets, $2d \times (d-1)$ edges and $2^d$
corners. More precisely, it possesses two facets and four edges of each axis
(one for each combination of maximal and minimal fixed coordinates). The
positions of the facets, the edges and the corners are given by the boundaries
of the intervals that define the boxes. Similarly, the boundary of a facet is
composed of $2\times (d-1)$ edges and $2^{d-1}$ corners. 

Given a facet $f$ of axis $j$ of a box $b$, we say that $f$ is the lower (and resp. upper) facet of axis $j$ if it contains $q_{min}(b)$ (and resp. $q_{max}(b)$).

\begin{definition}\label{def:touch}
    Let $f$ be a facet of axis $i$ and $f'$ be a facet of axis $j$.
    \begin{itemize}
        \item We say that $f$ and  $f'$ \emph{crosses} each other if the
              intersection of their interiors is non empty.
        \item We also say that \emph{$f$ properly touches $f'$} if the interior of $f$
              intersect $f'$ and they do not cross each other.
        \item We say that $f$ touch $f'$ if $f=f'$ or $f$ properly touches $f'$.
    \end{itemize}
\end{definition}

\begin{example}
    An example in dimension $3$ of two facets crossing and touching is given in Figure~\ref{fig:face-cross-not-cross}.

    \begin{figure}[!htb]
        \center{\minipdf{0.33}{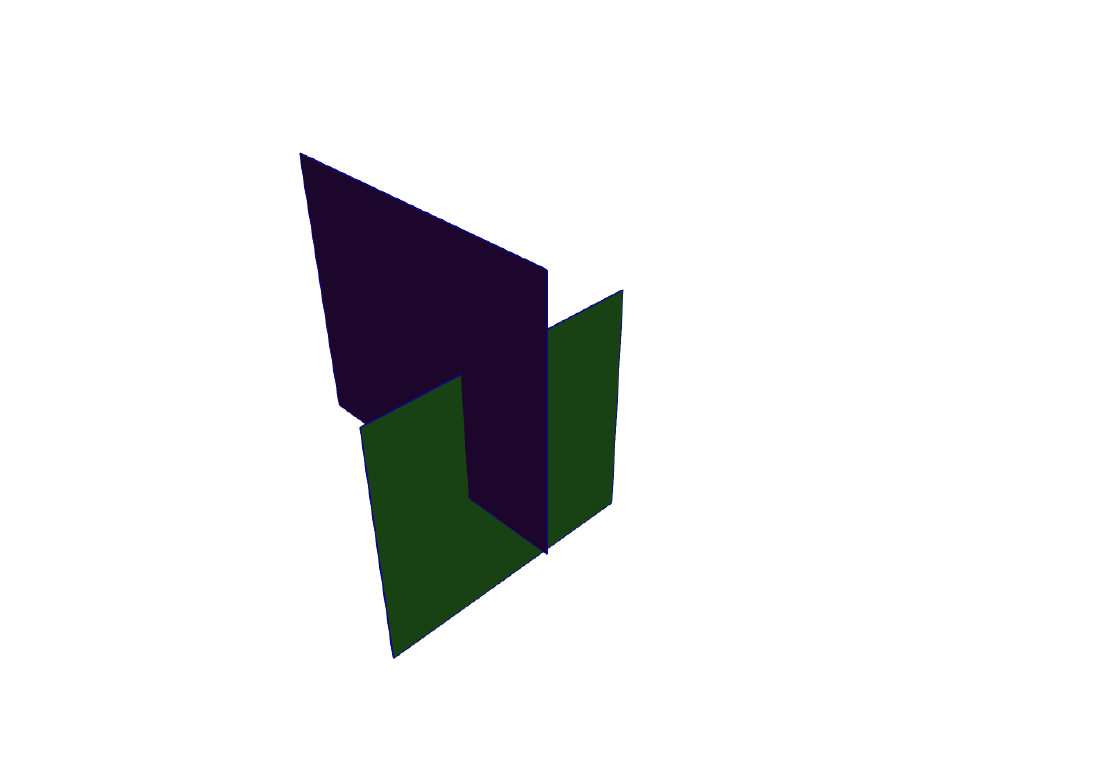}\minipdf{0.33}{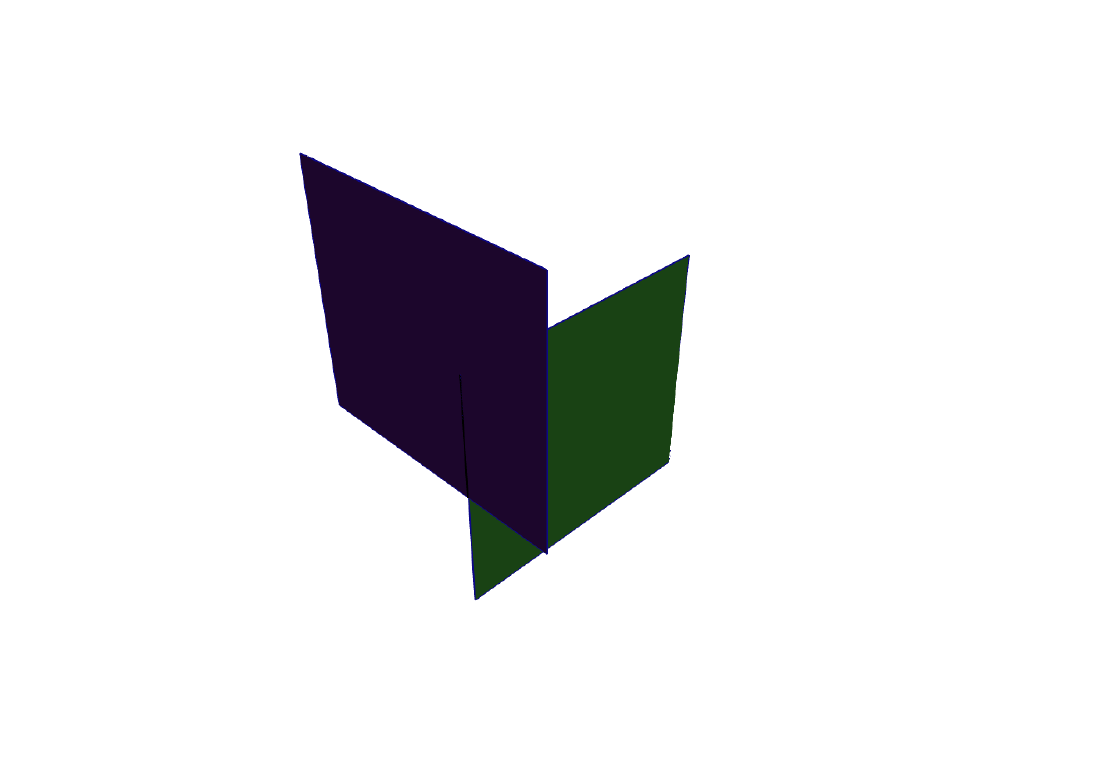} \minipdf{0.33}{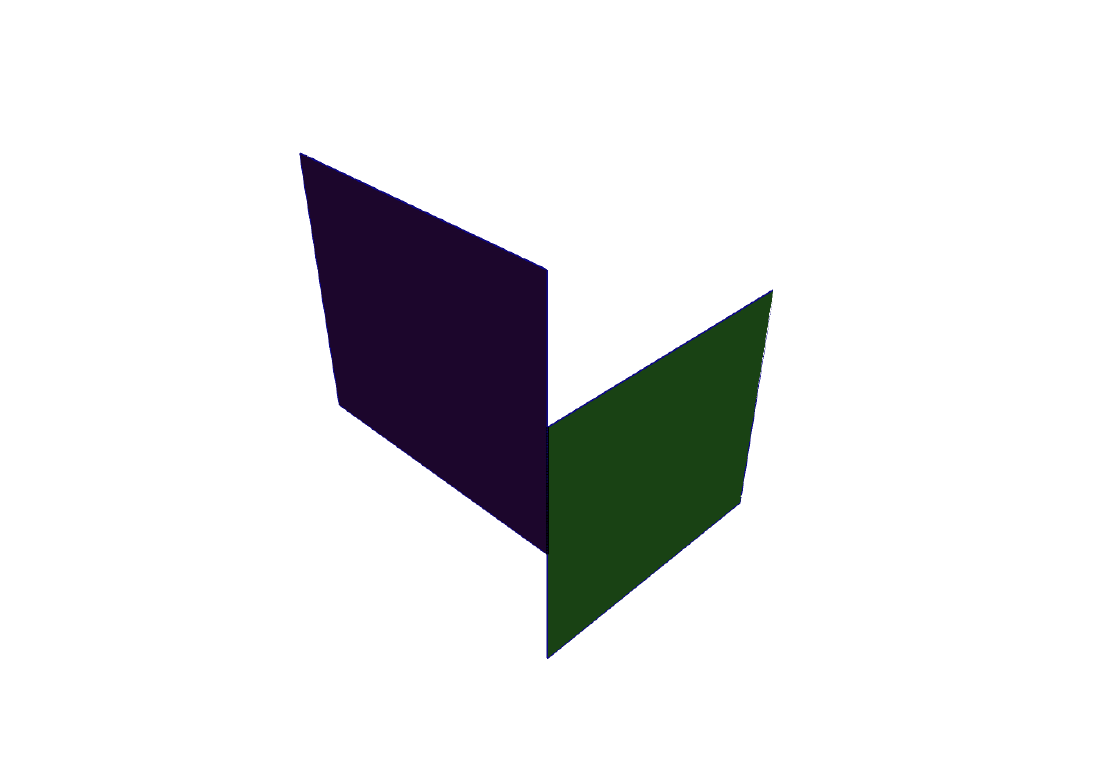}}
        \caption{On the left two facets crossing. In the middle, facet $B$ touches facet $A$. On the right, two facets touching each other.}
        \label{fig:face-cross-not-cross}
    \end{figure}
\end{example}

\medskip

Let us now give some explicit examples.
When considering the $2-$, $3-$ and resp. $4$-dimensional case, we denote the coordinates by $(x,y)$, $(x,y,z)$ and resp. $(x,y,z,t)$.
In these cases, a $3$-box corresponds to a rectangular parallelepiped, a facet of dimension $3$ of axis x to a rectangle drawn in the $yz$-plane and a $3$-edge of axis $yz$ to a segment drawn in the $x$ direction.
A $2$-box corresponds to a rectangle, a facet of dimensions two and axis $x$ to a segment drawn in the $y$ direction and a $2$-edge  to a vertex.

\subsection{Mosaic floorplans in dimension \texorpdfstring{$d$}{d}}

A $d$-dimensional floorplan $\mP$ is a box partitioned into boxes. The
interior boxes are called {\it blocks}. The facets $\mP$ is the union
of the facets of the blocks of $\mP$. A \emph{border} of axis $j$ is the
union of
all facets of axis $j$ that share the same position and that do not lie on the
bounding box of $\mP$. Given an inner facet $f$ of $\mP$, we denote by $b(f)$
the
border of $f$.

\begin{definition}
    A $d$-dimensional floorplan is \emph{generic} if all its borders are facets.
    \label{def:generic}
\end{definition}
In dimension 2, a floorplan is generic if every border is connected. In
dimension $d>2$ a border may be connected but not a facet as illustrated in
Figure~\ref{fig:non-generic}.

\begin{figure}[!htb]
    \center{\minipdf{0.45}{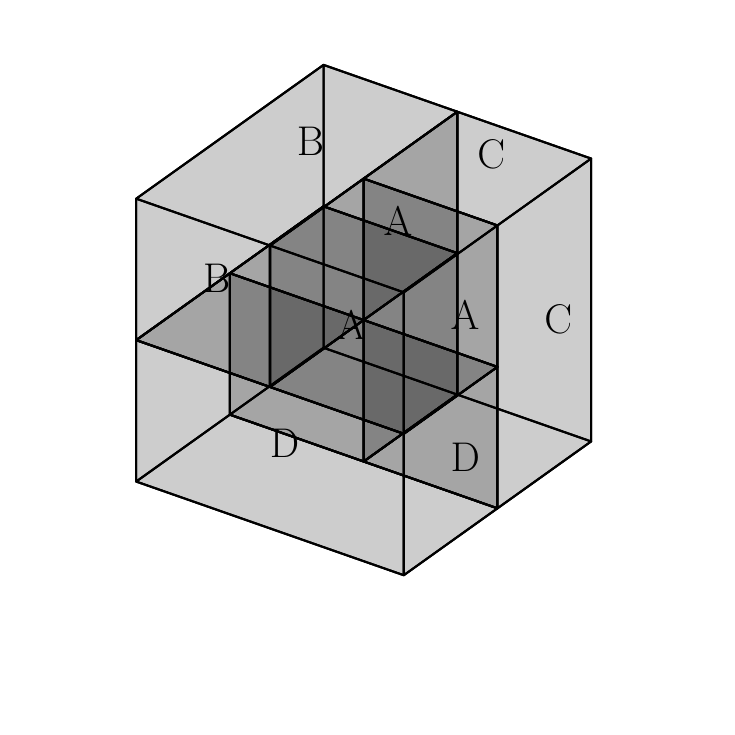}
        \minipdf{0.45}{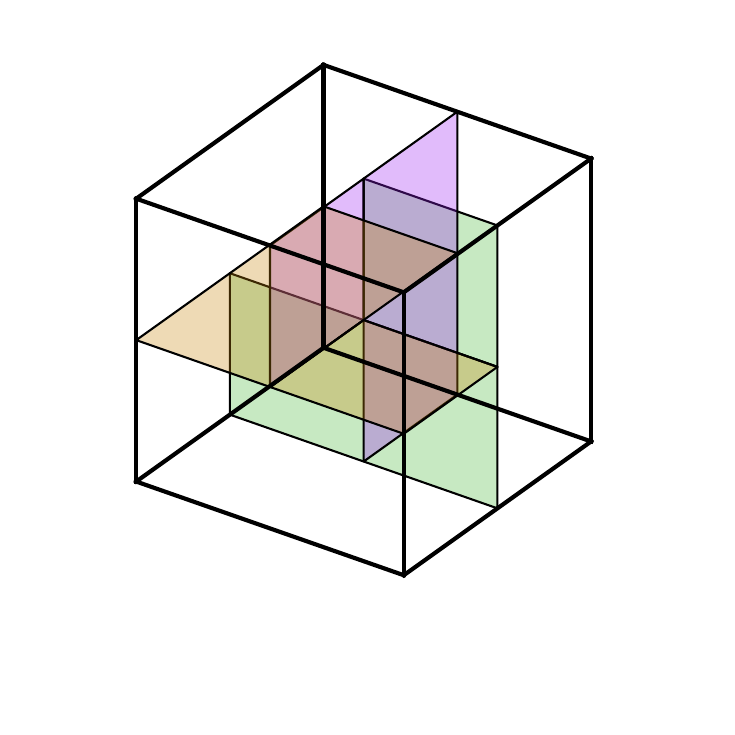}}
    \caption{A non-generic 3-floorplan. On the left the boxes. On the right the
        borders with "L" shapes.}
    \label{fig:non-generic}
\end{figure}

\begin{figure}[!htb]
    \center{\minipdf{1}{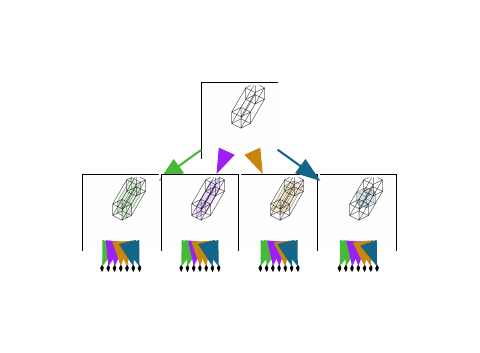}}
    \caption{The first 4-floorplans.}
    \label{fig:gen-tree-fp-4d}
\end{figure}

\begin{definition}
    A \emph{mosaic $d$-dimensional floorplan} (or a \emph{$d$-floorplan}) is a
    generic $d$-dimensional floorplan without border crossings. This last
    condition is
    referred to as the \emph{tatami} condition.
    \label{def:dfloorplan}
\end{definition}

In a $d$-floorplan, borders of different types are only allowed to touch each
other. This lead to \emph{block junctions} made of borders and with a $\TT$
shape (as in the right of Figure~\ref{fig:face-cross-not-cross}) that
generalise the one made of segments in the two dimensional case.

\begin{example}
    In dimension $2$, Definitions~\ref{def:generic} and~\ref{def:dfloorplan}  become the standard definition of a mosaic floorplan:
    \begin{itemize}
        \item A $2-$dimensional floorplan is a partition of a rectangle into interior-disjoint rectangles, called \emph{blocks}.
        \item In a $2-$dimensional floorplan, borders are segments. A $2-$dimensional floorplan is thus \emph{generic} if there are no segments that share the same fixed coordinate.
        \item A generic $2-$dimensional floorplan fulfills the tatami condition it no two segments cross each other. This condition imposes equivalently that no point can belong to the boundary of four rectangles. Examples of $2-$floorplans are given in Figure~\ref{fig:2dfloor}.

              \begin{figure}[!htb]
                  \center{\minipdf{0.25}{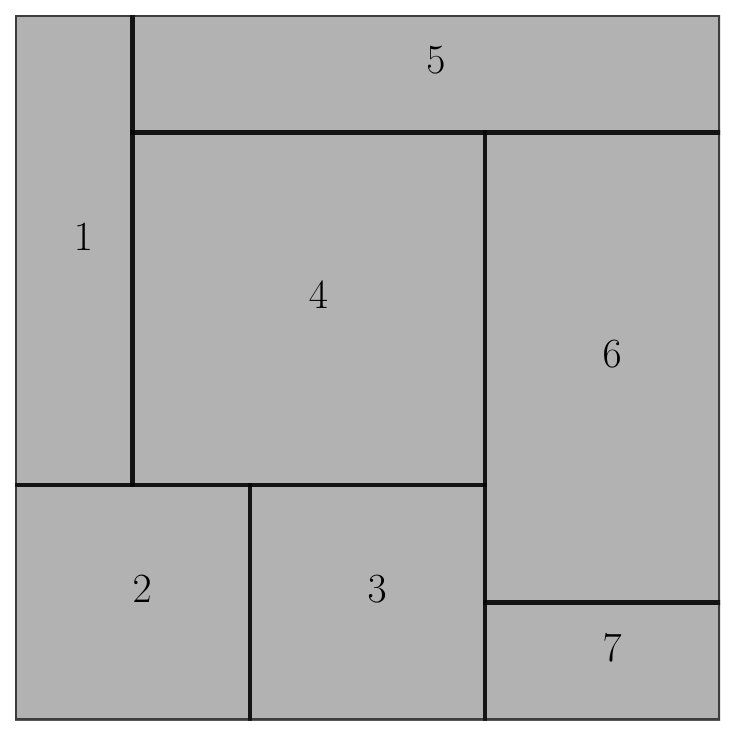}\minipdf{0.25}{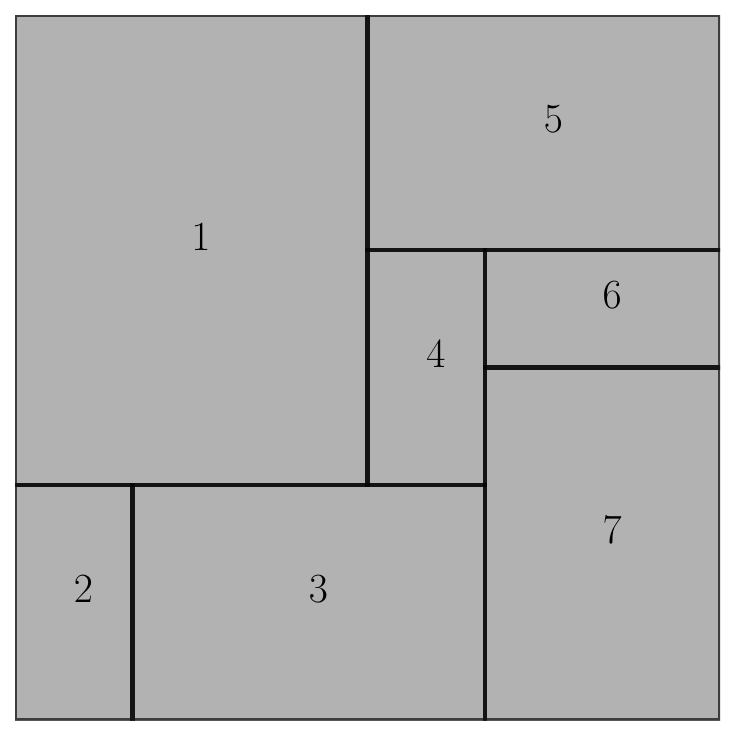}\minipdf{0.25}{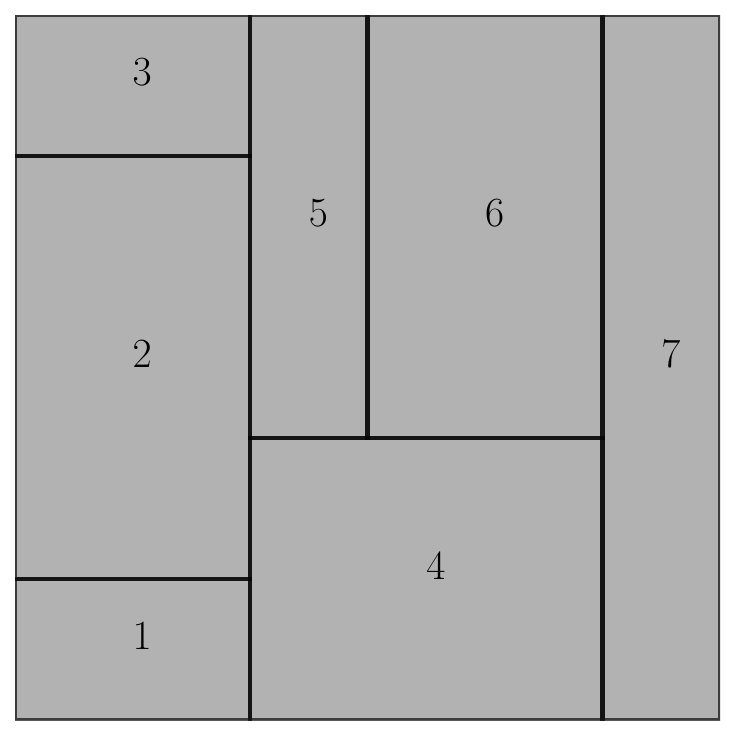}}
                  \caption{Examples of $2-$floorplans}
                  \label{fig:2dfloor}
              \end{figure}
    \end{itemize}

    In dimension $3$, the rectangles of the $2-$dimensional case are replaced by rectangular parallelepipeds. The blocks are now separated by rectangles and the tatami condition imposes that two such rectangles of different axis cannot cross each other. Some examples of $3-$floorplans are given in Figure~\ref{fig:ex_fp}.

    \begin{figure}[!htb]
        \center{\minipdf{0.325}{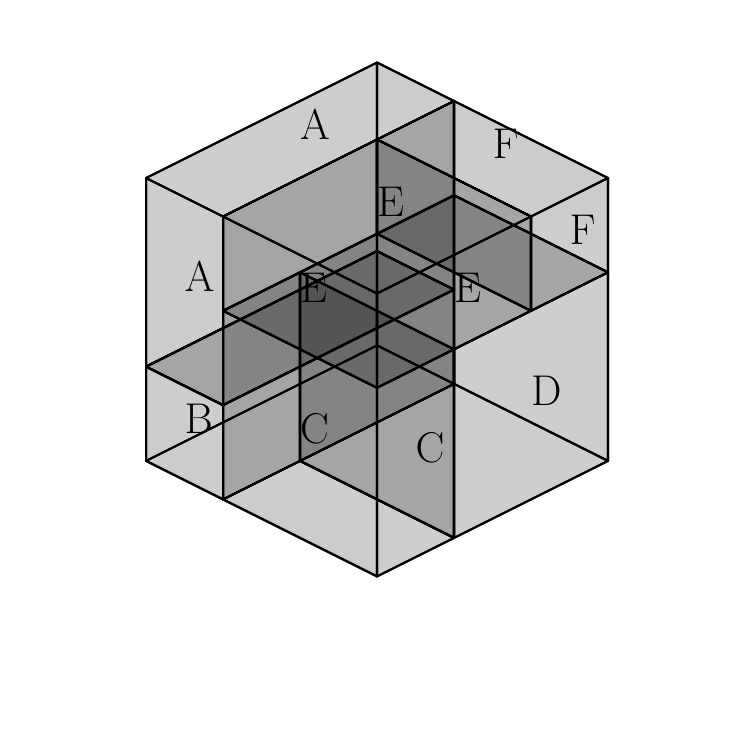}\minipdf{0.325}{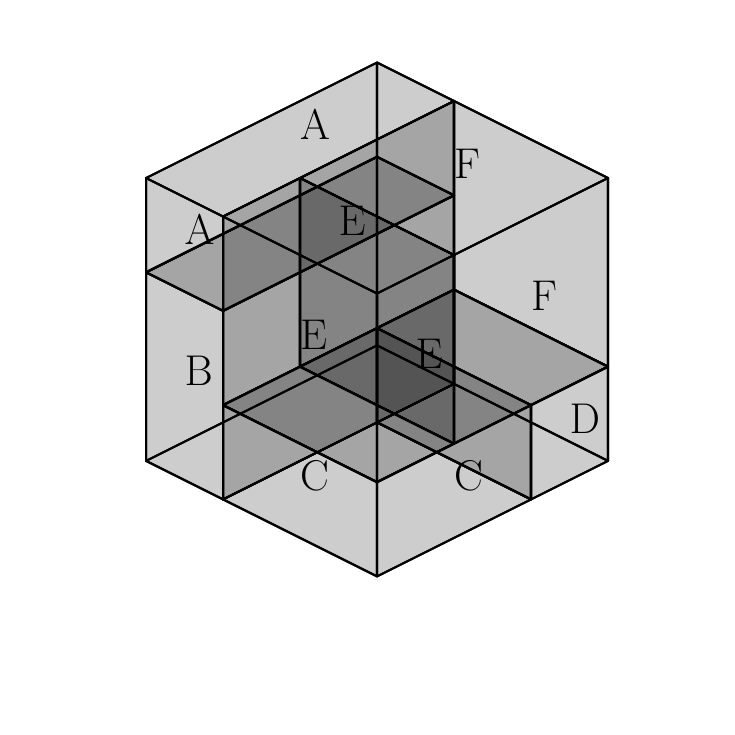}\minipdf{0.325}{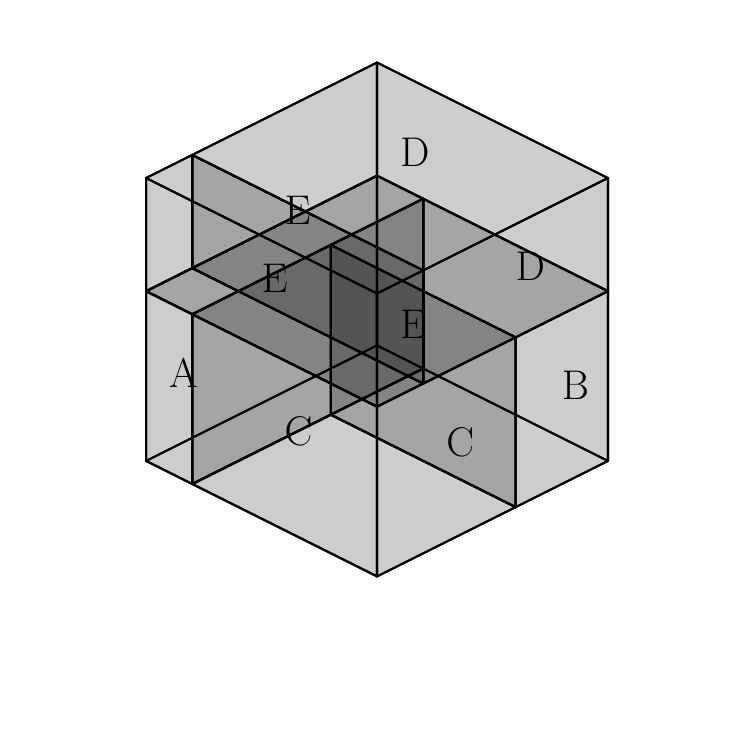}}
        \caption{Examples of $3-$floorplans}
        \label{fig:ex_fp}
    \end{figure}
\end{example}

\subsection{Equivalence relation for \texorpdfstring{$d$}{d}-floorplans}

As in~\cite{ackerman2006bijection,hong2000corner,asinowski2024}, we consider in
this paper  $d$-floorplans up to an \emph{equivalence relation}. This relation
is a straightforward generalisation of the so-called \emph{weak equivalence}.
Let us first define the direction relations between the blocks of a
$d$-floorplan; there are $d$ possible such relations, one per axis.

\begin{definition}
    Let $\mP$ be a $d$-floorplan and let $B_1$ and $B_2$ be two blocks in $\mP$.
    The neighborhood relation  \emph{$B_1$ is a left neighbor of $B_2$ in the $j$ direction} describes the following situation:
    \begin{itemize}
        \item $x_{j,max}(B_1)= x_{j,min}(B_2)= x_j$ and there is a border of
              axis $j$ in $\mP$ that contains the facet of axis $j$ and position $x_j$
              of the two blocks.
    \end{itemize}
    We define the direction relation \emph{$B_1$ precedes $B_2$ in the $j$ direction}
    (denoted $B_1 \overset{j}{\twoheadleftarrow} B_2$ with $j=1, \ldots, d$) as the transitive and reflexive closure of this neighborhood relation.
    
    We also define similarly the opposite relation, "\emph{$B_2$ follows $B_1$}" in the $j$ direction, denoted by $B_2 \overset{j}{\twoheadrightarrow} B_1$.
    \label{def:dirrel}
\end{definition}

\begin{example}
    In dimension $2$, there are two types of neighborhood relations, these are shown in Figure~\ref{fig:rect-adj}.
    \begin{figure}[!htb]
        \center{\begin{minipage}{0.15\textwidth}\center{\resizebox{0.9\textwidth}{!}{\input{figs/face-adjx.pgf}}}\end{minipage}\begin{minipage}{0.15\textwidth}\center{\resizebox{0.9\textwidth}{!}{\input{figs/face-adjy.pgf}}}\end{minipage}}
        \caption{Two neighbor rectangles with respect to the $x$ and $y$ axis.}
        \label{fig:rect-adj}
    \end{figure}
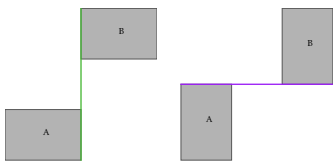
    Figure~\ref{fig:adj-3D} shows the $3$ types of neighborhood relations of the blocks of $3-$floorplans.
    \begin{figure}[!htb]
        \center{\minipdf{0.32}{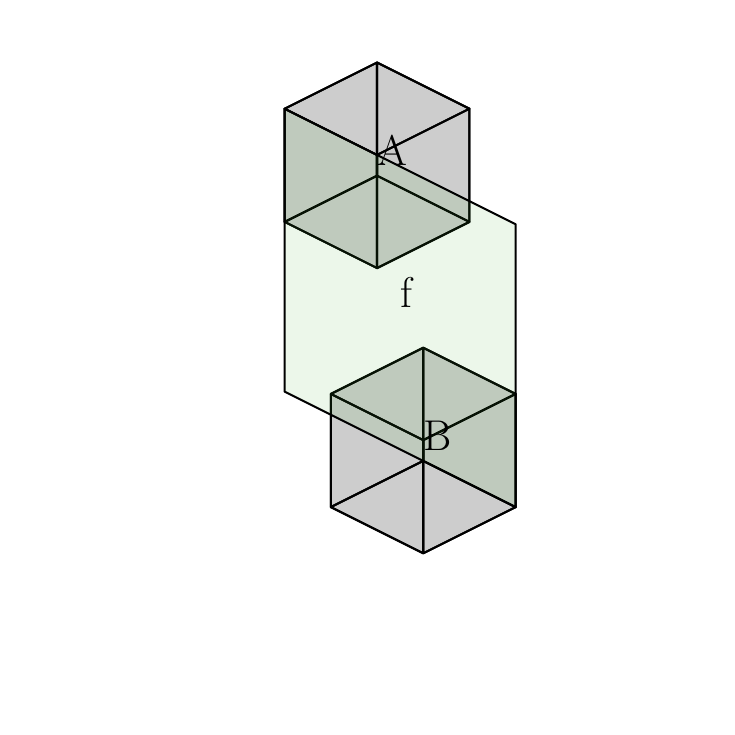}\minipdf{0.32}{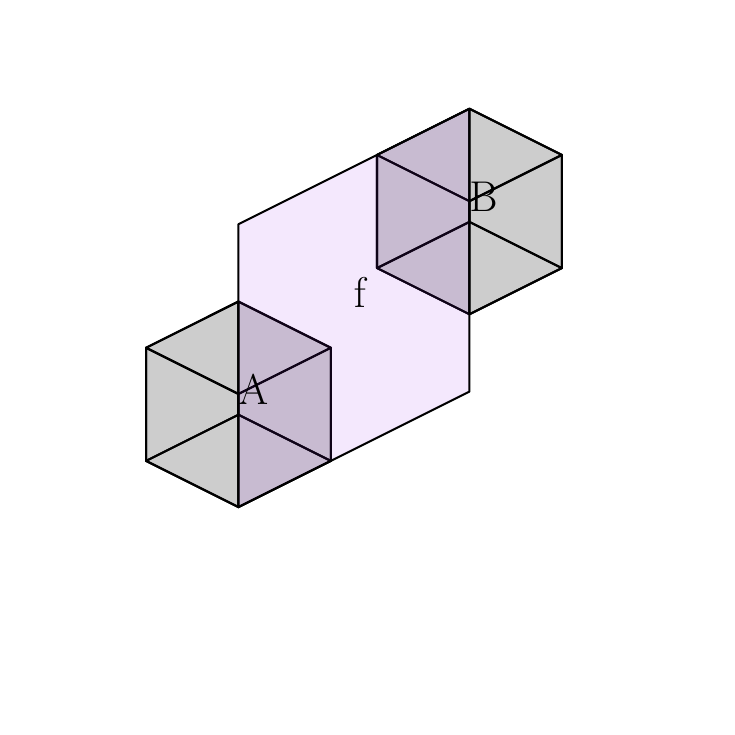} \minipdf{0.32}{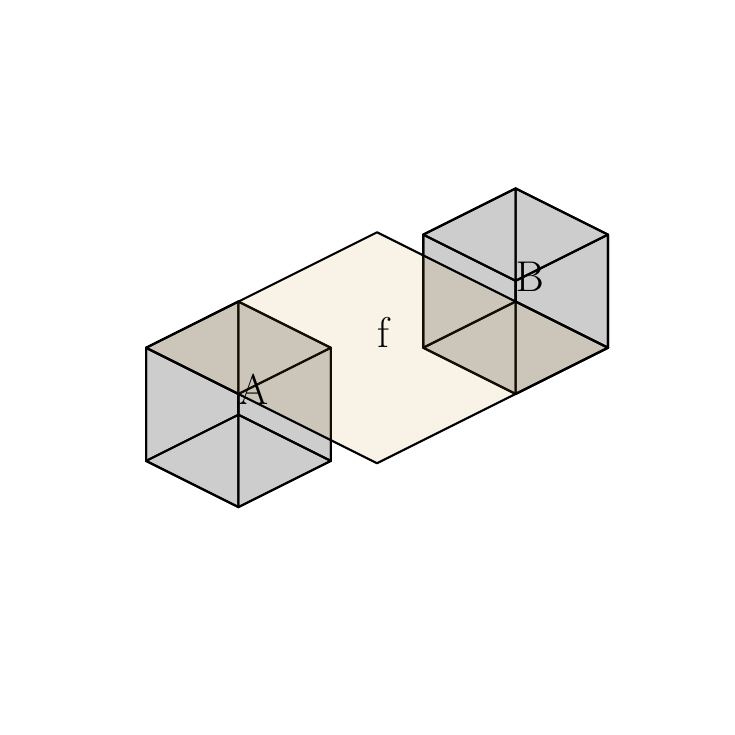}}
        \caption{The relations $A$ and $B$ are $x-$neighbors, resp. $y-$neighbors, and resp. $z-$neighbors for the blocks of a $3-$floorplan.}
        \label{fig:adj-3D}
    \end{figure}
\end{example}

\begin{remark}
    The relations $\overset{j}{\twoheadleftarrow}$ and $\overset{j}{\twoheadrightarrow}$  are partial orders for all $j$.
\end{remark}

\begin{definition}
    Two $d$-floorplans are considered \emph{weakly equivalent} if there exists
    a labeling of their blocks such that the direction relations are identical
    in both objects.
    \label{def:weakd}
\end{definition}

{\bf From now on, when we refer to a $d-$floorplans, we refer to its whole equivalence class}.
This equivalence relation can be seen from two aspects. On the one hand, one
can look at the adjacency graph of the blocks of a $d$-floorplan as in Figure
\ref{fig:adj-graph-3D}. Two equivalent $d$-floorplans are then identified if
their adjacency graphs are isomorph. On the other hand, one can consider local
modifications of the blocks. Two equivalent $d$-floorplans can then be
identified if one can be obtained from the other by a sequence of block
junctions's shifts (that are made of borders). Loosely speaking, considering
the equivalence given in Definition~\ref{def:weakd} is equivalent to say that
in a $d$-floorplan, we are only interested in the relative positions of the
blocks with respect to each other and to segments.

Notice also that, knowing all direction relations
$\overset{j}{\twoheadleftarrow}$ of the blocks of an equivalence class of
$d$-floorplans is sufficient to be able to reconstruct one of its
representatives. Hence, equivalence classes of $d$-floorplans can be
represented by sets of $d$ partial orders on the same ground set. However, not
every set of $d$ partial orders represents an equivalence class of
$d$-floorplans.

\begin{example}
    In Figure~\ref{fig:2dfloor}, the first two $2-$floorplans are equivalent and the third one is not. Additionally, Figure~\ref{fig:2d-orders} shows the adjacency graphs of the block direction relations for these two equivalent $2-$floorplans.

    Similarly, the first two $3-$floorplans in Figure~\ref{fig:ex_fp} are equivalent. Figure~\ref{fig:adj-graph-3D} shows the adjacency graphs of the block direction relations of these equivalent $3-$floorplans.

    \begin{figure}[!htb]
        \center{\minipdf{0.35}{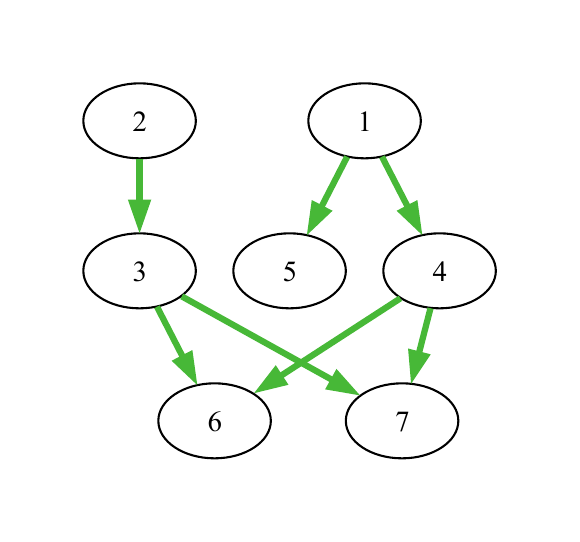} \minipdf{0.35}{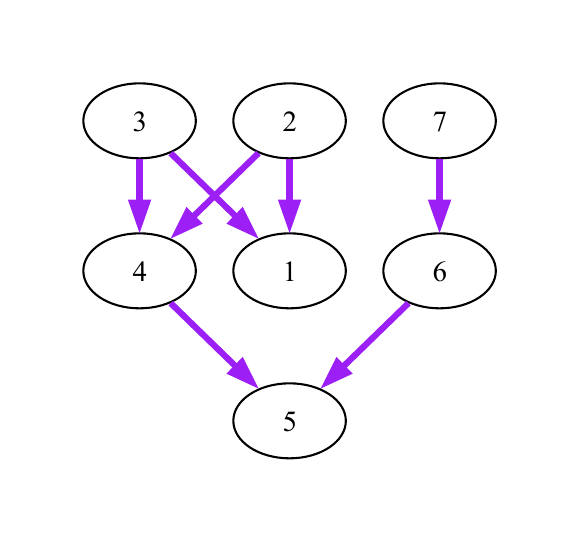}}
        \caption{Horizontal and vertical order relations between blocks of the $2$-floorplans on the left of Figure~\ref{fig:2dfloor}.}
        \label{fig:2d-orders}
    \end{figure}

    \begin{figure}[!htb]
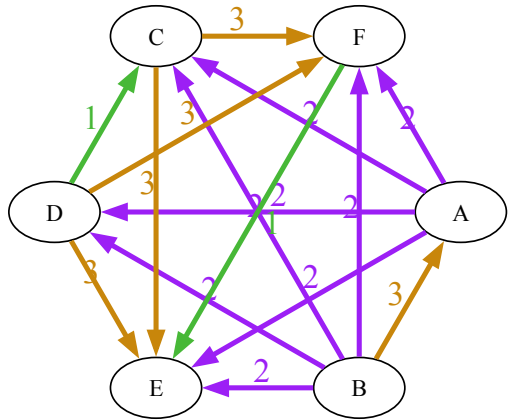

        \center{\minipdf{0.6}{figs/FP-ex-3d1-order-all.pdf}}
        \caption{Adjacency graphs of the $3-$floorplans on the left of Figure~\ref{fig:ex_fp}.}
        \label{fig:adj-graph-3D}
    \end{figure}

\end{example}

\subsection{Block deletion \texorpdfstring{$d$}{d}-floorplans}
\label{sec:peeld}

In the context of $2$-floorplan~\cite{ackerman2006bijection,hong2000corner} a block deletion operation was introduced. It consists of removing the block incident to a specific corner of the bounding box and then filling the resulting empty space by shifting one facet of the deleted block. In this section, we generalize this operation in higher dimensions.

Given a corner $q$ of a $d$-floorplan $\mP$. Let $B$ be the block incident to $q$ and let $\bar{q}$ be its opposite corner in $B$. A \emph{shifting facet} is a facet of $B$ incident to $\bar{q}$ such that $b(f)=f$.

A \emph{block deletion} operation with respect to $q$ consists of removing the block $B$ by shifting the facet $f$ until it reaches $q$.

\begin{figure}[!htb]
    \center{	\minipdf{0.32}{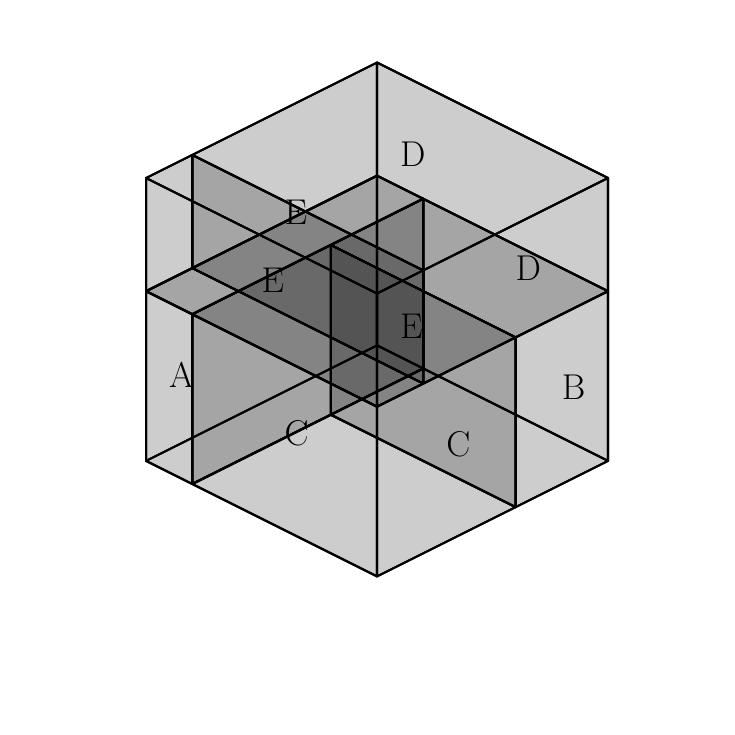}
        \minipdf{0.32}{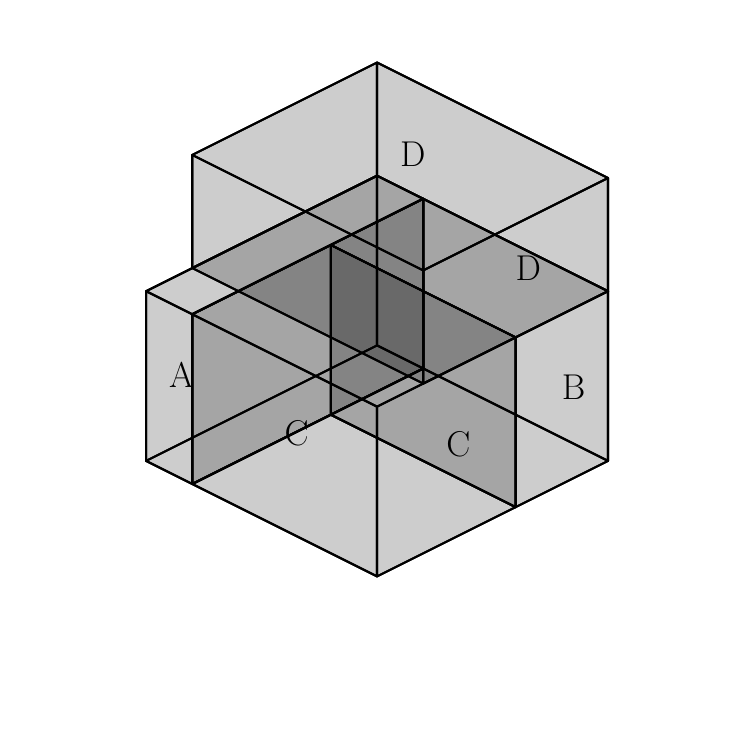}
        \minipdf{0.32}{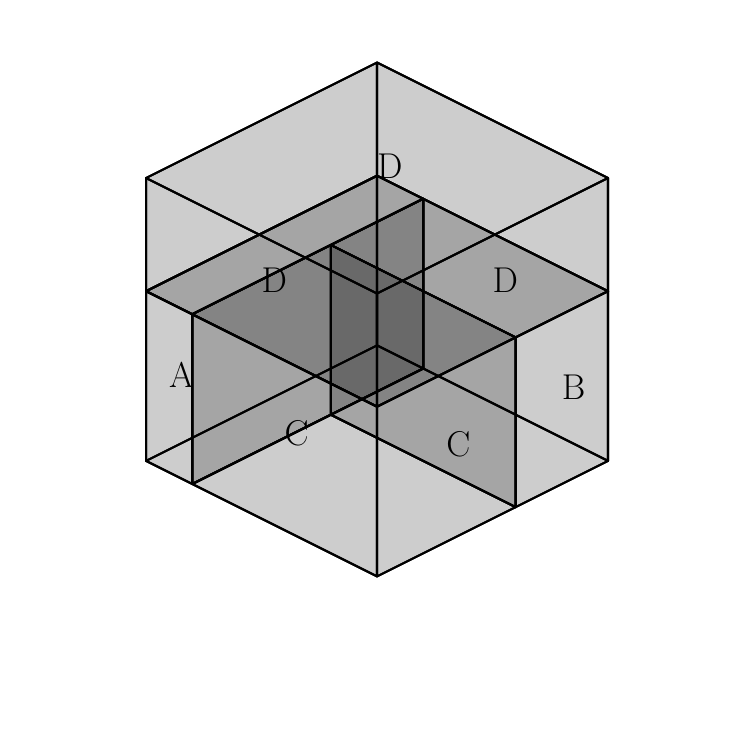}
        \caption{An example of a block deletion operation of a $3-$floorplan.}}
    \label{fig:block-delete}
\end{figure}

In the rest of this sub-section, we show that for each corner $q$ there is a unique block $B$ incident to $q$ and, more importantly, $B$ has a unique shifting facet (Lemma~\ref{lem:unique-shift}).

Let $\mathcal{F}_{B,q}$ be the set of facets of $B$ that are incident to a corner $q$. We denote by $b(\mathcal{F}_{B,q})$ the set of borders of the facets of $\mathcal{F}_{B,q}$.

\begin{lem}\label{lem:total-order}
    Let $B$ be a block of a $d$-floorplan $\mP$ and $q$ corner of $B$ that is not on the boundary of $\mP$.
    The relation touch (see Definition~\ref{def:touch}) is a total order on $b(\mathcal{F})_{B,q}$.
\end{lem}
\begin{proof}
    Let $f$ and $f'$ be two facets of $\mathcal{F}_{B,q}$.
    \begin{itemize}
        \item {\bf Reflexivity.} By definition $b(f)$ touches itself.
        \item {\bf Antisymmetry.} If $b(f)$ and $b(f')$ touch each other, then either $b(f)$ and $b(f')$ cross each other or $b(f)=b(f')$. Since $\mP$ is a $d$-floorplan, there are no crossings, hence the latter case occurs.

        \item {\bf Totality.} Assume that neither $b(f)$ touches $b(f')$ nor $b(f')$ touches $b(f)$, and let us find a contradiction (see Figure~\ref{fig:face-cross-not-cross} on the right). Let us consider the edge $e = f \cap f'$. Since $b(f)$ doesn't touch $b(f')$, $e$ doesn't intersect the interior of $b(f')$. By symmetry, it doesn't intersect the interior of $b(f)$. Hence $e$ is part of the boundary of $b(f)$ and $b(f')$. Let $\epsilon$ be a positive number arbitrarily small. Let $p_b$ be a point on the segment $q,\bar{q}$ at distance $\epsilon$ from $q$. Let $p_f$ and $p'_f$ be two points such that  $p_b$ is on the segment $[p_f,p'_f]$ and the segments $]p_b,p_f[$ and $]p_b,p_f'[$ intersect respectively facets $f$ and $f'$. Let $p'_b$ be the symmetric of $p_b$ with respect to the edge $e$. Since the the border $b(f')$ doesn't go beyond $e$, $p_f$ and $p_b'$ belongs to the same block $B'$. Similarly, $p'_f$ and $p_b'$ belongs also to $B'$. Since $p_b$ is between $p_f$ and $p'_f$ and not in $B'$, this means that $B'$ is not convex, which is a contradiction.

        \item {\bf Transitivity.} If $d=2$ $\mathcal{F}_{B,q}$ contains only 2 facets, so the property trivially hold. Let us consider the case $d>2$. Let $f''$ be another facet of $\mathcal{F}_{B,q}$ such that $b(f)$ touches $b(f')$ and $b(f')$ touches $b(f'')$. Since the touch relation is total, either $b(f)$ touches $b(f'')$ or $b(f'')$ touches $b(f)$. Let us assume the latter case, and let us find a contradiction. To do so we proceed as for the totality property. Let $e = f \cap f'$. Let $\epsilon$ be a positive number arbitrarily small. Let $p_b$ be a point on the segment $q,\bar{q}$ at distance $\epsilon$ from $q$. Let $p"_f$ and $p_e$ be two points such that  $p_b$ is on the segment $[p"_f,p_e]$ and the segments $]p_b,p"_f[$ and $]p_b,p_e[$ intersect respectively facet $f"$ and and the edge $e$. Let $p_1$ be the symmetric of $p"_f$ with respect to the hyperplane containing $f'$. Let $p_2$ be the symmetric of $p_1$ with respect to the hyperplane containing $f$. Since the border of $b(f')$ doesn't cross $f"$, $p"_f$ and $p_1$ belong to the same block $B'$. Similarly, $p_1$ and $p_2$ belong also to $B'$ since $b(f)$ doesn't cross $f'$ and $p_2$ and $p_e$ belong to the same block since $b(f")$ doesn't cross $f$. We deduce that $p"_f$ and $p_e$ belong to $B'$ and that $B'$ is not convex which is a contradiction.

              \begin{figure}[!htb]
                  \center{\minipdf{0.30}{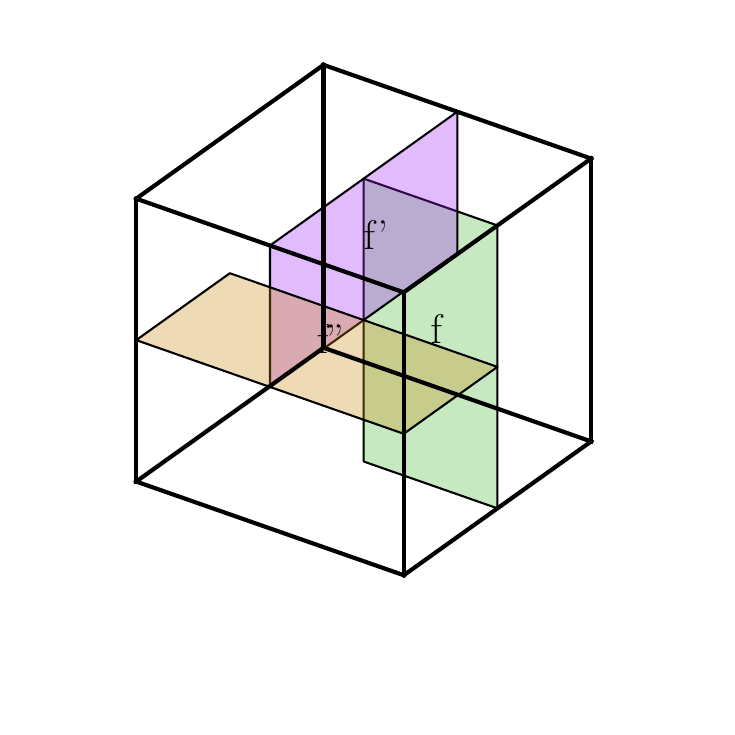}}
                  \caption{Non-transitive configuration.}
                  \label{fig:non-transitive}
              \end{figure}

    \end{itemize}
\end{proof}

\begin{lem}\label{lem:unique-shift}
    Let $B$ be a block of a $d$-floorplan $\mP$ and $q$ a corner of $\mP$. $B$ has a unique shifting facet.
\end{lem}
\begin{proof}
    By Lemma~\ref{lem:total-order}, there is a unique facet $f$ of $\mathcal{F}_{B,\bar{q}}$ such that $b(f)$ touches every border of $b(\mathcal{F}_{B,\bar{q}})$.
    Moreover, since $f$ and $b(f')$ have a common point $\overline{q}$, the facet $f$ touches $b(f')$ for every $f' \in \mathcal{F}_{B,\bar{q}}$. Hence $f$ touches each border of  $b(\mathcal{F}_{B,\overline{q}})$.
    This implies that $b(f)=f$.

    Moreover, for each other facet $f'$, $f$ touches $b(f')$, which means that the intersection of $f$ with the interior of $b(f')$ is non-empty. As $f$ and $f'$ are the facets of the same block, the intersection of $f$ with the interior of $f'$ is empty. Hence $b(f')\neq f'$. In other words, $f$ is the unique shifting facet of $B$.
\end{proof}

In the study of $d$-floorplans, the block deletion plays a crucial role as it
gives a natural definition of the parent of a $d$-floorplan and thus of the generating tree of $d$-floorplans . The encoding
of this generating tree is explained in Section~\ref{sec:gen_tree}.

\section{A Generating tree of $d$-floorplans}
\label{sec:gen_tree}
Let us define a generating tree based on the block deletion operation. Given a $d$-floorplan $\mP$ with more than one block, we define the parent of $\mP$ denoted $p(\mP)$ as the floorplan obtained after a block deletion with respect to the maximal corner $q_{\max}$ of $\mP$. This defines a generating tree on $d$-floorplans whose root is the $d$-floorplan with only one block. The \emph{children} of a $d$-floorplan $\mP$ are the floorplans $\mP'$ such that $p(\mP')=\mP$.

\begin{example}
    In Figure~\ref{fig:gen-tree-fp-2d} and~\ref{fig:gen-tree-fp-3d} are shown
    the first levels of the generating trees of $2$-floorplans and
    $3$-floorplans induced by the block deletion.
    \begin{figure}
        \centering
        \includegraphics[scale=1.3]{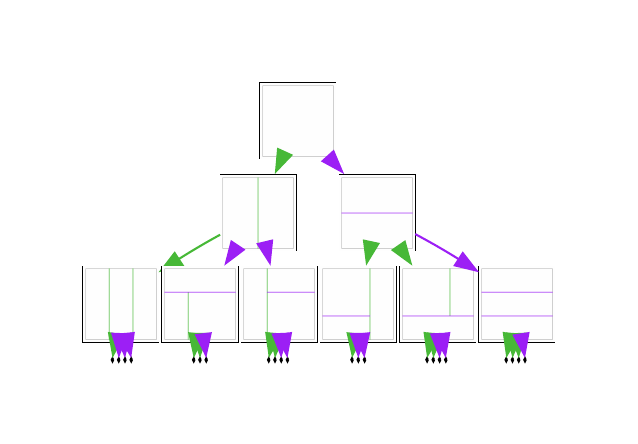}
        \caption{The first levels of the $2$-floorplan's generating tree}
        \label{fig:gen-tree-fp-2d}
    \end{figure}
    \begin{figure}
        \centering
        \includegraphics[scale=0.67]{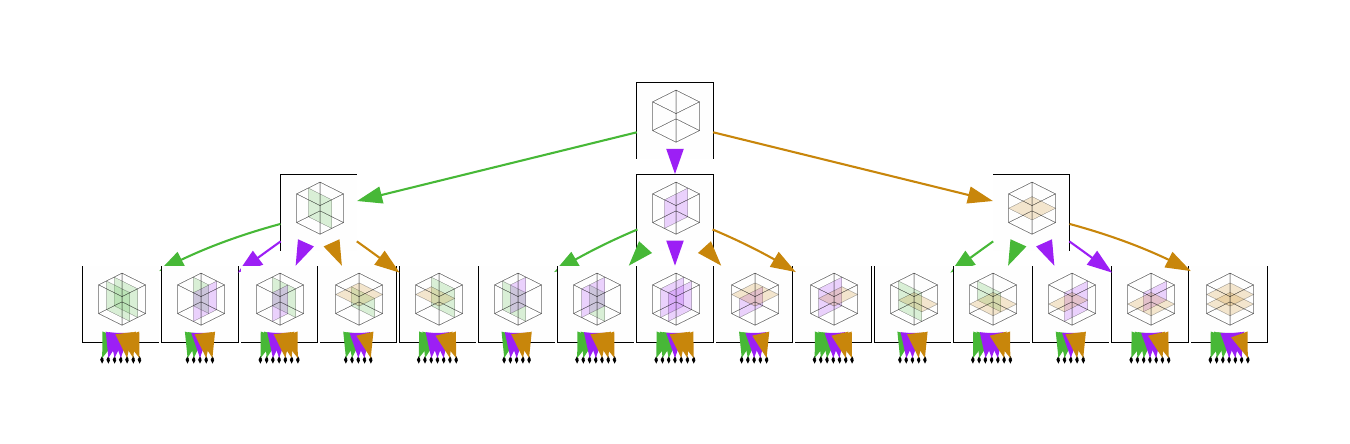}
        \caption{The first levels of the $3$-floorplan's generating tree}
        \label{fig:gen-tree-fp-3d}
    \end{figure}
\end{example}

In the rest of the section we will first show how to characterize the children of a $d$-floorplan. For this purpose, we will introduce the notion of \emph{pushable facet} and \emph{pushable corner}. Then we will show how to compute the pushable corners of the children of a $d$-floorplan, knowing only the pushable corners of the parent. Based on this, we will define a generating tree isomorphic to the one on $d$-floorplans, which allows us to enumerate efficiently the first $d$-floorplans.

\subsection{Pushable corners and block insertion}
In order to extract the structure of the generating tree, one first needs to introduce the inverse of a block deletion, an operation which we call a \emph{block insertion}. We define here this block insertion with respect to the maximal corner $q_{max}$, the definition for the other corners follow. We call \emph{maximal block}, the block that contains $q_{max}$.

Before a block deletion with respect to $q_{max}$, the shifting facet $f$ of the maximal block is its lower facet of a given axis $i$. Moreover, it was also the union of upper facets of axis $i$ of blocks of $\mP$. After the block deletion, $f$ becomes a facet that is included in the upper facet of axis $i$ of the bounding box of $\mP$, such that it contains contains $q_{max}$ and it is the union of maximal facets of axis $i$ of blocks of $\mP$.

A \emph{pushable facet} $f$ of axis $i$  (with respect to $q_{max}$) is a facet of axis $i$ included in the bounding box of $\mP$, that contains $q_{max}$ and that is the union of upper facets of axis $i$ of blocks of $\mP$. A \emph{pushable corner} of axis $i$ is a corner $q$ that is the minimal corner of a pushable facet of axis $i$. Remark that a pushable corner must be the minimal corner of an upper facet of a block, but not all minimal corners of upper facets of blocks are pushable corners. 

In Figures~\ref{fig:treelabel2d} and~\ref{fig:ins}, the pushable corners are materialized with arrows. Given a pushable corner $q$ of axis $i$ and its associated facet $f_q$, we define the \emph{block insertion} associated with $q$ as the $d$-floorplan obtained by flattening the blocks below the $f_q$ in the direction $i$ and inserting a block in the newly created space. This operation defines a mapping between the pushable corners of a $d$-floorplan and its children.  Let $q_{new}$ be the opposite corner of $q_{\max}$ of the newly created block. $q_{new}=(x_1, \dots, x_i-\alpha,x_{i+1},\dots,x_d)$ for a value of $\alpha$ small enough so that $x_i-\alpha$ is greater than the $i$-th coordinate of borders of axis $i$ of $\mP$. We denote $q_{new}^j$ the projection of $q_{new}$ on the maximal facet of axis $j$ of the bounding box of $\mP$.

\begin{figure}[!tbh]
    \centering
    \center{\minipdf{0.31}{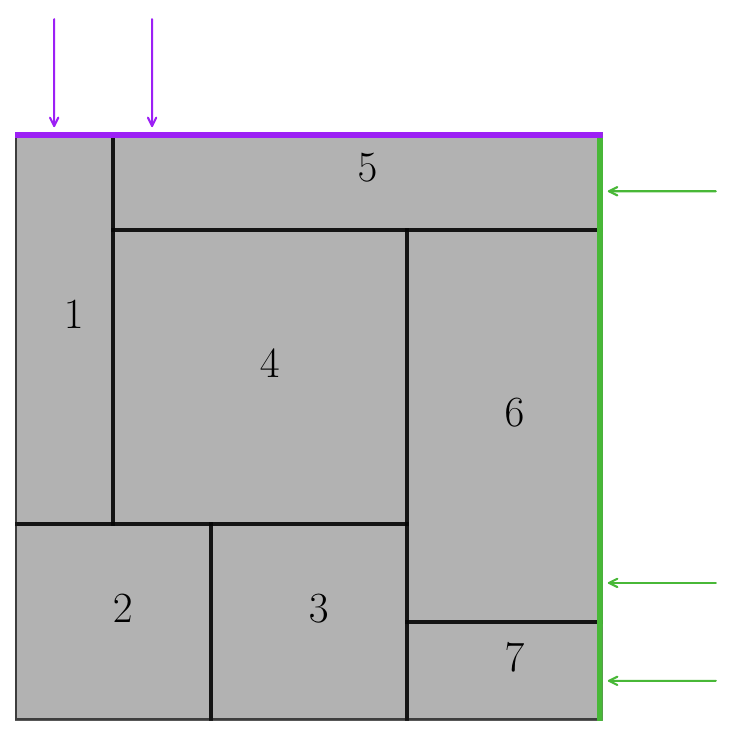} \minipdf{0.31}{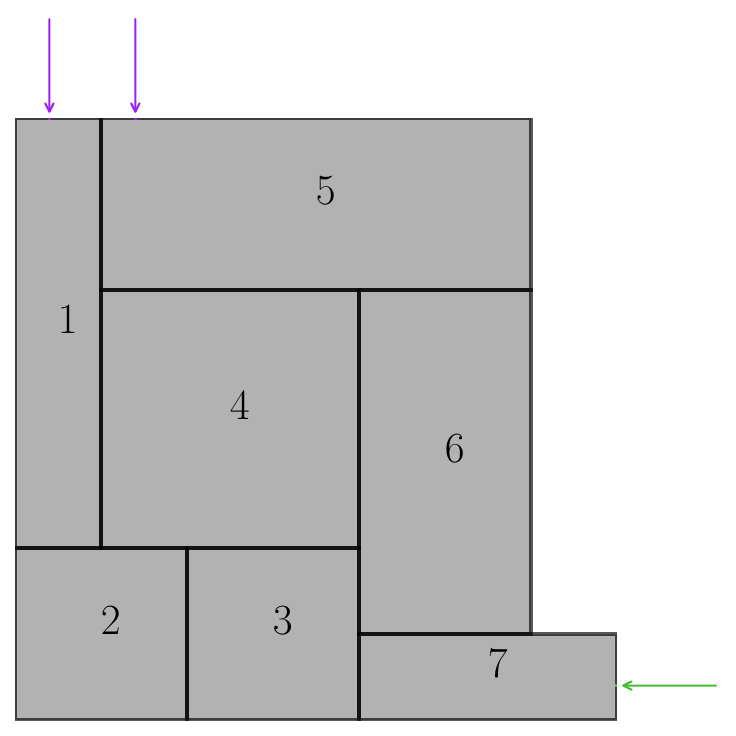} \minipdf{0.31}{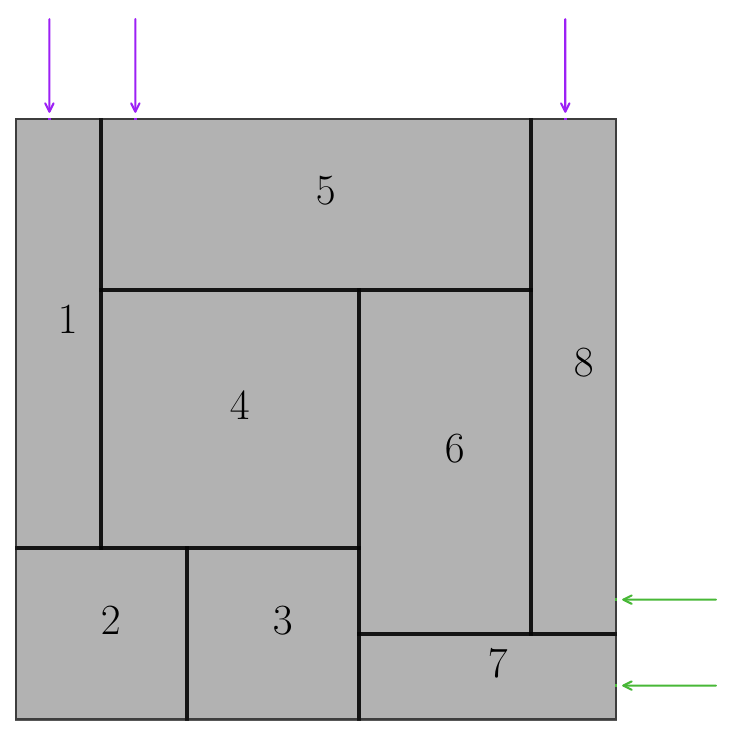}}
    \caption{Block insertion on a $2$-floorplan. Here we push on the pushable corner of block 6.}
    \label{fig:treelabel2d}
\end{figure}

\begin{figure}[!htb]
    \center{	\minipdf{0.325}{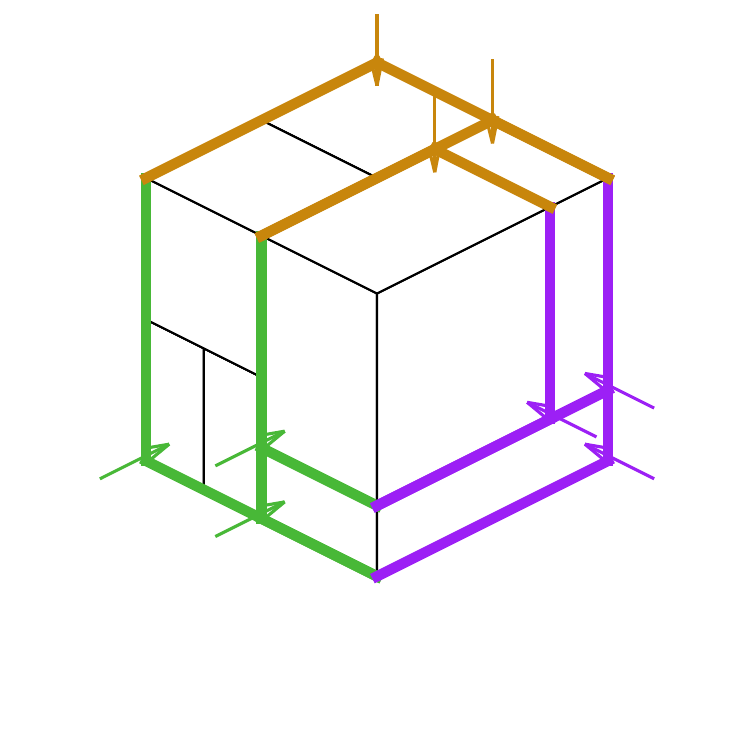}
        \minipdf{0.325}{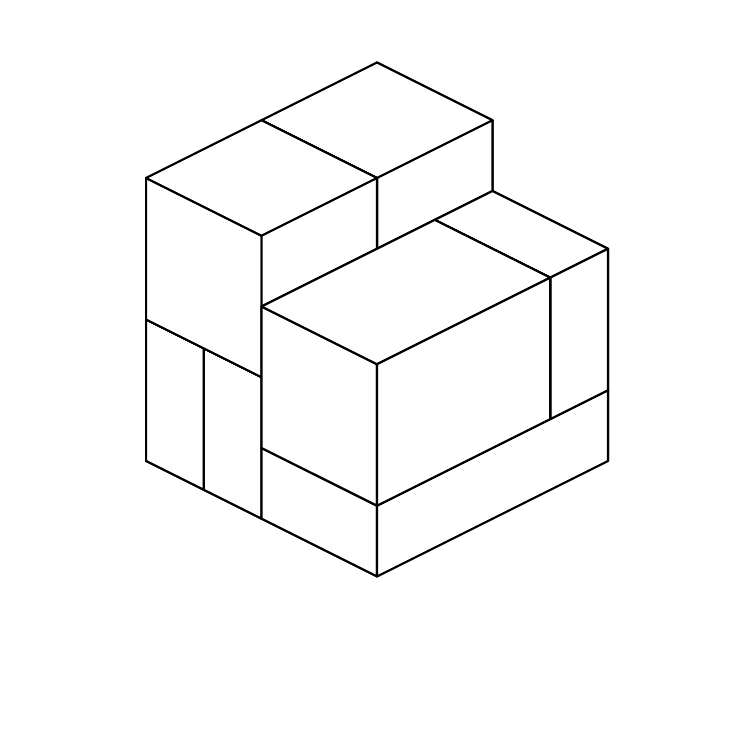}
        \minipdf{0.325}{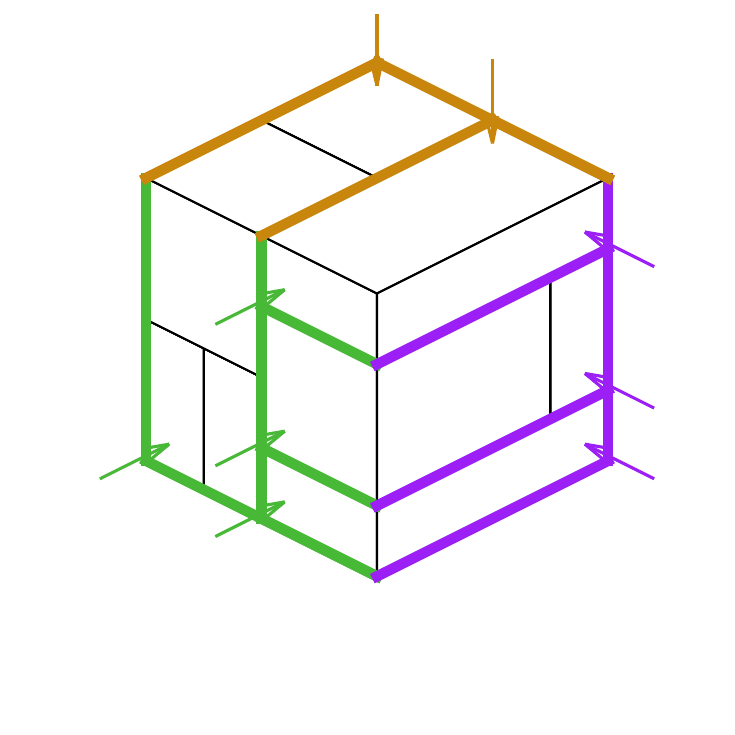} }
    \caption{An example of a block insertion on a $3$-floorplan.}
    \label{fig:ins}
\end{figure}

In $2$-floorplans, there is exactly one pushable corner of axis $x$ per rightmost block and one pushable corner of axis $y$ per topmost block. In higher dimensions, the set of pushable corners is more complex to describe.

\subsection{Pushable corners of the children of a $d$-floorplan}

We say that a pushable corner $q'=(x'_1,\dots,x'_d)$ of axis $i$ \emph{is shadowed by} a pushable corner $q=(x_1,\dots,x_d)$ if there exists $j\neq i$ such that $x'_j > x_j$.
One can remark that the pushable faces of a given axis are nested. This induces a total order on the pushable corners. The pushable corners of axis $i$ that are shadowed by a pushable corner $q$ of the same axis $i$ are exactly the pushable corners in the interior of the facet $f_q$. In $2$-floorplans, there are no other shadowed pushable corners. As one can see in Figure~\ref{fig:ins}, this is not the case in higher dimensions.

\begin{lem}\label{lem:pushable}
    Let $\mP$ be a $d$-floorplan and $q$ a pushable corner of axis $i$.
    The pushable corners of the child $\mP'$ of $\mP$ with respect to $q$ are the pushable corners of $\mP$ that are not shadowed by $q$ plus $\{ q_{new}^j, 1 \leq j \leq d \}$.
\end{lem}
\begin{proof}
    To prove this lemma, we will prove the following four properties that directly imply the result:
    \begin{enumerate}
        \item The shadowed pushable corners are not pushable in $\mP'$.
        \item The non-shadowed pushable corners are pushable in $\mP'$.
        \item $\{ q_{new}^j, 1 \leq j \leq d \}$ are pushable corners in $\mP'$.
        \item A corner that is not pushable in $\mP$ is not pushable $\mP'$.
    \end{enumerate}
    Let us prove these four properties.

        {\bf 1.}{\it The shadowed pushable corners are not pushable in $\mP'$.} Let $q'$ be a shadowed corner. First, assume that $q'$ is of axis $i$. The corresponding block in $\mP$ no longer intersects the maximal facet of axis $i$ of the bounding box of $\mP$ and $q'$ is a point of the interior of the maximal block. Hence it is not a pushable corner of $\mP'$. Now assume that $q'$ is of axis $j \neq i$. The facet $f_{q'}$ is the facet $f_{q_{new}^j}$. Because $q$ shadowed $q'$, $q_{new}^j$ is not in $f_{q'}$. Hence
    $f_{q_{new}^j}$ is not included in $f_{q'}$. So $f_{q'}$ is not the union of some block facets. So $q'$ is not pushable in $\mP'$.

        {\bf 2.}{\it The non-shadowed pushable corners are pushable in $\mP'$.} Let $q'$ be a non-shadowed pushable corner of axis $j$. As it is non-shadowed, the facet $f_{q'}$ is included in the facet $f_q$. The facet $f_{q'}$ is the union of facets of $\mP$. Hence $f_{q'}$ is also the union of facets of $\mP'$.

        {\bf 3.}{\it $\{ q_{new}^j, 1 \leq j \leq d \}$ are pushable corners in $\mP'$.} Let $1 \leq j \leq d$. The facet $f_{q_{new}^j}$ is a single facet in $\mP'$. Hence it is a pushable facet.

        {\bf 4.}{\it A corner that is not pushable in $\mP$ is not pushable $\mP'$.} Assume that there is a pushable corner $q'$ in $\mP'$ that is not pushable in $\mP$. The facet $f_{q'}$ is the union of facets in $\mP$. All these facets are also facets of $\mP$ except the newly created facet $f_{new}$. By construction, $f_{new}$ is the union of facets of $\mP$. Hence $f_{q'}$ is also a pushable facet of $\mP$, which is a contradiction.
\end{proof}

\subsection{Encoding the generating tree}

In the generating tree of $d$-floorplans, the number of children of a floorplan is determined by its set of pushable corners. One can remark that the relevant information about pushable corners is not the values of the coordinates of the corners, but simply the order of the values of the coordinates of the corners. We can thus encode the set of pushable corners by replacing their coordinates by their ranks among the pushable corners with respect to each axis.

Let us consider a set of pushable corners $\mathcal{Q}$. Let $q$ be a pushable corner of axis $i$. We encode $q$ by a vertical vector by replacing its $i$-th entry by the rank (starting at 0) of its $i$-th coordinate among corners of $\mathcal{Q}$ and we replace its $i$-th coordinate by "$.$", meaning that it has the largest rank on this coordinate. The \emph{label} of a $d$-floorplan $\mP$ with $\mathcal{Q}$ as a list of pushable corners is the concatenation of the vector of the pushable corners of $\mP$. For instance, the label associated with the $3$-floorplan on the right in Figure~\ref{fig:ins} is $\left(\vecttx{.}{0}{0}\vecttx{.}{1}{0}\vecttx{.}{1}{1}\vecttx{.}{1}{2}\vectty{0}{.}{0}\vectty{0}{.}{1}\vectty{0}{.}{2}\vecttz{0}{0}{.}\vecttz{0}{1}{.}\right)$.

A vector is of axis $i$ if its $i$-th entry is "$.$".

A vector $v'$ of axis $i'$ is \emph{shadowed} by another vector $v$ of axis $i$, if $v'_j > v_j$ for some $j \not \in \{i, i'\}$.

Given $\mathcal{L'}$ a list of vectors, and $i$ an axis, we define the \emph{new core vector associated with $\mathcal{L'}$ on the axis $i$} :
\[
    v_{\text{new}} :=
    \begin{matrix}
        \max_1     \\
        \vdots     \\
        \max_{i-1} \\
        \max_i + 1 \\
        \max_{i+1} \\
        \vdots     \\
        \max_d
    \end{matrix}
\]
where $max_j$ is the maximum of the $j$-th entry of the vectors in $\mathcal{L}$. Let $v_{\text{new}}^j$ be the vector obtained by replacing the $j$-th entry by "$.$".

Now we are ready to define a generating tree $T_d^l$ on vectors that is isomorphic to the generating tree of $d$-floorplans.

The root of $T_d^l$ has the label composed of a list of $d$-vectors such that the entries of the $i$-th vector are 0 except on the $i$-th entry which is "$.$".

Each node with label $\mathcal{L}$ has $|\mathcal{L}|$ children, one per vector. The label of the child associated with the vector $v$ of axis $i$ is $\mathcal{L'} \cup \{v_{\text{new}}^j, 1 \leq j \leq d\}$, where $\mathcal{L'}$ are the vectors of $\mathcal{L}$ that are not shadowed by $v$ and where $v_{\text{new}}$ is the new core vector associated with $\mathcal{L'}$ on the axis $i$.

\begin{figure}
    \centering
    \includegraphics[scale=0.69]{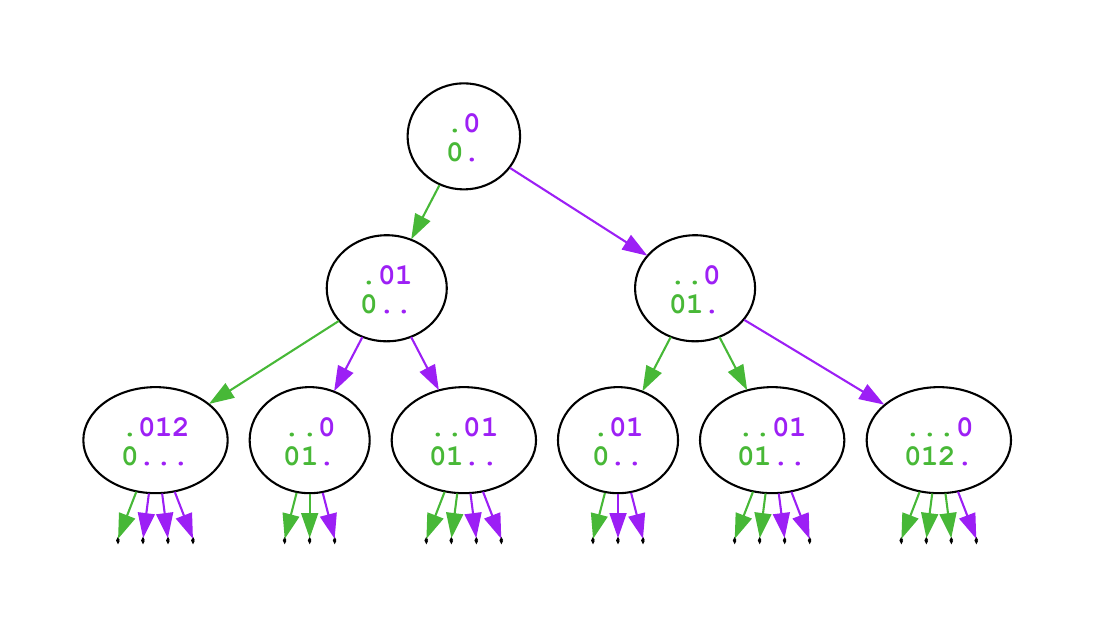}
    \caption{The first three levels of the generating tree $T_d^l$.}
    \label{fig:gen-tree-2d}
\end{figure}

Using the labeling and the rewriting rules above, one can find the first numbers of the enumeration sequence of $d$-floorplans for different values of $d$, see Table~\ref{tab:seqFlor}

Note that the sequences found for $d \ge 3$ do not appear in OEIS.

\begin{table}[!htb]
    \center{\begin{tabular}{|c||c|c|c|c|}
            \hline
            $n \backslash d$ & 2          & 3             & 4            & 5             \\
            \hline
            \hline

            1                & 1          & 1             & 1            & 1             \\
            \hline

            2                & 2          & 3             & 4            & 5             \\
            \hline

            3                & 6          & 15            & 28           & 45            \\
            \hline

            4                & 22         & 93            & 244          & 505           \\
            \hline

            5                & 92         & 651           & 2392         & 6365          \\
            \hline

            6                & 422        & 4917          & 25204        & 86105         \\
            \hline

            7                & 2074       & 39111         & 278788       & 1221565       \\
            \hline

            8                & 10754      & 322941        & 3193204      & 17932745      \\
            \hline

            9                & 58202      & 2742753       & 37547284     & 270120905     \\
            \hline

            10               & 326240     & 23812341      & 450627808    & 4151428385    \\
            \hline
            11               & 1882960    & 210414489     & 5497697848   & 64839587065   \\
            \hline
            12               & 11140560   & 1886358789    & 67979951368  & 1026189413865 \\
            \hline
            13               & 67329992   & 17116221531   & 850063243936 &               \\
            \hline
            14               & 414499438  & 156900657561  &              &               \\
            \hline
            15               & 2593341586 & 1450922198319 &              &               \\
            \hline
        \end{tabular}}
    \caption{Values of $|F_n^{d}|$ for the first values of $n$.}
    \label{tab:seqFlor}
\end{table}

\begin{example}
    Figure~\ref{fig:gen_tree_3d} shows the first levels of the generating trees of $2-$floorplans and $3-$floorplans obtained using the labeling introduced in this section.
    \begin{figure}[!htb]
        \centering
        $\vcenter{\hbox{\includegraphics[angle=90,origin=c,height=0.75\textwidth]{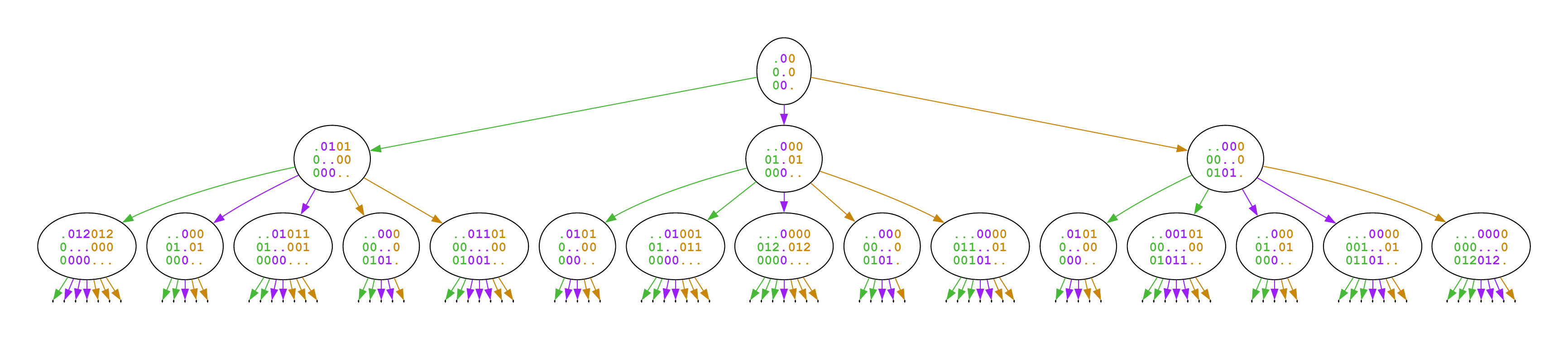}}}$
        $\vcenter{\hbox{\includegraphics[angle=90,origin=c,height=0.9\textwidth]{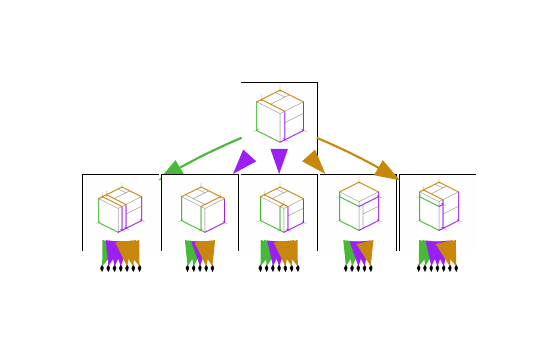}}}$
        \caption{On the left: Generating trees of  $2$-floorplans. On the right: Part of the generating tree in $3d$.}
        \label{fig:gen_tree_3d}
    \end{figure}

\end{example}

\section{Bijection between higher dimensional floorplans and \texorpdfstring{$d$}{d}-permutations}
\label{sec:bijec}

In~\cite{ackerman2006bijection}, a bijection between $2$-floorplans and Baxter permutations was proposed. This bijection relies on peeling orders of the blocks of the floorplans with respect to two corners of the floorplans (see Figure~\ref{fig:2d-orders}).

\begin{figure}[!htb]
    \center{\minipdf{0.31}{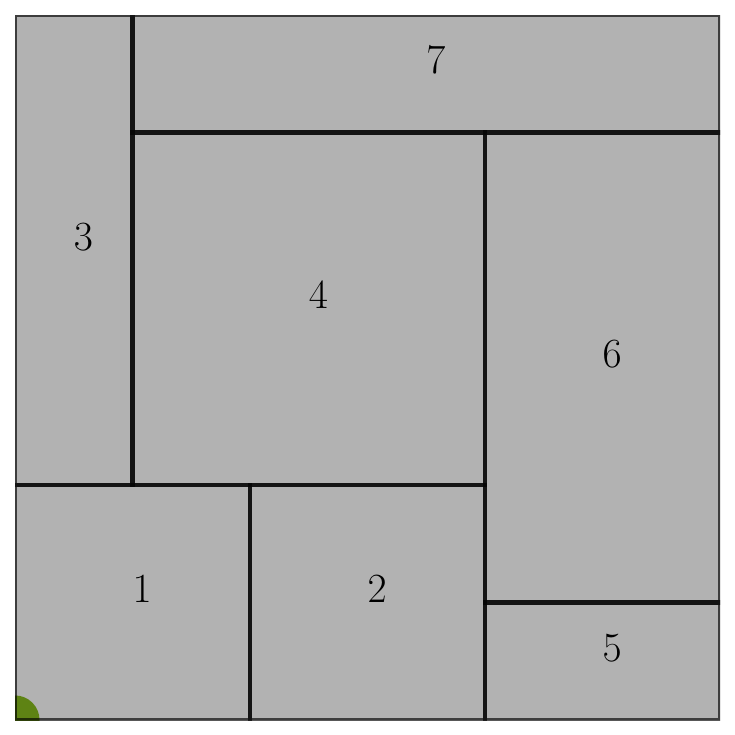} \minipdf{0.31}{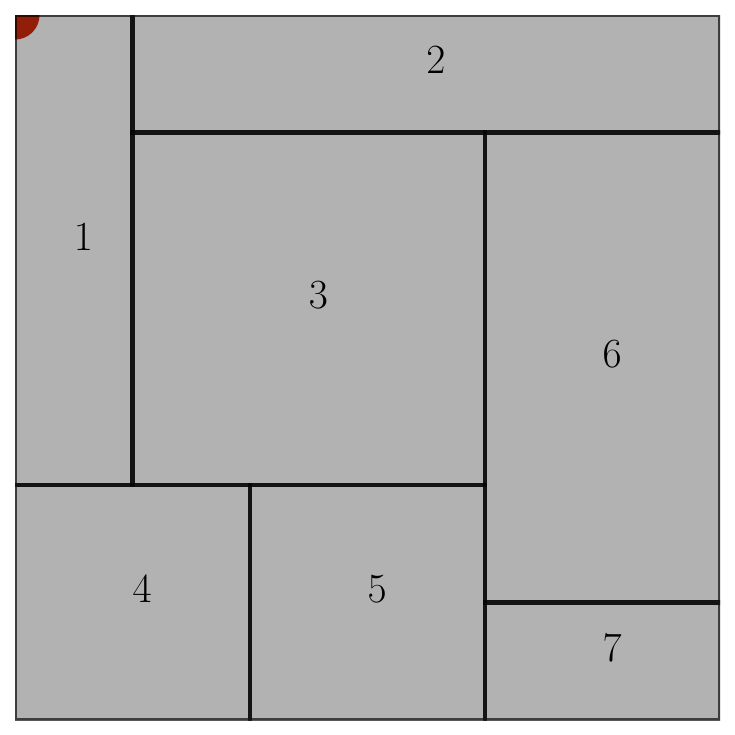} \begin{minipage}{0.31\textwidth}\center{\resizebox{0.9\textwidth}{!}{\input{figs/bp-fp2d-ex1.pgf}}}\end{minipage}}
    \caption{A mosaic floorplan, the labels of the blocks are given by the peeling order of the floorplan with respect to the top-left corner (on the left) and the bottom left corner (in the middle). On the right, the corresponding Baxter permutation 4513762.}
    \label{fig:ex_peeling-order}
\end{figure}

In this section, we propose a generalisation of this bijection between $2^{d-1}$-floorplans and the set of $d$-permutations characterised by some forbidden patterns. In this generalisation, we consider $d$ peeling orders. As it will become clear in the sequel, the choice of the peeling corners is crucial.

Without any loss of generality, we consider partitions of the unitary $d$-dimensional cube ($x_{i,min}(\mathcal{P})=0$ and $x_{i,max}(\mathcal{P})=1$, for $i=1,\ldots,d$).

\subsection{Definition of the bijection}

In higher dimensions, permutations are generalised by {\it multidimensional permutations}, also called $d$-permutations~\cite{asinowski2010separable,bonichon2022baxter}. A $d$-permutation of $[n]$ is a sequence of $d-1$ permutations  $\bpi =(\pi_1,\ldots,\pi_{d-1})$. Given a $d$-permutation $\bpi$, $\pi_0$ is the identity permutation on $[n]$.

The \emph{diagram} of a $d$-permutation $\bpi$ is the set of points in $P_\bpi:=\{ (\pi_0(i), \pi_1(i),\dots,\pi_{d-1}(i)), i \in [n] \}$.

A permutation can be considered as a total order of $[n]$, similarly a $d$-permutation can be considered as a $(d-1)$-tuple of total orders of $[n]$. Thus, any set of $d$ peeling orders obtained from a $2^{d-1}$-floorplan can be used to define a $d$-permutation and yields a mapping $\phi$ from $2^{d-1}$-floorplans to $d$-permutations. Similarly, one can define a mapping $\psi$ from $d$-permutations to lists of $2^{d-1}$ partial orders (on the same ground set) as in Definitions~\ref{def:direction} and~\ref{def:BP2FP}. It is shown in section~\ref{sec:proof} that these two mappings are bijective.

The block deletion operation also provides a notion of total orders of the blocks of a $d$-floorplan.
\begin{definition}
    Given a corner $q$ of a $d$-floorplan $\mP$ one can perform recursive block deletions using this corner until there remains a single block. We call this operation the \emph{peeling} of $\mP$ with respect to $q$.

    The \emph{peeling order} $ \twoheadiagarrow^q$ of the blocks of $\mP$ is the total order defined by the order of deletion of the blocks during the peeling of $\mP$ with respect to the corner $q$. For two blocks $A$ and $B$, one has $A \twoheadiagarrow^q B$ if $A$ is deleted before $B$ during the peeling.
    \label{peeling}
\end{definition}

\begin{example}
    In Figure~\ref{fig:ex_peeling-order} are given the labels in the peeling orders of the blocks of the first two mosaic floorplans given in Figure~\ref{fig:2dfloor}. On the left, the label of the blocks is given by the peeling order with respect to the bottom-left corner. On the right, the label is given by the one with respect to the top-left corner.

    Figure~\ref{fig:ex-peel-3d} shows the last $3$-floorplan of Figure~\ref{fig:ex_fp}. The label of the blocks is given by their labels in the peeling orders with respect to the corner $q'_0=(0,0,0)$ (on the left), the corner $q'_1=(0,1,0)$ (in the middle), and the corner $q'_2=(0,0,1)$ (on the right).
    \begin{figure}[!htb]
        \center{	\minipdf{0.31}{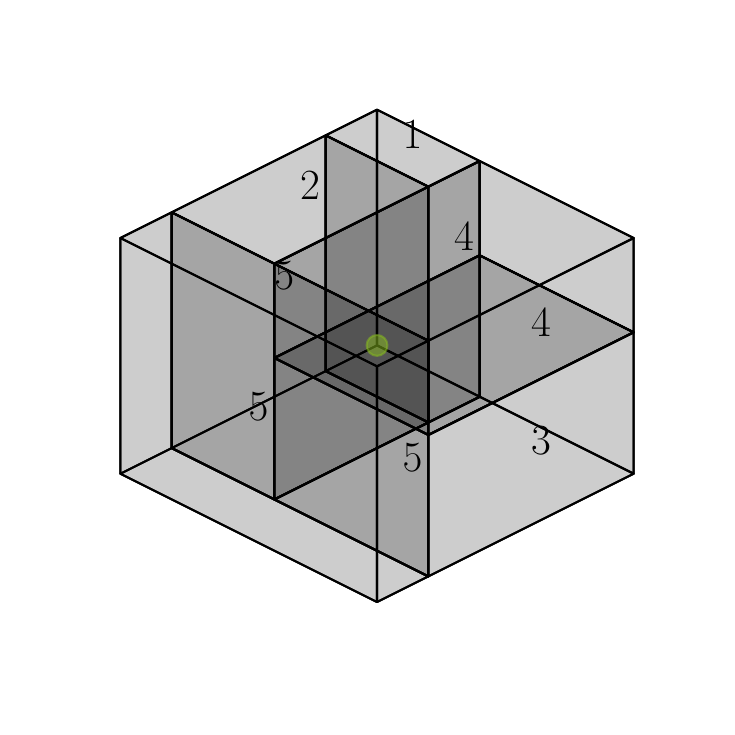}
            \minipdf{0.31}{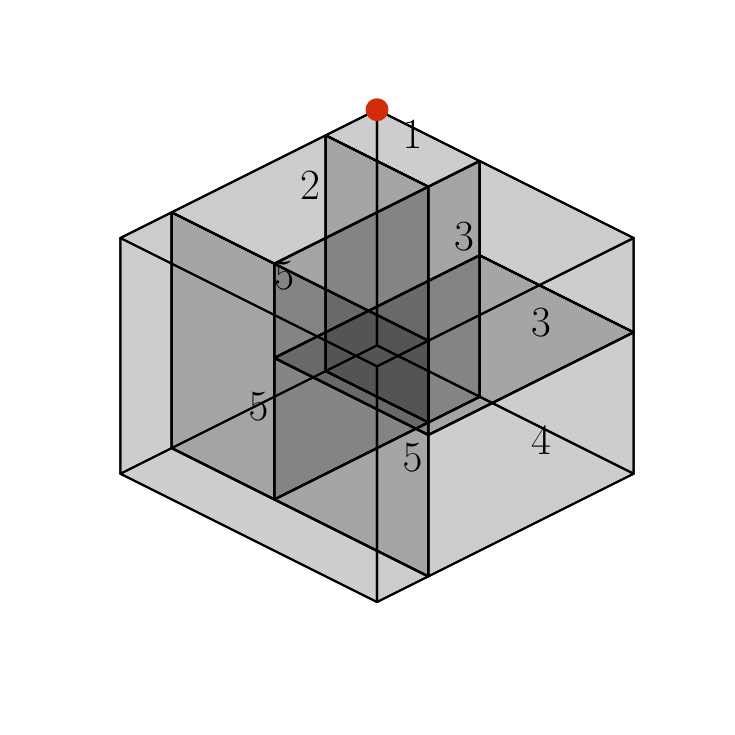}\minipdf{0.31}{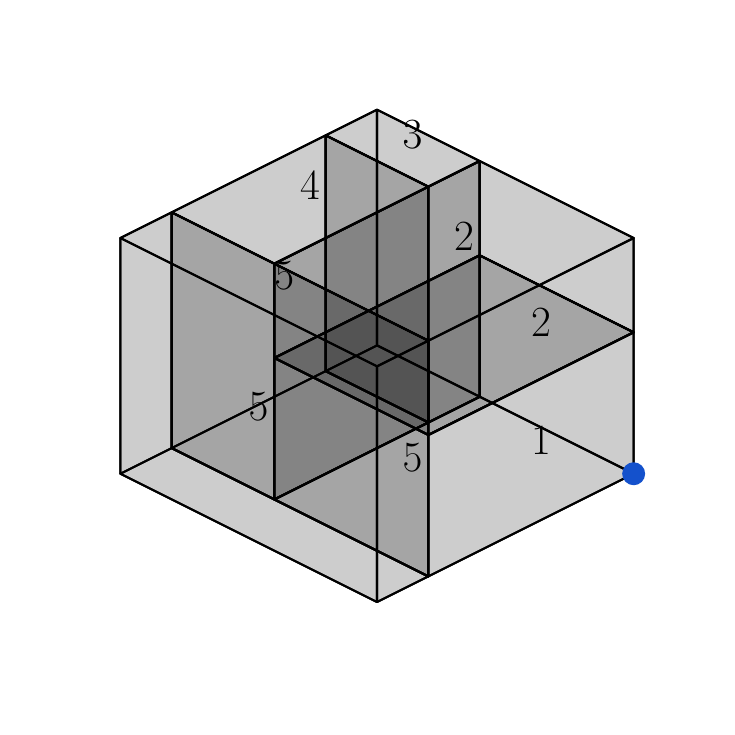}}
        \caption{Peeling orders of the blocks of a $3$-floorplan with respect to the corners $q=(1,1,1)$ (on the left), $q=(0  ,0,1)$ (in the middle), and $q=(0,1,0)$ (on the right).}
        \label{fig:ex-peel-3d}
    \end{figure}
\end{example}

Knowing the order of deletion of the blocks with respect to a corner $q$ gives also the order of deletion with respect to the opposite corner of $q$. Hence, in a $d$-floorplan, there are $2^{d-1}$ independent peeling orders, one per pair of corners.

\begin{definition}
    Let $\mP$ be a $2^{d-1}$-floorplan, let also $\textbf{c}= c_0 \ldots c_{d-1}$ be a set of corners of $\mP$ and let $\sigma_{c_i}$ be the permutation defined by the list of the blocks of $\mP$, ordered according to $\twoheadiagarrow^{c_i}$. The $d$-permutation of $\mP$ with respect to ${\bf c}$ is defined as $\bpi:=(\sigma_{c_1} \sigma_{c_0}^{-1},\ldots, \sigma_{c_{d-1}} \sigma_{c_0}^{-1})$. This defines a mapping $\chi_{\bf c}$ from $2^{d-1}$-floorplans to $d$-permutations given by $\chi_{\bf c}(\mathcal{P}):= \bpi$.
    \label{def:FP2BP}
\end{definition}

Even though any set of $d$ peeling orders of a $2^{d-1}$-floorplan can be used to define a $d$-permutation, the mapping $\chi_{\bf c}$ is not bijective for any such set. For $\chi_{\bf c}$ to be bijective, the set of corners {\bf c} must fulfill the condition that half of the coordinates of each corner are different from the coordinates of the other ones. This condition allows us to prove the different lemmas and theorems of Section~\ref{sec:proof}. Note that, even with such a condition on the corner set, there are still several possible choices for {\bf c}. In this paper, we consider an example of such a corner set that is given in Definition~\ref{def:deletion_order}.

\begin{definition}
    Let $\mP$ be a $2^{d-1}$-floorplan with $n >1$ blocks. We define the set of \emph{canonical corners $\textbf{q}$} as the set of $d$ corners $q_0 \ldots q_{d-1}$ for which the coordinates of $q_i$ are given by an alternation of packets of $2^{d-1-i}$ zeros and ones such that
    \begin{equation}
        q_i(\mathcal{P}) = \big( \underbrace{0,\ldots,0}_{2^{d-1-i}}, \underbrace{1,\ldots,1}_{2^{d-1-i}}, \underbrace{0,\ldots,0}_{2^{d-1-i}}, \; \ldots \; \big) \; .
    \end{equation}
    We further call the set of peeling orders with respect to these corners the \emph{canonical peeling orders}.
    \label{def:deletion_order}
\end{definition}

For instance, for a $2$-floorplan, the canonical corners are: $q_0=(0,0)$ and $q_1=(0,1)$. For a $4$-floorplan $q_0=(0,0,0,0)$, $q_1=(0,0,1,1)$, and $q_2=(0,1,0,1)$. Note that a $3$-floorplan can also be seen as a $4$-floorplan, where each block has width 1 on the last coordinate. In that case, the canonical peeling orders can also be computed with the following corners $q'_0=(0,0,0)$, $q'_1=(0,0,1)$, and $q'_2=(0,1,0)$ (see Figure~\ref{fig:ex-peel-3d}).

We can now define the mapping $\phi$ as a restriction of the mapping $\chi_{\bf c}$ to the set of canonical corners {\bf q}:
\begin{definition}
    We call $\phi$, the mapping $\chi_{\bf q}$ where {\bf q} is the canonical set of corners. Starting from a $2^{d-1}$-floorplan $\mP$, the mapping $\phi$ gives a $d$-permutation $\phi(\mathcal{P})$ called the \emph{canonical $d$-permutation} of $\mP$.
    \label{def:FP2BP}
\end{definition}

Let us also define a mapping $\psi$ that extracts partial orders from the points of a $d$-permutation.  A \emph{direction} ${\bf dir}$ is an element of $\{+1,-1\}^d$. A direction is \emph{positive} if its first element is $+1$. In a $d$-permutation, there are $2^{d-1}$ positive directions. Let also ${\bf dir}$ be a positive direction, the \emph{opposite} of ${\bf dir}$, denoted $(-{\bf dir})$, is the direction such that $(-{\bf dir})=(-1)\times {\bf dir}$.

\begin{definition}
    Let $\bpi$ be a $d$-permutation with $n$ points and let $p_1$ resp. $p_2$ be two points in $\bpi$. The \textbf{direction} between the two points $(p_1,p_2)$, denoted $\textbf{dir}(p_1,p_2)$, is defined as the direction ${\bf dir}$ such that $\left(\text{sign}(x_0(p_2)-x_0(p_1)),\ldots, \text{sign}(x_{d-1} (p_2) - x_{d-1}(p_1))\right)={\bf dir}$. We say that \emph{$p_1$ precedes $p_2$} with respect to the positive direction ${\bf dir}$ if $p_1=p_2$ or  $\textbf{dir}(p_1,p_2)=({\bf dir})$. Similarly, we say that $p_1$ follows $p_2$ in the positive direction ${\bf dir}$ if $p_1=p_2$ or $\textbf{dir}(p_1,p_2)=(-{\bf dir})$. These two relations define two partial orders $<_{\bf dir}$ and $>_{\bf dir}$ of the points of $\bpi$.
    \label{def:direction}
\end{definition}

\begin{definition}
    Let $\bpi$ be a $d$-permutation. Let $F=\{{\bf dir}^1,\ldots,{\bf dir}^{2^{d-1}}\}$ be the set of positive directions in dimension $d$. Let also $<_{{\bf dir}_i}$ and $>_{{\bf dir}_i}$  be the partial orders of the points of $\bpi$ with respect to the direction ${\bf dir}_i$. The mapping $\psi$ gives the $2^{d-1}$ partial orders $(<_{{\bf dir}^1}|\ldots| <_{{\bf dir}^{2^{d-1}}})$ of a $d$-permutation $\bpi$. $\big(\psi(\bpi) := (<_{{\bf dir}^1}| \ldots |$ $<_{{\bf dir}^{2^{d-1}}})\big)$
    \label{def:BP2FP}
\end{definition}

As explained in subsection~\ref{sec:peeld}, a $2^{d-1}$-floorplan can be described by a set of $2^{d-1}$ partial orders on the same ground set. These partial orders correspond to the direction relations of the blocks of this $2^{d-1}$-floorplan. Additionally, the partial orders of a $d$-permutation have the same properties as the direction relations of the blocks of a $2^{d-1}$-floorplan. One can thus try to build a $2^{d-1}$-floorplan from the partial orders of the points of a $d$-permutation (the algorithm to do so is provided in section~\ref{sec:proof}). This is done by associating to each point of a permutation, a block in the $2^{d-1}$-floorplans to be built. However, as it is proven in section~\ref{sec:proof}, only sets of partial orders coming from a subset of the $d$-permutations can be used to build a $2^{d-1}$-floorplans (this subset is defined in the next section).

\subsection{Forbidden pattern and main theorem}
\label{ssec:mainthm}

We first recall from~\cite{bonichon2022baxter} some useful definitions on $d$-permutations and forbidden patterns. We then give the main theorem of this paper.

\begin{definition}
    Let $\textbf{i}=i_1 \dots i_{d'}$ be a sequence of indices in $\{0,\ldots,d\}$, let also $\bsig=(\sigma_1, \dots , \sigma_{d-1})$ be a $d$-permutation of size $n$. The \emph{projection} on $\textbf{i}$ of $\bsig$ is the $d'-$permutation given by $\text{proj}_{\bf i}(\bsig) := (\sigma_{i_2} \sigma_{i_1}^{-1},\dots, \sigma_{i_{d'}} \sigma_{i_1}^{-1}) $. A projection is \emph{direct} if $i_1 < i_2 < \dots < i_{d'}$.
    \label{def:proj}
\end{definition}

\begin{definition}
    Let the $d$-permutation $\bsig = (\sigma_1, \dots , \sigma_{d-1}) \in S^{d-1}_n$ and the $d'-$permutation $\bpi = (\pi_1,\ldots , \pi_{d'-1}) \in S^{d'-1}_k$ with $k \leq n$ and $d'\leq d$. Then \emph{$\bsig$ contains the pattern $\bpi$} if there exists a direct projection $\bsig' = \text{proj}_i(\bsig)$ of dimension $d'$ and a set of indices $c_1 < \ldots < c_k$ s.t. $\sigma'_i(c_1) \ldots \sigma'_i(c_k)$ is order-isomorphic to $\pi_i$ for all $i$. A permutation avoids a pattern if it doesn’t contain it. Let $s$ be a symmetry operation of the $[n]^d$ grid (seen as a $d$-cube), $s(P_\bpi)$ is a diagram of a $d$-permutation that we denote $s(\bpi)$. We also denote by $\Sym(\bpi)$ the family of permutations obtained by applying the symmetries of the $d$-cube on $\bpi$.
    \label{def:avoidance}
\end{definition}

Let also $p_i$ and $p_j$ be the $i^{th}$ and $j^{th}$ points with respect to the axis $0$ in a $d$-permutation $\sigma$. We say that $p_i$ and $p_j$ are \emph{k-adjacents} if one has $\sigma_k(p_i) = \sigma_k(p_j) \pm 1 $. We also say that $p_i$ and $p_j$ are $0$-adjacent if $i= j \pm 1$. One can now define a \emph{generalised vincular pattern} as follows:

\begin{definition}
    A \emph{generalised vincular pattern} $\vinpat{\bpi}{X_0, \ldots, X_{d-1}}$ is a $d$-permutation $\bpi$ of size $k$ along with a list of \emph{adjacencies} given by subsets of $[k-1]$. Given a $d$-permutation $\sigma$, we say that the set of points $p_1 \ldots p_k$ is an occurrence of $\vinpat{\bpi}{X_0, \ldots, X_{d-1}}$ if it is an occurrence of $\bpi$ and if it satisfies that for any $j$ in any $X_k$ the $j^{th}$ and the $j+1^{th}$ point (of the occurrence) with respect to the axis $k$ are $k$-adjacent.
    \label{def:vincular}
\end{definition}

Let us denote by $S^{d-1}_n$ the set of $d$-permutations with $n$ points and by $S^{d-1}_n(\pi_1, \ldots, \pi_m)$ the set of such $d$-permutations avoiding a list of $m$ patterns $\pi_1 \ldots \pi_m$ (that can be vincular).

\begin{example}
    Consider the patterns $k_1=\perm{132}{213}$, $k_2=231$, $k_3= \vinpatd{132}{213}{1}{.}{.}$ and  $k_4= \vinpatd{132}{213}{2}{1}{.}$. The $3$-permutation $\bpi=\perm{1432}{3124}$ contains an occurrence of $k_1$, given by the $1^{st},3^{rd}$ and $4^{th}$ points of $\bpi$, it also contains an occurrence of $k_4$ but not one of $k_3$ because of the adjacency constraint on these patterns. As $\text{proj}_{1,2}(\bpi)= 3421$, $\bpi$ also contains an occurrence of $k_2$.
\end{example}

One can consider classes of pattern-avoiding $d$-permutations, examples of such a class are the separable $d$-permutations that avoid $\Sym(\perm{312}{213})$ and $\Sym({2413})$~\cite{asinowski2010separable} and the Baxter $d$-permutations that avoid $\sbaxpb$ $\sbaxpc$, $\sbaxpa$, and $\sbaxpd$~\cite{bonichon2022baxter}. 

In this paper, we deal with $d-$permutations avoiding $\text{Sym}(\perm{312}{213}), \sbaxpa$. We call this permutation class $F^{d-1}$, and we denote its set of elements with n points as $F^{d-1}_n$.

Let us now state the main result of this section.

\begin{thm}

The mappings $\phi$ and $\psi$ (Def.~\ref{def:BP2FP} and~\ref{def:FP2BP}) define a bijection between $F^{d-1}_n$ and the set of $2^{d-1}$-floorplans with $n$ blocks. One has $\psi= \phi^{-1}$.
    \label{th:final}
\end{thm}

As noted in~\cite[Corollary 3.3]{asinowski2010separable} we can naturally extend the previous theorem to floorplans of arbitrary dimensions (not necessarily a power of 2).
\begin{cor}
    Let $q \leq 2^{d-1}$.
    There is a bijection between the set of $d$-permutations avoiding $2^{d-1}-q$ (among $2^{d-1}$) $d$-permutations of size $2$ and $\text{Sym}(\perm{312}{213}), \sbaxpa$ with the set of $2^{d-1}-q$-dimensional floorplans.
\end{cor}

A \emph{guillotine partition} is a $d$-floorplan recursively defined as follows: a $d$-floorplan with one block is a guillotine partition and two guillotine floorplans merged along a common extremal facet of their bounding boxes is a guillotine partition. Guillotine partitions of dimension 2 are called \emph{slicing floorplans}.

A $d$-permutation is \emph{separable} if it is of size 1 or if its diagram can be partitioned into 2 parts $P_1$ and $P_2$ such that there exists a direction ${\bf dir}$ such that $\forall p_1 \in P_1, \forall p_2 \in P_2, \; \textbf{dir}(p_1,p_2)={\bf dir}$. In~\cite{asinowski2010separable} it is shown that separable $d$-permutations are in bijection with $2^{d-1}$-dimensional guillotine partitions. In the same paper the authors also show that separable $d$-permutations are the $d$-permutations avoiding $\Sym(\perm{312}{213})$ and $\Sym({2413})$. From the recursive definitions of guillotine partitions and separable $d$-permutations, we get the following corollary:

\begin{cor}
    The bijection defined by Theorem~\ref{th:final} is a bijection between  $2^{d-1}$-dimensional guillotine partitions and separable $d$-permutations.
\end{cor}

\begin{example}
    In dimension $2$, the bijection given above reduces to the one known between mosaic floorplans and Baxter permutations (see~\cite{ackerman2006bijection}). One can summarize the bijection between a $2$-floorplan and a permutation in the following table:

    \begin{table}[!htb]
        \center{\begin{tabular}{|c||c|}
                \hline
                Permutation                           & $2$-floorplans                         \\
                \hline
                \hline

                points                                & blocks                                 \\
                \hline

                positive direction between two points & direction relation between two blocks  \\
                \hline

                x-coordinate of the points            & peeling order w.r.t bottom left corner \\
                \hline

                y-coordinate of the points            & peeling order w.r.t top left corner    \\
                \hline
            \end{tabular}}
    \end{table}

    Figure~\ref{fig:ex_peeling-order} shows a $2$-floorplan and its Baxter permutation obtained through the bijection given above. This defines the permutation $\pi = 4513762$.

    In the $3$-dimensional case, Figure~\ref{fig:bijec} shows an illustration of the extended bijection. To obtain the relation between the $3$-floorplan and the $3$-permutation, one considers a $4$-floorplan for which there is an empty direction relation and that can thus be projected into a $3$-floorplan. In the permutation space this results in considering $3$-permutations that avoid a pattern of size $2$.

    \begin{figure}[!htb]
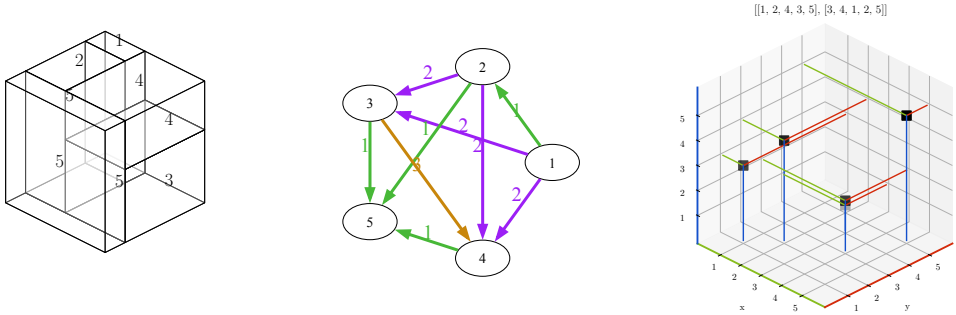

        \center{
            \minipdf{0.30}{figs/EXAMPLE2_FP.pdf}
            \minipdf{0.30}{figs/example2fp-order-all.pdf} \minipdf{0.30}{figs/EXAMPLE2_BP.pdf}}
        \caption{On the left, a $3$-floorplan of size 5. In the middle, the 3 corresponding orders $( \protect\overset{1}{\twoheadleftarrow}, \protect\overset{2}{\twoheadleftarrow}, \protect\overset{3}{\twoheadleftarrow})$.  On the right, the 3-permutation $\bpi=\perm{12435}{34125}$ corresponding to the $3$-floorplan (seen as a $4$-floorplan in which there are no covering relations in $ \protect\overset{4}{\twoheadleftarrow}$).}
        \label{fig:bijec}
    \end{figure}

\end{example}

\section{Proof of the bijection}
\label{sec:proof}

We prove in this section the bijection of Theorem \ref{th:final}. The proof strategy is the following:

\begin{itemize}
    \item We prove different properties of the direction relations and of the peeling orders of $2^{d-1}$-floorplans. These properties are the main tools used at the different steps of the proof of the bijection.
    \item We prove that $\phi$ maps $2^{d-1}$-floorplans to $d$-permutations in $F^{d-1}_n$ and that this mapping is injective.
    \item We prove that the partial orders extracted by the mapping $\psi$ from a $d-$permutation in $F^{d-1}_n$ can be used to construct a $2^{d-1}$-floorplan. We provide an algorithm performing the construction.
    \item We prove that, applying this algorithm to a $d-$permutation in $F^{d-1}_n$ and applying the mapping $\phi$ on the resulting $2^{d-1}$-floorplans gives back the original $d$-permutation.
\end{itemize}

\begin{remark}
    Let $\mP$ be a $d$-floorplan with $n > 1$ blocks. Performing a block
    deletion
    in $\mP$ doesn't change the $\overset{j}{\twoheadleftarrow}$ relations of
    the
    other blocks.
    \label{orderdel}
\end{remark}

\begin{lem}
    Let $\mP$ be a $d$-floorplan and let $A$ and $B$ be two blocks in $\mP$.
    Then, there exists a unique $j$ such that $A \overset{j}{\twoheadleftarrow} B$ or $A \overset{j}{\twoheadrightarrow} B$.
    \label{lem:partial}
\end{lem}
\begin{proof}
    By Remark~\ref{orderdel}, one can delete blocks of $\mP$ (using different
    corners) without changing the direction relation of $A$ and $B$. For
    each corner of $\mP$, delete blocks with respect to this corner until either
    $A$ or $B$ contains it and call the resulting floorplan $\mathcal{P}'$.
    A block cannot contain both a corner of a $d$-floorplan and its opposite
    without being the only block in this $d$-floorplan. Thus, in $\mathcal{P}'$
    the block $A$ will contain half of the corners of $\mathcal{P}'$ and
    $B$ the other half. This implies that the blocks $A$ and $B$ have a
    single facet of the same type that is not contained in the boundary of
    $\mathcal{P}'$ and that there may be a smaller $d$-dimensional floorplan
    between these two faces in $\mathcal{P}'$. This situation is shown in Figure
    ~\ref{fig:lemma31}.

    We thus have $A \overset{j}{\twoheadleftarrow} B$ or $B
        \overset{j}{\twoheadleftarrow} A$ for a single $j$ (given by the axis of
    the border not contained in the boundaries of $\mathcal{P}'$) in
    $\mathcal{P}'$. According to Remark~\ref{orderdel} we have the same
    relation between the two blocks in $\mP$.
\end{proof}

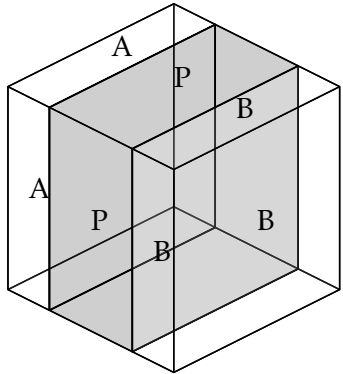
\begin{figure}[!htb]
    \centering
    \begin{minipage}{0.5\textwidth}\center{\resizebox{0.9\textwidth}{!}{\input{figs/fp-lemma31.pgf}}}\end{minipage}
    \caption{Illustration of Lemma~\ref{lem:partial}.}
    \label{fig:lemma31}
\end{figure}

\begin{remark}
    If a block $A$ follows immediately a block $B$ in any peeling order, then there exists
    a $j$ such that $A$ and $B$ are $j$-neighbors.
    \label{rem:neighbor}
\end{remark}

Given two relations $<_1$ and $<_2$ on a set $X$, we define the \emph{union} $<_1 \cup <_2$ of these two relations: $A <_1 \cup <_2 B$ if $A <_1 B$ or $A <_2 B$. Similarly, we define the \emph{intersection} $<_1 \cap <_2$ of these two relations: $A <_1 \cap <_2 B$ if $A <_1 B$ and $A <_2 B$. Given a corner $q$, a peeling order $\twoheadiagarrow^q$ is \emph{compatible} with the direction relation $\overset{j}{\twoheadleftarrow}$ if the $j$-th coordinate of $q$ equals 0 and is compatible with the direction relation $\overset{j}{\twoheadrightarrow}$ if the $j$-th coordinate of $q$ equals 1. A peeling order $\twoheadiagarrow^q$ is \emph{canonical} if $q$ or $\overline{q}$ is in the canonical set of corners.

\begin{remark}
    \label{rem:compatible}
    If $ A \overset{j}{\twoheadleftarrow}B $ and $\twoheadiagarrow^q$ is compatible with $\overset{j}{\twoheadleftarrow}$ then $A \twoheadiagarrow^q B$.
\end{remark}

\begin{lem}
    \label{lem:cover}
    Given a corner $q$, the peeling order $\twoheadiagarrow^q$ is the union of $d$ direction relations that are compatible with it $\twoheadiagarrow^q$.
\end{lem}
\begin{proof}
    Without loss of generality assume that $q=q_0$ and let us show that $\twoheadiagarrow^q = \bigcup_{j=1}^d \overset{j}{\twoheadleftarrow}$. By remark~\ref{rem:neighbor}, $A \twoheadiagarrow^q B$ implies $A(\bigcup_{j=1}^d \overset{j}{\twoheadleftarrow}) B$. By remark~\ref{rem:neighbor} we get the other implication.
\end{proof}

\begin{lem}
    \label{lem:recover}
    Given a $2^{d-1}$-floorplan, the direction relation $\overset{j}{\twoheadleftarrow}$ (resp. $\overset{j}{\twoheadrightarrow}$) is the intersection of the $d-1$ canonical peeling orders that are compatible with $\overset{j}{\twoheadleftarrow}$ (resp. $\overset{j}{\twoheadrightarrow}$).
\end{lem}
\begin{proof}
    Without loss of generality assume that $j=1$ and let us show that $\overset{1}{\twoheadleftarrow} = \bigcap_{i=0}^{d-1} \twoheadiagarrow^{q_i}$. By remark~\ref{rem:neighbor} we have that $A \overset{1}{\twoheadleftarrow}B $ implies $A \bigcap_{i=0}^{d-1} \twoheadiagarrow^{q_i} B$.
    Now, let us assume that $A$ and $B$ are not comparable with respect to $\overset{1}{\twoheadleftarrow}$. By Lemma~\ref{lem:partial} they must be comparable with respect to $\overset{j}{\twoheadleftarrow}$ for some $j\neq 1$.
    By remark~\ref{rem:neighbor} $A \twoheadiagarrow^{q}$ for each canonical corner $q$ that is compatible with $\overset{j}{\twoheadleftarrow}$. By the construction of the canonical corners, at least one of these compatible corners is of the form $\overline{q}_i$. Hence we don't have $A \twoheadiagarrow^{q_i} B$. Hence we don't have $A \bigcap_{i=0}^{d-1} \twoheadiagarrow^{q_i} B$, which shows the other implication.
\end{proof}
Consider a $4-$floorplan,
the peeling orders can be written as the union of the partial orders:

\begin{align*}
    \twoheadiagarrow^{q_0} = \overset{x}{\twoheadleftarrow} \cup \overset{y}{\twoheadleftarrow} \cup \overset{z}{\twoheadleftarrow} \cup \overset{t}{\twoheadleftarrow},   \\
    \twoheadiagarrow^{q_1} = \overset{x}{\twoheadleftarrow} \cup \overset{y}{\twoheadleftarrow} \cup \overset{z}{\twoheadrightarrow} \cup \overset{t}{\twoheadrightarrow}, \\
    \twoheadiagarrow^{q_2} = \overset{x}{\twoheadleftarrow} \cup \overset{y}{\twoheadrightarrow} \cup \overset{z}{\twoheadleftarrow} \cup \overset{t}{\twoheadrightarrow}.
\end{align*}

The partial orders can be written as the intersection of the peeling orders with respect to the canonical corners:
\begin{align*}
    \overset{x}{\twoheadleftarrow} = \twoheadiagarrow^{q_0} \cap \twoheadiagarrow^{q_1} \cap \twoheadiagarrow^{q_2},            \\
    \overset{y}{\twoheadleftarrow} = \twoheadiagarrow^{q_0} \cap \twoheadiagarrow^{q_1} \cap \twoheadiagarrow^{\overline{q_2}}, \\
    \overset{z}{\twoheadleftarrow} = \twoheadiagarrow^{q_0} \cap \twoheadiagarrow^{\overline{q_1}} \cap \twoheadiagarrow^{q_2}, \\
    \overset{t}{\twoheadleftarrow} = \twoheadiagarrow^{q_0} \cap \twoheadiagarrow^{\overline{q_1}} \cap \twoheadiagarrow^{\overline{q_2}}.
\end{align*}

Remark \ref{rem:diford} is a direct consequence of Lemmas \ref{lem:cover} and \ref{lem:recover}.
\begin{remark}
    Let four blocks $A,B,C,D$ (two blocks can be the same) of a $2^{d-1}-$floorplan. If for two canonical corners $q_1$ and $q_2$, one has $$A \twoheadiagarrow^{{q_1}} B \text{ and } A\twoheadiagarrow^{q_2} B$$ $$C \twoheadiagarrow^{{q_1}} D \text{ and } D \twoheadiagarrow^{q_2} C,$$ then the axis $j$ that defines the direction relation between $A$ and $B$ is different than the one defining the direction relation between $C$ and $D$. The converse is also true.
    \label{rem:diford}
\end{remark}
\begin{lem}
    For any $2^{d-1}-$floorplan $\mP$, the $d$-permutation $\bpi=\phi
        (\mathcal{P})$ is in $F^{d-1}_n$.
    \label{lem:FP2BP}
\end{lem}

\begin{proof}
    Let us show by contradiction that any permutation containing one of the forbidden patterns of $F^{d-1}_n$ cannot be obtained from a $2^{d-1}-$floorplan. For each family of patterns, it suffices to prove this for one of their representative, the proofs of the other patterns follow by symmetry.

    Let $\mP$ be a $2^{d-1}-$floorplan with $n$ blocks and let us do the
    following suppositions:

    \begin{itemize}
        \item The $d$-permutation $\phi(\mathcal{P})$ contains \baxpa.
              This implies that there are four blocks $A,B,C,D$ in $\mP$ such
              that $$A \twoheadiagarrow^{q_l} B \twoheadiagarrow^{q_l}   C
                  \twoheadiagarrow^{q_l} D,$$  $$C \twoheadiagarrow^{q_m} A
                  \twoheadiagarrow^{q_m}   D \twoheadiagarrow^{q_m} B,$$
              for some canonical corner $q_l$ and $q_m$ ($l<m$). From the adjacency conditions of the patterns, one also has that $B$ and $C$ are $i$ neighbors and that $A$ and $D$ are $j$ neighbors by  Remark~\ref{rem:neighbor}. Between the pairs  $A-D$ and $B-C$, one has $$ A \twoheadiagarrow^{q_l}  D \quad \text{and} \quad  A \twoheadiagarrow^{q_m} D  \,,$$
              $$B \twoheadiagarrow^{q_l}   C  \quad \text{and} \quad  C \twoheadiagarrow^{q_m} B\,.$$.

              By Remark \ref{rem:diford}, one also has that $i$ and $j$ are not the same axis. Consider now the $2^{d-1}-$floorplan $\mathcal{P'}$ where:
              \begin{itemize}
                  \item All the blocks preceding $A$ in the peeling order with respect to $q_l$ have been removed by a sequence of block deletions using the corner $q_l$.
                  \item All the blocks coming after $D$ in this same peeling
                        order have been removed by a sequence of block deletion using
                        the opposite corner of $q_l$ in $\mP$.
                  \item All the blocks preceding $C$ in the peeling order with respect to $q_m$  have been removed by a sequence of block deletions using the corner $q_m$.
                  \item All the blocks coming after $B$ in this same peeling
                        order have been removed by a sequence of block deletion using
                        the opposite corner of $q_m$ in $\mP$.
              \end{itemize}

              The block $A$ thus contains the corner $q_l$ of $\mathcal{P}'$ and $D$ its opposite, similarly the block $C$ contains the corner $q_m$ and $B$ its opposite in $\mathcal{P}'$ (see Figure~\ref{fig:fp-patt1}). Recall also that $C$ and $B$ are $i$ neighbors, because of the previous point   there must be a border of type $i$ that slices $\mathcal{P}'$ in two parts. A similar argument can be applied to $A$ and $D$, thus in $\mathcal{P}'$ there must also be a border of type $j$ that cuts $\mathcal{P}'$ in two which is not possible as this leads to two intersecting borders. Since peeling operation can not create intersecting borders, $\mathcal{P}$ contains also intersecting borders, which leads to a contradiction.
              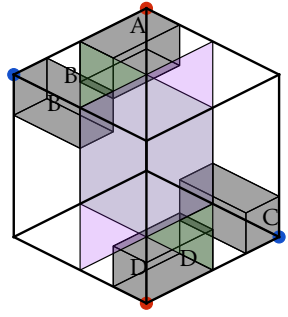
\begin{figure}[!htb]
                  \centering
                  \begin{minipage}{0.4\textwidth}\center{\resizebox{0.9\textwidth}{!}{\input{figs/fp-bis-patt1.pgf}}}\end{minipage}
                  \caption{Floorplan configuration associated with \sbaxpa.}
                  \label{fig:fp-patt1}
              \end{figure}

        \item The $d$-permutation $\phi(\mathcal{P})$ contains $\perm{312}{213}$.
              There are thus three blocks $A,B,C$ such that $$A
                  \twoheadiagarrow^{q_k} B \twoheadiagarrow^{q_k}   C,$$  $$B
                  \twoheadiagarrow^{q_l} C \twoheadiagarrow^{q_l}   A,$$ $$B
                  \twoheadiagarrow^{q_m} A \twoheadiagarrow^{q_m}   C.$$ Without loss of generality we assume that $k=0$, $l=1$ and $m=2$.
              Let us start with an argument for a $4-$floorplan. First, remove all the blocks preceding $A$ and following $C$ in the peeling order with respect to  $q_0$, using the block deletion with respect to $q_0$ and its opposite corner in the floorplan. Do the same operation for blocks preceding $B$ and following $A$ in the peeling order with respect to $q_1$ (using now $q_1$ and its opposite corner for the block deletions). Do it also for the blocks preceding $B$ and following $C$ in the peeling order with respect to $q_2$ (using $q_2$ and its opposite corner). Call the resulting floorplan $\mathcal{P}'$.
              Now suppose one deletes the block $B$ and all the blocks
              preceding $A$ in the peeling order with respect to $q_2$ and call the resulting floorplan $\mathcal{P}''$. In
              $\mathcal{P}''$, the block $A$ contains thus the three corners
              $q_0$, $\overline{q}_1$ and $q_2$. This implies that in this floorplan
              $A$ must contain any corner whose position is of the form
              $p=(*,*,0,*)$ where the star symbol can either be a zero or a
              one. For any such corner one must have in $\mP$ (by remark
              ~\ref{orderdel}) that $A \twoheadiagarrow^{p} C$.
              Similarly one can perform the same deletion procedure for the
              block $A$ resp. $C$ and all the blocks coming after $C$ resp. $B$
              in the peeling orders with respect to $q_1$ resp. $q_0$ to find
              that for any corner of the form $p'=(1,*,*,*)$ resp.
              $p"=(*,*,*,1)$ one must have $C \twoheadiagarrow^{p'} B$ resp. $B
                  \twoheadiagarrow^{p"} A$ in $\mP$. Thus, for the corner $p^*= (1,0,0,1)$ (which fits the three types of
              corner described above) one must have $A\twoheadiagarrow^{p^*}
                  C \twoheadiagarrow^{p^*} B \twoheadiagarrow^{p^*} A$
              which leads to a contradiction and concludes the proof in that case.

              Let us comment on this argument. In dimension $4$, the coordinates of the corners that are contained by one of the blocks $A$, $B$ or $C$ is given by $3$ free coordinates and one fixed coordinates. The fixed coordinate is found by looking at the position of the three canonical corners ($q_1$ $q_2$ $q_3$ and their opposites $\bar{q_1}$ $\bar{q_2}$ $\bar{q_3}$) touched by the corresponding block. This coordinate is given by the one coordinate matching for the three touched corners. Additionally, from Definition~\ref{def:deletion_order}
              it follows
              that in any dimension and for three canonical corners, there are ${1/4}^{th}$ of the coordinates that match and the rest that differs. For example, in dimension $4$ and for the block A in the previous argument, the fixed coordinate was the $z$ coordinate which is the only one with the same value in the position of the three corners $q_1$, $\bar{q}_2$ and $q_3$.

              The previous argument can be generalised to higher dimensions but one cannot easily express the position of a corner where the contradictions occurs. However, it is possible to prove that there always exists such a corner. It suffices to show that the fixed coordinates of the position of the corners contained by a block that also contains the corners $\bar{q_0}$, $q_1$, $q_2$ resp. $q_0$, $\bar{q_1}$, $q_2$ resp.  $q_0$, $q_1$, $\bar{q_2}$ are all different. The position of the corners where the contradictions occurs is then given by fixing the coordinates to their values fixed in each cases and letting free the other ones. For any two corners $q_A$ and $q_B$, the pairs $q_1A- q_B$ and  $\bar{q_A} - q_B$ are contained in one of three sets (different between the two pairs). Thus, by definition of the canonical set of corners (Definition~\ref{def:deletion_order}) the fixed coordinates cannot overlap between the three sets of corners and there exist a corner $p^*$ such that
              $A\twoheadiagarrow^{p^*}
                  C \twoheadiagarrow^{p^*} B \twoheadiagarrow^{p^*} A$
              which leads to a contradiction and concludes the proof.
    \end{itemize}

\end{proof}

\begin{lem}
    The mapping $\phi$ from $2^{d-1}$-floorplans to $d$-permutations is injective.
    \label{lem:injection}
\end{lem}

\begin{proof}

    This is a direct consequence of Lemmas \ref{lem:cover} and \ref{lem:recover}. Two $2^{d-1}$-floorplans do not have the same direction relations of their blocks. Thus they also do not have the same peeling orders, which implies that no two $2^{d-1}$-floorplans produce the same $d-$permutation.
\end{proof}

In the mapping $\phi$ from $2^{d-1}$-floorplans to $d$-permutations, we associate each block of a $2^{d-1}$-floorplan to a point of a $d-$permutation. Given a block $A$, each label of this block in a canonical peeling order gives a coordinate of the corresponding point $p_A$ in the $d-$permutation. Here, this association is chosen such that the peeling order of the blocks  with respect to the canonical corner $q_i$ gives the $x_i$ coordinates of the $d-$permutation points. This association between peeling orders and coordinates of points also corresponds to an \emph{axis-direction association}.

Given two points $p_A$ and $p_B$ of a $d-$permutation $\bpi=(\pi_1, \ldots,\pi_d)$. The relation $\pi_i(p_A) < \pi_i(p_B)$ can be written as the union of $2^{d-1}$ partial order relations of these points (defined by the different possible directions between two points; this is a direct consequence of the definition of the partial orders). This construction is similar to the one between the canonical peeling orders and the direction relations of the blocks of $2^{d-1}$-floorplans. To any axis $j$, there is a corresponding positive direction ${\bf dir}$ such that, for any $2^{d-1}$-floorplan $\mathcal{P}$, the partial order $\overset{j}{\twoheadleftarrow}$ on the blocks of $\mathcal{P}$ is the same as the partial order $<_{\bf dir}$ on the points of $\phi(\mathcal{P})$. We say that $j$ is the \emph{associated axis} of ${\bf dir}$ which we denote by ${ ax}({\bf dir})$. 

Given a positive direction ${\bf dir}:=(\text{dir}_0,\ldots,\text{dir}_{d-1})$ and a canonical corner $q_i$, we call its \emph{signed canonical corner} the corner $q_i^{\text{dir}_i}$ such that $q_i^{+1}=q_i$ and  $q_i^{-1}=\overline{q_i}$. The associated axis {ax}$({\bf dir})$ is defined by the condition  
\begin{equation*}
\overset{{ ax}({\bf dir})}{\twoheadleftarrow}=  \twoheadiagarrow^{q_0^{\text{dir}_0}}  \cap \ldots \cap \twoheadiagarrow^{q_{d-1}^{\text{dir}_{d-1}}} \,.
\end{equation*}
As explained above, this axis is defined such that if $A\overset{{ ax}({\bf dir})}{\twoheadleftarrow} B$ in a $2^{d-1}$-floorplan $\mP$, one has $p_A <_{\bf dir} p_B$ in $\phi(\mP)$.

Using these associated axes, We can define an algorithm that realises a $2^{d-1}$-floorplan $\mathcal{P}(\psi(\bpi))$, from the set of $2^{d-1}$ partial orders $\psi(\bpi)$ obtained from a $d$-permutation $\bpi \in F^{d-1}_n$ (Alg.~\ref{alg:BP2FP}). This algorithm is a
direct generalisation of Algorithm BP2FP in~\cite{ackerman2006bijection}. We denote here, in the $d$-permutation $\bpi$, the point whose first coordinate is equal to $i$ as $p_i$.

\begin{algorithm}
    \caption{$d$-perm2FP \\ \textbf{Input}: A $d$-permutation $\bpi$ that
        belongs to $F^{d-1}_n$   \\ \textbf{Output}: A $2^{d-1}-$floorplan with $n$
        blocks.}  \label{alg:cap}
    \begin{algorithmic}[1]
        \State Setup a list \textit{saillant}[$j$] \textbf{for} each axis $j$;
        \State Setup a list \textit{blocks to push};
        \State Create a block of length $n$ in each direction, name it $1$ and add it to \textit{saillant}[$j$] for all axes $j$;
        \For{$i=2,\ldots n$}
        \State $f= \textbf{dir}\left(p_{i-1}, p_i \right)$;
        \For{$k \in \textit{saillant}[{\bf ax}(f)]$}
        \If{$ \textbf{dir}\left(p_k, p_i \right)=f$}
        \State Add the block $k$ to \textit{blocks to push};
        \EndIf
        \EndFor
        \State Create a new block called $i$ from the corner $\bar{q}_0$ by pushing for $n-i$ units along the axis {\bf ax}(f) all the blocks in \textit{blocks to push};
        \State Remove all the blocks in \textit{blocks to push} from \textit{saillant}[{\bf ax}(f)];
        \State Remove all blocks from \textit{blocks to push};
        \State Add $i$ to \textit{saillant}[j] \textbf{for} each axis $j$;
        \EndFor
    \end{algorithmic}
    \label{alg:BP2FP}
\end{algorithm}

\begin{figure}[!htb]
    \center{\minipdf{0.2}{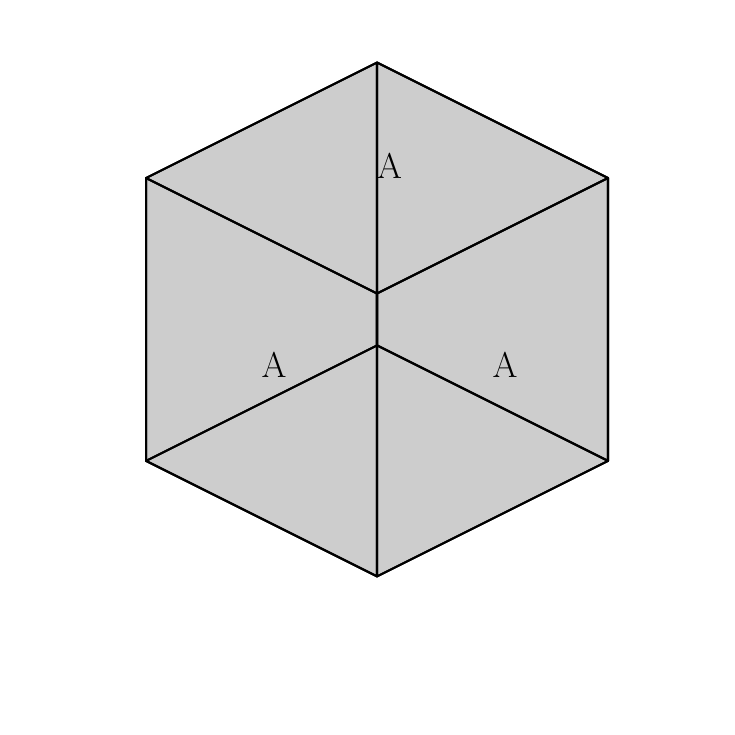}\minipdf{0.2}{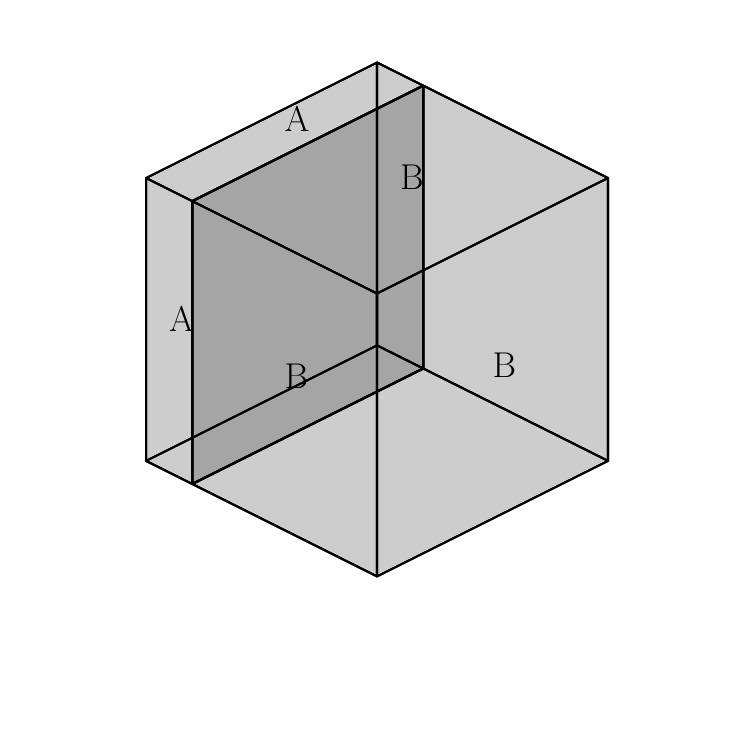}\minipdf{0.2}{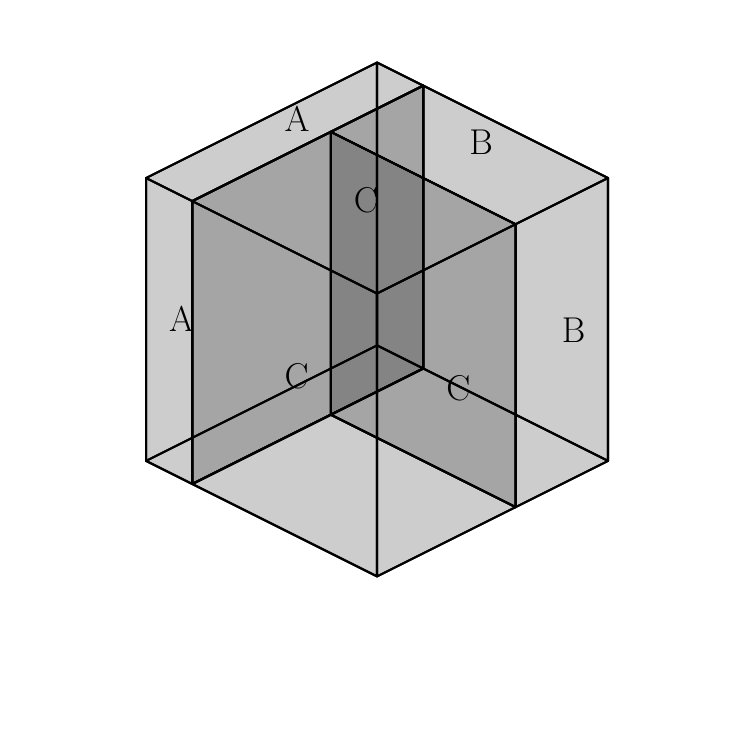}\minipdf{0.2}{figs/fp-prune-4.pdf}\minipdf{0.2}{figs/fp-prune-5.pdf}}
    \caption{The execution of Algorithm~\ref{alg:BP2FP} of the 3-permutation $\bpi=\perm{12435}{34125}$.}
    \label{fig:BP2FP}
\end{figure}

\begin{lem}
    Let ${\bf \pi}$ be a $d$-permutation in $F^{d-1}_n$. The object $\mP({\bf \pi})$ obtained by applying the algorithm $d$-perm2FP on ${\bf \pi}$  is a $2^{d-1}$-floorplan with $n$ blocks such that 
    any block in the output $2^{d-1}$-floorplan corresponds to a point $p_A$ in the input $d-$permutation.
    
    Additionally, for two blocks $A$ and $B$  in $\mP({\bf \pi})$, corresponding to two points $p_A$ and $p_B$ such that ${\bf dir}(p_A,p_B)$ is a positive direction, one has $A  \overset{{ax}({\bf dir}(p_A,p_B))}{\twoheadleftarrow} B$.
    \label{lem:dir}
\end{lem}

\begin{proof}
    We prove this by induction on the inserted blocks. Let $\mP_{k}$ be the object obtained after the $(k)^{th}$ step of algorithm $d$-perm2FP. We suppose that: \begin{itemize}
        \item  $\mP_k$ is a $2^{d-1}$-floorplan obtained by performing a block insertion in $\mP_{k-1}$.
        \item  {\it saillant}[l] corresponds to the list of blocks in $\mP_k$ with a facet lying in the upper boundary of axis $l$ that we call $f_l$.
        \item The axis $j$ used to perform the block insertion in $\mP_{k-1}$ is defined as ${ ax}({\bf dir}(p_{k-1},p_k))$
        \item Let $D$ be the block named $i_{min}$, for which $i_{min}$ is the smallest $i$ in {\it saillant}$[j]$ such that $\textbf{dir}\left(p_{i}, p_k \right)= \textbf{dir}\left(p_{k-1}, p_k \right)$. This block generates a pushable facet $f_{push}$ in $\mP_{k-1}$. This pushable facet is the one used to perform the block insertion of the block named $k$.
        \item For two blocks $i$ and $m$ in $\mP_k$, one has $i  \overset{{ax}({\bf dir}(p_i,p_m))}{\twoheadleftarrow} m$.
    \end{itemize}
    
    At step $k=2$, this is clearly the case. Let us now prove that this is true at the step $k+1$. This amounts to prove that:

    \begin{itemize}
        \item The list {\it blocks to push} corresponds to the list of blocks that have a facet lying in $f_{push}$. This implies that, pushing the blocks in {\it blocks to push} from the boundary of $\mP_{k}$, and drawing $k+1$ in the space created, is equivalent to performing a block insertion using $f_{push}$.   
        \item At the end of the step $k+1$, the list {\it saillant}[l] corresponds to the list of blocks in $\mP_{k+1}$ with a facet lying in the upper boundary of axis $l$.
        \item For any block $w<k+1$, one has $w  \overset{{ ax}({\bf dir}(p_i,p_k))}{\twoheadleftarrow} k$. 
    \end{itemize}

    Let us denote ${ ax}({\bf dir}(p_{k},p_k+1))$ as $j$ and the upper boundary of axis $j$ in $\mP_{k}$ as $f_j$.  Let also $\pi_{v}(p_L)$ be the coordinate $v$ of the point $p_L$ in ${\bf \pi}$,
    
    {\bf Proof of the first item:}
    In order to prove the first point, we first prove that the list {\it blocks to push} corresponds to the end of the list {\it saillant}$[j]$, whose first element is $D$. In other words, we prove that any block $C$ in {\it saillant}$[j]$ with  $i_{min}<C$ is in {\it blocks to push}.
    
    We prove this by contradiction, suppose this is not the case, there are then blocks in {\it saillant}$[j]$ called $C_1 C_2 \ldots C_q C$ such that $i_{min} < C_1 <\ldots C_q <C$ ($q\geq 0$) such that $\textbf{dir}(p_{C_1},p_{k+1})\neq \textbf{dir}(p_{k},p_{k+1})$ \ldots $\textbf{dir}(p_{C},p_{k+1})\neq \textbf{dir}(p_{k},p_{k+1})$.
    Let us consider the block  $C+1$, there are two possibilities, either $C+1$ is in {\it saillant}$[j]$ or not. If it is one has $\textbf{dir}(p_{D},p_{k+1})=\textbf{dir}(p_{C+1},p_{k+1})$.  If it is not, let $B$ be a block in {\it saillant}$[j]$ that is at the end of the chain of blocks that links $C+1$ and $f_j$ (i.e a block whose facet is contained in $f_j$ such that $C+1\overset{j}{\twoheadleftarrow}B$). 
    
    Let $p_{k+1}$, $p_B$,$p_{c+1}$, $p_C$ and $p_D$ be the points of ${\bf \pi}$ associated to the block $k+1,\,B,\,C+1,\,C$ and $D$. Since $B$ and $D$ are in {\it blocks to push} ($B$ is in this list because $C<B$), one has $\textbf{dir}(p_{B},p_{k+1})= \textbf{dir}(p_{D},p_{k+1})$. Additionally, since $B\overset{j}{\twoheadleftarrow}C+1$ one has $\textbf{dir}(p_{C+1},p_{B})=\textbf{dir}(p_{B},p_{k+1})$ which implies by transitivity $\textbf{dir}(p_{C+1},p_{k+1})= \textbf{dir}(p_{B},p_{k+1})$. One thus has $\textbf{dir}(p_{D},p_{k+1})=\textbf{dir}(p_{C+1},p_{k+1})$. Finally, since $D$ was in {\it saillant}$[j]$ at both steps $C+1$ and $k+1$, one has $\textbf{dir}(p_{D},p_{k+1}) \neq \textbf{dir}(p_{D},p_{C+1})$. By Remark \ref{rem:diford}, the previous statements translate as:
    \begin{align*}
        &\pi_{0}(p_D) < \; \pi_{0}(p_C) < \: \pi_{0}(p_{C+1}) < \: \pi_{0}(p_{k+1}) ,
        \end{align*}
        \begin{align*}
         \pi_{l}(p_{C+1}) < \;  \pi_{l}(p_D) <  \; &\pi_{l}(p_{k+1}) \quad \text{ or } \quad  \pi_{l}(p_{k+1}) < \; \pi_{l}(p_D) < \; \pi_{l}(p_{C+1}) \,\text{ for at least one $l$},
     \end{align*}
      \begin{align*}   
         \pi_{m}(p_D) < \;  \pi_{m}(p_{C+1}) < \;  \pi_{m}(p_{k+1}) \quad &\text{ or } \quad  \pi_{m}(p_{C+1}) < \;  \pi_{m}(p_D) < \;  \pi_{m}(p_{k+1})  \\ & \text{ or } \qquad \qquad \qquad  \qquad  \qquad \qquad \qquad \text{ for all $m\neq l \neq0$.}\\
         \pi_{m}(p_{k+1}) < \; \pi_{m}(p_{C+1}) < \; \pi_{m}(p_D) \quad &\text{ or } \quad \pi_{m}(p_{k+1}) < \; \pi_{m}(p_D) < \:  \pi_{m}(p_{C+1}) \,
    \end{align*}
    Here, one additionally has $\pi_{0}(p_D)=i_{min}$, $\pi_{0}(p_C)=C$, $\pi_{0}(p_{C+1})=C+1$ and $\pi_{0}(p_{k+1})=k+1$. Since $C$ is not in {\it blocks to push}, one has that   $\textbf{dir}(p_{C},p_{k+1}) \neq \textbf{dir}(p_{B},p_{k+1}) = \textbf{dir}(p_{D},p_{k+1})$. By Remark \ref{rem:diford},  There must thus be at least one coordinate $v$ such that  
        \begin{align*}
         \pi_{v}(p_{C+1}),\pi_{v}(p_D) < \;  \pi_{v}(p_{k+1}) <  \; &\pi_{v}(p_C) \quad \text{ or } \quad \pi_{v}(p_C) < \; \pi_{v}(p_{k+1}) < \;  \pi_{v}(p_{C+1}),\pi_{v}(p_D) \,.
     \end{align*}
    If $v=l$, one has 
    \begin{align*}
        \pi_{0}(p_D) < \; \pi_{0}(p_C) < \,& \pi_{0}(p_{C+1}) < \: \pi_{0}(p_{k+1}) , \\
        \pi_{l}(p_{C+1}) < \;  \pi_{l}(p_D) <  \; \pi_{l}(p_{k+1})  <  \; \pi_{l}(p_C) \quad &\text{ or } \quad    \pi_{l}(p_C) < \; \pi_{l}(p_{k+1}) < \; \pi_{l}(p_D) < \; \pi_{l}(p_{C+1})\,,
        \end{align*}
    which are occurrences of the forbidden patterns in  $\vinpat{2413}{2}$. As proven in \cite{bonichon2022baxter}, any occurrence of $\vinpat{2413}{2}$ is also an occurrence of $\sbaxpa$. In the $2$-dimensional case, this is the only possibility and any block in {\it saillant}$[j]$ generates a pushable facet. In this case the proof of the first point is thus completed. In arbitrary dimensions, one has to consider more configurations.
    
    If $v \neq l$  there are then six possibilities for $\pi_l$:
    \begin{align*}
        \text{1.} \quad  \pi_{l}(p_{C+1}) < \;  \pi_{l}(p_D) <  \; \pi_{l}(p_C) <  \;\pi_{l}(p_{k+1}) \quad &\text{ or } \quad  \text{2.} \quad \pi_{l}(p_{C+1}) <  \; \pi_{l}(p_C) < \;  \pi_{l}(p_D)  <  \;\pi_{l}(p_{k+1}) \\ 
        \text{3.} \quad \pi_{l}(p_C) < \; \pi_{l}(p_{C+1}) <  \;   \pi_{l}(p_D)  <  \;\pi_{l}(p_{k+1}) \quad &\text{ or } \quad \text{4.} \quad \pi_{l}(p_{k+1}) < \; \pi_{l}(p_C) < \; \pi_{l}(p_D) < \; \pi_{l}(p_{C+1}) \\ 
        \text{5.} \quad \pi_{l}(p_{k+1})  < \; \pi_{l}(p_D) < \; \pi_{l}(p_C) < \; \pi_{l}(p_{C+1}) \quad &\text{ or } \quad   \text{6.} \quad \pi_{l}(p_{k+1}) < \; \pi_{l}(p_D) < \; \pi_{l}(p_{C+1})  < \; \pi_{l}(p_C)\,.
    \end{align*}
    Additionally, there are two possibilities for $\pi_v$:
    \begin{align*}
        \text{a.} \quad  \pi_{v}(p_D) < \;  \pi_{v}(p_{C+1}) <  \; \pi_{v}(p_{k+1}) <  \;\pi_{v}(p_C)  &, \\   \text{b.} \quad \pi_{v}(p_C) <  \; \pi_{v}(p_{k+1}) < \;  \pi_{v}(p_{C+1})  <  \;\pi_{v}(p_D) \,&.
    \end{align*}
    For the configurations: $1. a.$, $1. b.$, $2. a.$, $2. b.$, $4. a.$, $4. b.$, $5. a.$, $5. b.$ ; the configuration of the points $p_C,p_{C+1}$ and $p_{k+1}$ corresponds to occurrences of forbidden patterns in $\Sym{\perm{312}{213})}$. For the other configurations, the points   $p_C,p_D$ and $p_{k+1}$ leads also to occurrence of the same patterns. This proves that any block $C$ in {\it saillant}$[j]$ with $i_{min}<C$ is in {\it blocks to push}.
    
    To complete the first point, it thus remain to prove that a corner of $D$ generates the pushable facet  $f_{push}$ and that the list {\it blocks to push} corresponds to the blocks in $\mP_k$ with a facet contained in $f_{push}$. 
    
    At step $i_{min}$, the minimal corner of $D$ generates $f_{push}$. Let us prove that this is also the case at step $k+1$. The only possibility for this to not be the case is that, at some step $E$, the list {\it blocks to push} corresponds to a pushable facet whose generating corner shadows the minimal corner of $D$. Again, we prove that this is not possible by contradiction. 
    
    Since $D<E$, as for $C+1$ before, it is either in {\it saillant}$[j]$ at step $k+1$ or not. In both cases one has $\textbf{dir}(p_{D},p_{k+1})= \textbf{dir}(p_{E},p_{k+1})$. As for $C+1$ and $D$ before, one must have $\textbf{dir}(p_{D},p_{k+1})\neq  \textbf{dir}(p_{D},p_{E})$. These statements leads to the conditions:
    \begin{align*}
        & \pi_{0}(p_D)  < \; \pi_{0}(p_E) < \: \pi_{0}(p_{k+1}) ,
        \end{align*}
        \begin{align*}
         \pi_{l}(p_E) < \;  \pi_{l}(p_D) <  \; &\pi_{l}(p_{k+1}) \quad \text{ or } \quad  \pi_{l}(p_{k+1}) < \; \pi_{l}(p_D) < \; \pi_{l}(p_E) \,\text{ for at least one $l$},
     \end{align*}
      \begin{align*}   
         \pi_{m}(p_D) < \;  \pi_{m}(p_E) < \;  \pi_{m}(p_{k+1}) \quad &\text{ or } \quad  \pi_{m}(p_E) < \;  \pi_{m}(p_D) < \;  \pi_{m}(p_{k+1})  \\ & \text{ or } \qquad \qquad \qquad  \qquad  \qquad \qquad \qquad \text{ for all $j\neq i \neq0$.}\\
         \pi_{m}(p_{k+1}) < \; \pi_{m}(p_E) < \; \pi_{m}(p_D) \quad &\text{ or } \quad \pi_{m}(p_{k+1}) < \; \pi_{m}(p_D) < \:  \pi_{m}(p_E) \,
    \end{align*}

    Because at step $E$ we insert a block using a generating corner that shadows the minimal corner of $D$, there must be a block $F<D$ in {\it saillant}$[j]$, such that $\textbf{dir}(p_{F},p_{D}) \neq \textbf{dir}(p_{F},p_{E})$. Because $F<D$, the block $F$ is not in {\it blocks to push}, thus $\textbf{dir}(p_{F},p_{k+1}) \neq \textbf{dir}(p_{D},p_{k+1})$. This implies that for a coordinate $l$ and a coordinate $v$
    \begin{align*}
        & \pi_{0}(p_F)  < \; \pi_{0}(p_D)  < \; \pi_{0}(p_E) < \: \pi_{0}(p_{k+1}) ,
        \end{align*}
    \begin{align*}   
        \text{1.} \quad \pi_{l}(p_D) < \; \pi_{l}(p_F) < \; \pi_{l}(p_E) < \;  \pi_{l}(p_{k+1}) \quad &\text{ or } \quad \text{2.} \quad \pi_{l}(p_E) < \;\pi_{l}(p_F) < \;  \pi_{l}(p_D) < \;  \pi_{l}(p_{k+1}) \\
        \text{3.} \quad \pi_{l}(p_{k+1}) < \; \pi_{l}(p_E) <\pi_{l}(p_F) < \; \; \pi_{l}(p_D) \quad &\text{ or } \quad \text{4.} \quad\pi_{l}(p_{k+1}) < \; \pi_{l}(p_D) < \:\pi_{l}(p_F) < \;  \pi_{l}(p_E) \,
    \end{align*}
    
    \begin{align*}   
        \text{a.} \quad \pi_{v}(p_D) < \; \pi_{v}(p_E) < \;  \pi_{v}(p_{k+1}) < \; \pi_{v}(p_F)  \quad &\text{ or } \quad  \text{b.} \quad\pi_{v}(p_E)  < \;  \pi_{v}(p_D) < \;  \pi_{v}(p_{k+1}) < \; \pi_{v}(p_F)\\ \text{c.} \quad
        \pi_{v}(p_F) < \; \pi_{v}(p_{k+1}) < \; \pi_{v}(p_E) < \; \pi_{v}(p_D) \quad &\text{ or } \quad \text{d.} \quad \pi_{v}(p_F) < \; \pi_{v}(p_{k+1}) < \; \pi_{k}(p_D) < \;  \pi_{k}(p_E) \,
    \end{align*}
    For the configurations $1.a.$, $1.b.$, $1. c.$, $1. d.$, $3. a.$, $3. b.$, $3. c.$, $3. d.$, the points $p_D$, $p_F$, $p_{k+1}$ are occurrences of forbidden patterns in $\Sym{\perm{312}{213})}$. Similarly for the other configurations, the points  $p_E$, $p_F$, $p_{k+1}$ are also occurrences of patterns in $\Sym{\perm{312}{213})}$. It proves that at step $k+1$ the minimal corner of $D$ generates $f_{push}$.
    
    From the discussion above, the blocks in {\it blocks to push} are the end of the list {\it saillant}$[j]$, starting at $D$. Additionally, at each step $i_{min}<l<k+1$, it has been proven that there was no block insertion performed by the algorithm, that uses a generating corner which shadows the minimal corner of $D$. Thus at each of these steps, the upper facet of axis $j$ of the block inserted lies within $f_{push}$. All the block in {\it blocks to push} have thus their minimal facet of axis $j$ contained in $f_{push}$. Finally,  for a block facet to be included in $f_{push}$ it must be a facet of a block $D'$ with $D<D'$ (a block facet can be contained in a pushable facet if and only if it is inserted after the block whose minimal corner generates the pushable facet). Thus there cannot be a block with a facet included $f_{push}$ that is not in {\it blocks to push}. 
    
    This concludes the proof of the first item.
    
    {\bf Proof of the second item:}
     At the beginning of the step $k+1$, for any axis $l$, {\it saillant}$[l]$ corresponds  to the list of blocks in $\mP_k$ with a facet lying in the upper boundary of axis $l$. At the end  of the step $k+1$, the algorithm removes from  {\it saillant}$[j]$ all the blocks in {\it blocks to push}. These blocks correspond to the blocks pushed when $k+1$ is inserted in $\mP_{k}$. It then adds in  all {\it saillant}$[l]$ the block $k+1$. Thus, these updates in {\it saillant}$[l]$, for any axis $l$, correspond exactly to the updates of the list of the blocks with a facet lying in the upper boundary of axis $l$, after the block insertion of $k+1$ in $\mP_k$.

    {\bf Proof of the third item:}
    Let us now consider a block $A$ such that $A<k+1$. There are two possibilities:
    \begin{itemize}
        \item $ {\bf dir}(p_{A},p_{k+1}))=  {\bf dir}(p_{k},p_{k+1}))$ and $A$ is in  {\it saillant}$[j]$
        \item $ {\bf dir}(p_{A},p_{k+1}))=  {\bf dir}(p_{k},p_{k+1}))$ but $A$ is not in {\it saillant}$[j]$
        \item  $ {\bf dir}(p_{A},p_{k+1})) \neq  {\bf dir}(p_{k},p_{k+1}))$
    \end{itemize}

    In the first case, it is clear that $A \overset{{ ax}({\bf dir}(p_{A},p_{k+1}))}{\twoheadleftarrow} k+1  $. 
    
    In the second case, there is a block $C$ such that $A<C<k+1$, $\textbf{dir}(p_C,p_{k+1})=\textbf{dir}(p_k,p_{k+1})$ and $\textbf{dir}(p_A,p_{C})=\textbf{dir}(p_k,p_{k+1})$. From the first case and by the induction hypothesis one has  $A \overset{{ ax}({\bf dir}(p_{k},p_{k+1}))}{\twoheadleftarrow} C  \overset{{ ax}({\bf dir}(p_{k},p_{k+1}))}{\twoheadleftarrow} k+1$. By transitivity, one thus has $A \overset{{ ax}({\bf dir}(p_{k},p_{k+1}))}{\twoheadleftarrow} k+1$ which is equivalent to $A \overset{{ ax}({\bf dir}(p_{A},p_{k+1}))}{\twoheadleftarrow} k+1$. 
    
    In the third case, let us reconsider the block $D$. For any block $C$ such that $C<D<k+1$ and ${\bf dir}(p_C,p_{k+1}) \neq {\bf dir}(p_k,p_{k+1})$, one has by the induction hypothesis and by the definition of $D$ (it is the block that generates the pushable facet $f_{push}$) that $A \overset{{ ax}({\bf dir}(p_{A},p_{D}))}{\twoheadleftarrow} D$ and $A \overset{{ ax}({\bf dir}(p_{A},p_{D}))}{\twoheadleftarrow} k+1$. One must thus show that $\textbf{dir}(p_A,p_{D}) = \textbf{dir}(p_A,p_{k+1})$ if $\textbf{dir}(p_A,p_{k+1}) \neq \textbf{dir}(p_D,p_{k+1})$. Suppose this is not the case, one must then have in the permutation $\bpi$ the following conditions for at least one $l$ and one $m$:
    \begin{align*}
        \pi_0(p_A) <                        & \pi_0(p_D) < \pi_0(p_{k+1}) \;,                         \\
        \pi_l(p_A) <\pi_i(p_{k+1}) < \pi_i(p_D) & \text{ or } \pi_i(p_D) <\pi_i(p_{k+1}) < \pi_i(p_A) \;, \\
        \pi_m(p_{k+1}) <\pi_i(p_A) < \pi_i(p_D)  & \text{ or } \pi_i(p_{D}) <\pi_i(p_A) < \pi_i(p_{k+1}) \;.
    \end{align*}

    These are occurrences of the patterns $ \perm{132}{213}$,$ \perm{132}{231}$,$ \perm{312}{213}$ and $ \perm{312}{231}$ which belong to the forbidden patterns in  Sym$(\perm{312}{213})$.
    
    This concludes the proof of the third item.
\end{proof}

\begin{myproof}{theorem}{\ref{th:final}}
      Using Lemma \ref{lem:dir}, for any $\bpi \in F^{d-1}_n$, the output of algorithm $d$-perm2FP is a $2^{d-1}$-floorplan that we call $\mathcal{P}(\psi(\bpi))$. Using Lemmas~\ref{lem:FP2BP} and \ref{lem:injection} , the mapping $\phi$ is injective and maps $2^{d-1}$-floorplans to $d-$permutations in $F^{d-1}_n$, The bijection is thus proven by showing that $\phi(\mathcal{P}(\psi(\bpi)) = \bpi$.
    
    Let $\bpi=(\pi_1, \ldots, \pi_{d-1})$ be a $d$-permutation
    in $F^{d-1}_n$. By Lemma \ref{lem:dir}, one has that the direction relation of any two blocks $A$ and $B$ in $\mathcal{P}(\psi(\bpi))$ is given by the associated axis of the direction of their corresponding points in $\bpi$. Recall that in $\mathcal{P}(\psi(\bpi))$, the blocks are labeled by the first coordinate of their corresponding point in $\bpi$.  By the definition of the associated axes and of Algorithm \ref{alg:BP2FP}, one has that:
    \begin{itemize}
        \item The peeling order with respect to the canonical corner $q_0$ is given by $1, 2, \ldots, n$.
        \item The peeling orders of the blocks with respect to the canonical corner $q_i$ (for $d-1>i\geq1$) in $\mathcal{P}(\psi(\bpi))$ is the total order defined by $\pi_i$.
    \end{itemize}
    Finally, using definition \ref{def:FP2BP}, one has: $\phi(\mathcal{P}(\psi(\bpi)) = \bpi$ 
\end{myproof}

\section*{Concluding remarks}

Several perspectives for future work can be exhibited. First, as the rewriting rule is 
more involved for $d \geq 2$, deriving an equation for the generating function
of $d$-floorplans is a highly non trivial open problem, even for the
$3-$dimensional case.
Moreover, in order
to define equivalences between $d$-floorplans, we considered here relative orders of boxes. Equivalence based on borders have been studied in~\cite{asinowski2013orders} showing that these classes of floorplans are in bijection with anti-Baxter permutations (permutations avoiding  $\Sym(\vinpat{2143}{2})$). It would be interesting to generalize the results of~\cite{asinowski2013orders} to $d$-floorplans.

\bibliographystyle{alpha}
\bibliography{sample}

\end{document}

%% file: preambule.tex
\usepackage[utf8]{inputenc}
\usepackage{amsmath}
\usepackage{amsthm}
\usepackage{amssymb}
\usepackage{mathtools}
\usepackage{mathrsfs}
\usepackage{xcolor}
\usepackage{tikz}
\usepackage{pgf}

\makeatletter\@ifpackageloaded{underscore}{}{\usepackage[strings]{underscore}}\makeatother
\newtheorem{thm}{Theorem}[section]
\newtheorem{lem}[thm]{Lemma}

\newtheorem{cor}[thm]{Corollary}

\definecolor{axeZ}{RGB}{20,81,204}
\definecolor{axeY}{RGB}{212,43,11}
\definecolor{axeX}{rgb}{0.33, 0.42, 0.18}
\definecolor{mygreen}{rgb}{0.20  0.43 0.22} 
\definecolor{mypurple}{rgb}{0.37 0 0.51}  

\definecolor{axisX}{rgb}{0.281, 0.722, 0.216}
\definecolor{axisY}{rgb}{0.613, 0.125, 0.961}
\definecolor{axisZ}{rgb}{0.785, 0.527, 0.047}
\definecolor{axisT}{rgb}{0.082, 0.402, 0.539}

\newcommand\perm[2]{\ensuremath{({\color{axeZ}#1}, {\color{axeY}#2})}}
\newcommand\bsig{{\ensuremath{\boldsymbol{\sigma}}}}

\newcommand\mP{{\ensuremath{\mathcal{P}}}}

\newcommand\bpi{{\boldsymbol{\pi}}}

\DeclareMathOperator{\Sym}{Sym}

\newcommand\twoheadiagarrow{\mathrel{\rotatebox{135}{$\twoheadrightarrow$}}}

\newcommand\vinpat[2]{\ensuremath{#1|_{\color{axeX}#2}}}
\newcommand\vinpatb[3]{\ensuremath{#1|_{{\color{axeX}#2},{\color{axeY}#3}}}}
\newcommand\vinpatd[5]{\ensuremath{\vinpat{\perm{#1}{#2}}{{{\color{axeX}#3},{\color{axeY}#4},
					{\color{axeZ}#5}}}}}

\newcommand\baxpa{\ensuremath{\vinpatb{2413}{2}{2}}}
\newcommand\baxpb{\ensuremath{\vinpatd{312}{213}{1}{2}{.}}}
\newcommand\baxpc{\ensuremath{\vinpatd{3412}{1432}{2}{2}{.}}}
\newcommand\baxpd{\ensuremath{\vinpatd{2143}{1423}{2}{2}{.}}}

\newcommand\sbaxpa{\ensuremath{\Sym(\baxpa)}}
\newcommand\sbaxpb{\ensuremath{\Sym(\baxpb)}}
\newcommand\sbaxpc{\ensuremath{\Sym(\baxpc)}}
\newcommand\sbaxpd{\ensuremath{\Sym(\baxpd)}}

\newcommand\vectt[4]{\ensuremath{{\color{#1}\begin{smallmatrix}#2\\#3\\#4\end{smallmatrix}}}}

\newcommand\vecttx[3]{\vectt{axisX}{#1}{#2}{#3}}
\newcommand\vectty[3]{\vectt{axisY}{#1}{#2}{#3}}
\newcommand\vecttz[3]{\vectt{axisZ}{#1}{#2}{#3}}


%% file: figs/face-adjx.pgf
\begingroup%
\makeatletter%
\begin{pgfpicture}%
\pgfpathrectangle{\pgfpointorigin}{\pgfqpoint{4.900000in}{4.900000in}}%
\pgfusepath{use as bounding box, clip}%
\begin{pgfscope}%
\pgfsetbuttcap%
\pgfsetmiterjoin%
\definecolor{currentfill}{rgb}{1.000000,1.000000,1.000000}%
\pgfsetfillcolor{currentfill}%
\pgfsetlinewidth{0.000000pt}%
\definecolor{currentstroke}{rgb}{1.000000,1.000000,1.000000}%
\pgfsetstrokecolor{currentstroke}%
\pgfsetdash{}{0pt}%
\pgfpathmoveto{\pgfqpoint{0.000000in}{0.000000in}}%
\pgfpathlineto{\pgfqpoint{4.900000in}{0.000000in}}%
\pgfpathlineto{\pgfqpoint{4.900000in}{4.900000in}}%
\pgfpathlineto{\pgfqpoint{0.000000in}{4.900000in}}%
\pgfpathlineto{\pgfqpoint{0.000000in}{0.000000in}}%
\pgfpathclose%
\pgfusepath{fill}%
\end{pgfscope}%
\begin{pgfscope}%
\pgfpathrectangle{\pgfqpoint{0.100000in}{0.100000in}}{\pgfqpoint{4.700000in}{4.700000in}}%
\pgfusepath{clip}%
\pgfsetbuttcap%
\pgfsetmiterjoin%
\definecolor{currentfill}{rgb}{0.000000,0.000000,0.000000}%
\pgfsetfillcolor{currentfill}%
\pgfsetfillopacity{0.300000}%
\pgfsetlinewidth{2.007500pt}%
\definecolor{currentstroke}{rgb}{0.000000,0.000000,0.000000}%
\pgfsetstrokecolor{currentstroke}%
\pgfsetstrokeopacity{0.300000}%
\pgfsetdash{}{0pt}%
\pgfpathmoveto{\pgfqpoint{0.100000in}{0.100000in}}%
\pgfpathlineto{\pgfqpoint{2.450000in}{0.100000in}}%
\pgfpathlineto{\pgfqpoint{2.450000in}{1.666667in}}%
\pgfpathlineto{\pgfqpoint{0.100000in}{1.666667in}}%
\pgfpathlineto{\pgfqpoint{0.100000in}{0.100000in}}%
\pgfpathclose%
\pgfusepath{stroke,fill}%
\end{pgfscope}%
\begin{pgfscope}%
\pgfpathrectangle{\pgfqpoint{0.100000in}{0.100000in}}{\pgfqpoint{4.700000in}{4.700000in}}%
\pgfusepath{clip}%
\pgfsetbuttcap%
\pgfsetmiterjoin%
\pgfsetlinewidth{2.007500pt}%
\definecolor{currentstroke}{rgb}{0.000000,0.000000,0.000000}%
\pgfsetstrokecolor{currentstroke}%
\pgfsetstrokeopacity{0.500000}%
\pgfsetdash{}{0pt}%
\pgfpathmoveto{\pgfqpoint{0.100000in}{0.100000in}}%
\pgfpathlineto{\pgfqpoint{2.450000in}{0.100000in}}%
\pgfpathlineto{\pgfqpoint{2.450000in}{1.666667in}}%
\pgfpathlineto{\pgfqpoint{0.100000in}{1.666667in}}%
\pgfpathlineto{\pgfqpoint{0.100000in}{0.100000in}}%
\pgfpathclose%
\pgfusepath{stroke}%
\end{pgfscope}%
\begin{pgfscope}%
\pgfpathrectangle{\pgfqpoint{0.100000in}{0.100000in}}{\pgfqpoint{4.700000in}{4.700000in}}%
\pgfusepath{clip}%
\pgfsetbuttcap%
\pgfsetmiterjoin%
\definecolor{currentfill}{rgb}{0.000000,0.000000,0.000000}%
\pgfsetfillcolor{currentfill}%
\pgfsetfillopacity{0.300000}%
\pgfsetlinewidth{2.007500pt}%
\definecolor{currentstroke}{rgb}{0.000000,0.000000,0.000000}%
\pgfsetstrokecolor{currentstroke}%
\pgfsetstrokeopacity{0.300000}%
\pgfsetdash{}{0pt}%
\pgfpathmoveto{\pgfqpoint{2.450000in}{3.233333in}}%
\pgfpathlineto{\pgfqpoint{4.800000in}{3.233333in}}%
\pgfpathlineto{\pgfqpoint{4.800000in}{4.800000in}}%
\pgfpathlineto{\pgfqpoint{2.450000in}{4.800000in}}%
\pgfpathlineto{\pgfqpoint{2.450000in}{3.233333in}}%
\pgfpathclose%
\pgfusepath{stroke,fill}%
\end{pgfscope}%
\begin{pgfscope}%
\pgfpathrectangle{\pgfqpoint{0.100000in}{0.100000in}}{\pgfqpoint{4.700000in}{4.700000in}}%
\pgfusepath{clip}%
\pgfsetbuttcap%
\pgfsetmiterjoin%
\pgfsetlinewidth{2.007500pt}%
\definecolor{currentstroke}{rgb}{0.000000,0.000000,0.000000}%
\pgfsetstrokecolor{currentstroke}%
\pgfsetstrokeopacity{0.500000}%
\pgfsetdash{}{0pt}%
\pgfpathmoveto{\pgfqpoint{2.450000in}{3.233333in}}%
\pgfpathlineto{\pgfqpoint{4.800000in}{3.233333in}}%
\pgfpathlineto{\pgfqpoint{4.800000in}{4.800000in}}%
\pgfpathlineto{\pgfqpoint{2.450000in}{4.800000in}}%
\pgfpathlineto{\pgfqpoint{2.450000in}{3.233333in}}%
\pgfpathclose%
\pgfusepath{stroke}%
\end{pgfscope}%
\begin{pgfscope}%
\pgfpathrectangle{\pgfqpoint{0.100000in}{0.100000in}}{\pgfqpoint{4.700000in}{4.700000in}}%
\pgfusepath{clip}%
\pgfsetrectcap%
\pgfsetroundjoin%
\pgfsetlinewidth{3.011250pt}%
\definecolor{currentstroke}{rgb}{0.281250,0.721569,0.215686}%
\pgfsetstrokecolor{currentstroke}%
\pgfsetdash{}{0pt}%
\pgfpathmoveto{\pgfqpoint{2.450000in}{0.100000in}}%
\pgfpathlineto{\pgfqpoint{2.450000in}{4.800000in}}%
\pgfusepath{stroke}%
\end{pgfscope}%
\begin{pgfscope}%
\definecolor{textcolor}{rgb}{0.000000,0.000000,0.000000}%
\pgfsetstrokecolor{textcolor}%
\pgfsetfillcolor{textcolor}%
\pgftext[x=1.275000in,y=0.883333in,left,base]{\color{textcolor}{\rmfamily\fontsize{20.000000}{24.000000}\selectfont\catcode`\^=\active\def^{\ifmmode\sp\else\^{}\fi}\catcode`\%=\active\def
\end{pgfscope}%
\begin{pgfscope}%
\definecolor{textcolor}{rgb}{0.000000,0.000000,0.000000}%
\pgfsetstrokecolor{textcolor}%
\pgfsetfillcolor{textcolor}%
\pgftext[x=3.625000in,y=4.016667in,left,base]{\color{textcolor}{\rmfamily\fontsize{20.000000}{24.000000}\selectfont\catcode`\^=\active\def^{\ifmmode\sp\else\^{}\fi}\catcode`\%=\active\def
\end{pgfscope}%
\end{pgfpicture}%
\makeatother%
\endgroup%

%% file: figs/face-adjy.pgf
\begingroup%
\makeatletter%
\begin{pgfpicture}%
\pgfpathrectangle{\pgfpointorigin}{\pgfqpoint{4.900000in}{4.900000in}}%
\pgfusepath{use as bounding box, clip}%
\begin{pgfscope}%
\pgfsetbuttcap%
\pgfsetmiterjoin%
\definecolor{currentfill}{rgb}{1.000000,1.000000,1.000000}%
\pgfsetfillcolor{currentfill}%
\pgfsetlinewidth{0.000000pt}%
\definecolor{currentstroke}{rgb}{1.000000,1.000000,1.000000}%
\pgfsetstrokecolor{currentstroke}%
\pgfsetdash{}{0pt}%
\pgfpathmoveto{\pgfqpoint{0.000000in}{0.000000in}}%
\pgfpathlineto{\pgfqpoint{4.900000in}{0.000000in}}%
\pgfpathlineto{\pgfqpoint{4.900000in}{4.900000in}}%
\pgfpathlineto{\pgfqpoint{0.000000in}{4.900000in}}%
\pgfpathlineto{\pgfqpoint{0.000000in}{0.000000in}}%
\pgfpathclose%
\pgfusepath{fill}%
\end{pgfscope}%
\begin{pgfscope}%
\pgfpathrectangle{\pgfqpoint{0.100000in}{0.100000in}}{\pgfqpoint{4.700000in}{4.700000in}}%
\pgfusepath{clip}%
\pgfsetbuttcap%
\pgfsetmiterjoin%
\definecolor{currentfill}{rgb}{0.000000,0.000000,0.000000}%
\pgfsetfillcolor{currentfill}%
\pgfsetfillopacity{0.300000}%
\pgfsetlinewidth{2.007500pt}%
\definecolor{currentstroke}{rgb}{0.000000,0.000000,0.000000}%
\pgfsetstrokecolor{currentstroke}%
\pgfsetstrokeopacity{0.300000}%
\pgfsetdash{}{0pt}%
\pgfpathmoveto{\pgfqpoint{0.100000in}{0.100000in}}%
\pgfpathlineto{\pgfqpoint{1.666667in}{0.100000in}}%
\pgfpathlineto{\pgfqpoint{1.666667in}{2.450000in}}%
\pgfpathlineto{\pgfqpoint{0.100000in}{2.450000in}}%
\pgfpathlineto{\pgfqpoint{0.100000in}{0.100000in}}%
\pgfpathclose%
\pgfusepath{stroke,fill}%
\end{pgfscope}%
\begin{pgfscope}%
\pgfpathrectangle{\pgfqpoint{0.100000in}{0.100000in}}{\pgfqpoint{4.700000in}{4.700000in}}%
\pgfusepath{clip}%
\pgfsetbuttcap%
\pgfsetmiterjoin%
\pgfsetlinewidth{2.007500pt}%
\definecolor{currentstroke}{rgb}{0.000000,0.000000,0.000000}%
\pgfsetstrokecolor{currentstroke}%
\pgfsetstrokeopacity{0.500000}%
\pgfsetdash{}{0pt}%
\pgfpathmoveto{\pgfqpoint{0.100000in}{0.100000in}}%
\pgfpathlineto{\pgfqpoint{1.666667in}{0.100000in}}%
\pgfpathlineto{\pgfqpoint{1.666667in}{2.450000in}}%
\pgfpathlineto{\pgfqpoint{0.100000in}{2.450000in}}%
\pgfpathlineto{\pgfqpoint{0.100000in}{0.100000in}}%
\pgfpathclose%
\pgfusepath{stroke}%
\end{pgfscope}%
\begin{pgfscope}%
\pgfpathrectangle{\pgfqpoint{0.100000in}{0.100000in}}{\pgfqpoint{4.700000in}{4.700000in}}%
\pgfusepath{clip}%
\pgfsetbuttcap%
\pgfsetmiterjoin%
\definecolor{currentfill}{rgb}{0.000000,0.000000,0.000000}%
\pgfsetfillcolor{currentfill}%
\pgfsetfillopacity{0.300000}%
\pgfsetlinewidth{2.007500pt}%
\definecolor{currentstroke}{rgb}{0.000000,0.000000,0.000000}%
\pgfsetstrokecolor{currentstroke}%
\pgfsetstrokeopacity{0.300000}%
\pgfsetdash{}{0pt}%
\pgfpathmoveto{\pgfqpoint{3.233333in}{2.450000in}}%
\pgfpathlineto{\pgfqpoint{4.800000in}{2.450000in}}%
\pgfpathlineto{\pgfqpoint{4.800000in}{4.800000in}}%
\pgfpathlineto{\pgfqpoint{3.233333in}{4.800000in}}%
\pgfpathlineto{\pgfqpoint{3.233333in}{2.450000in}}%
\pgfpathclose%
\pgfusepath{stroke,fill}%
\end{pgfscope}%
\begin{pgfscope}%
\pgfpathrectangle{\pgfqpoint{0.100000in}{0.100000in}}{\pgfqpoint{4.700000in}{4.700000in}}%
\pgfusepath{clip}%
\pgfsetbuttcap%
\pgfsetmiterjoin%
\pgfsetlinewidth{2.007500pt}%
\definecolor{currentstroke}{rgb}{0.000000,0.000000,0.000000}%
\pgfsetstrokecolor{currentstroke}%
\pgfsetstrokeopacity{0.500000}%
\pgfsetdash{}{0pt}%
\pgfpathmoveto{\pgfqpoint{3.233333in}{2.450000in}}%
\pgfpathlineto{\pgfqpoint{4.800000in}{2.450000in}}%
\pgfpathlineto{\pgfqpoint{4.800000in}{4.800000in}}%
\pgfpathlineto{\pgfqpoint{3.233333in}{4.800000in}}%
\pgfpathlineto{\pgfqpoint{3.233333in}{2.450000in}}%
\pgfpathclose%
\pgfusepath{stroke}%
\end{pgfscope}%
\begin{pgfscope}%
\pgfpathrectangle{\pgfqpoint{0.100000in}{0.100000in}}{\pgfqpoint{4.700000in}{4.700000in}}%
\pgfusepath{clip}%
\pgfsetrectcap%
\pgfsetroundjoin%
\pgfsetlinewidth{3.011250pt}%
\definecolor{currentstroke}{rgb}{0.613281,0.125000,0.960784}%
\pgfsetstrokecolor{currentstroke}%
\pgfsetdash{}{0pt}%
\pgfpathmoveto{\pgfqpoint{0.100000in}{2.450000in}}%
\pgfpathlineto{\pgfqpoint{4.800000in}{2.450000in}}%
\pgfusepath{stroke}%
\end{pgfscope}%
\begin{pgfscope}%
\definecolor{textcolor}{rgb}{0.000000,0.000000,0.000000}%
\pgfsetstrokecolor{textcolor}%
\pgfsetfillcolor{textcolor}%
\pgftext[x=0.883333in,y=1.275000in,left,base]{\color{textcolor}{\rmfamily\fontsize{20.000000}{24.000000}\selectfont\catcode`\^=\active\def^{\ifmmode\sp\else\^{}\fi}\catcode`\%=\active\def
\end{pgfscope}%
\begin{pgfscope}%
\definecolor{textcolor}{rgb}{0.000000,0.000000,0.000000}%
\pgfsetstrokecolor{textcolor}%
\pgfsetfillcolor{textcolor}%
\pgftext[x=4.016667in,y=3.625000in,left,base]{\color{textcolor}{\rmfamily\fontsize{20.000000}{24.000000}\selectfont\catcode`\^=\active\def^{\ifmmode\sp\else\^{}\fi}\catcode`\%=\active\def
\end{pgfscope}%
\end{pgfpicture}%
\makeatother%
\endgroup%

%% file: figs/bp-fp2d-ex1.pgf
\begingroup%
\makeatletter%
\begin{pgfpicture}%
\pgfpathrectangle{\pgfpointorigin}{\pgfqpoint{3.625679in}{3.888024in}}%
\pgfusepath{use as bounding box, clip}%
\begin{pgfscope}%
\pgfsetbuttcap%
\pgfsetmiterjoin%
\definecolor{currentfill}{rgb}{1.000000,1.000000,1.000000}%
\pgfsetfillcolor{currentfill}%
\pgfsetlinewidth{0.000000pt}%
\definecolor{currentstroke}{rgb}{1.000000,1.000000,1.000000}%
\pgfsetstrokecolor{currentstroke}%
\pgfsetdash{}{0pt}%
\pgfpathmoveto{\pgfqpoint{0.000000in}{0.000000in}}%
\pgfpathlineto{\pgfqpoint{3.625679in}{0.000000in}}%
\pgfpathlineto{\pgfqpoint{3.625679in}{3.888024in}}%
\pgfpathlineto{\pgfqpoint{0.000000in}{3.888024in}}%
\pgfpathlineto{\pgfqpoint{0.000000in}{0.000000in}}%
\pgfpathclose%
\pgfusepath{fill}%
\end{pgfscope}%
\begin{pgfscope}%
\pgfsetbuttcap%
\pgfsetmiterjoin%
\definecolor{currentfill}{rgb}{1.000000,1.000000,1.000000}%
\pgfsetfillcolor{currentfill}%
\pgfsetlinewidth{0.000000pt}%
\definecolor{currentstroke}{rgb}{0.000000,0.000000,0.000000}%
\pgfsetstrokecolor{currentstroke}%
\pgfsetstrokeopacity{0.000000}%
\pgfsetdash{}{0pt}%
\pgfpathmoveto{\pgfqpoint{0.445679in}{0.499691in}}%
\pgfpathlineto{\pgfqpoint{3.525679in}{0.499691in}}%
\pgfpathlineto{\pgfqpoint{3.525679in}{3.579691in}}%
\pgfpathlineto{\pgfqpoint{0.445679in}{3.579691in}}%
\pgfpathlineto{\pgfqpoint{0.445679in}{0.499691in}}%
\pgfpathclose%
\pgfusepath{fill}%
\end{pgfscope}%
\begin{pgfscope}%
\pgfpathrectangle{\pgfqpoint{0.445679in}{0.499691in}}{\pgfqpoint{3.080000in}{3.080000in}}%
\pgfusepath{clip}%
\pgfsetrectcap%
\pgfsetroundjoin%
\pgfsetlinewidth{0.803000pt}%
\definecolor{currentstroke}{rgb}{0.690196,0.690196,0.690196}%
\pgfsetstrokecolor{currentstroke}%
\pgfsetdash{}{0pt}%
\pgfpathmoveto{\pgfqpoint{0.830390in}{0.499691in}}%
\pgfpathlineto{\pgfqpoint{0.830390in}{3.579691in}}%
\pgfusepath{stroke}%
\end{pgfscope}%
\begin{pgfscope}%
\pgfsetbuttcap%
\pgfsetroundjoin%
\definecolor{currentfill}{rgb}{0.000000,0.000000,0.000000}%
\pgfsetfillcolor{currentfill}%
\pgfsetlinewidth{0.803000pt}%
\definecolor{currentstroke}{rgb}{0.000000,0.000000,0.000000}%
\pgfsetstrokecolor{currentstroke}%
\pgfsetdash{}{0pt}%
\pgfsys@defobject{currentmarker}{\pgfqpoint{0.000000in}{-0.048611in}}{\pgfqpoint{0.000000in}{0.000000in}}{%
\pgfpathmoveto{\pgfqpoint{0.000000in}{0.000000in}}%
\pgfpathlineto{\pgfqpoint{0.000000in}{-0.048611in}}%
\pgfusepath{stroke,fill}%
}%
\begin{pgfscope}%
\pgfsys@transformshift{0.830390in}{0.499691in}%
\pgfsys@useobject{currentmarker}{}%
\end{pgfscope}%
\end{pgfscope}%
\begin{pgfscope}%
\definecolor{textcolor}{rgb}{0.000000,0.000000,0.000000}%
\pgfsetstrokecolor{textcolor}%
\pgfsetfillcolor{textcolor}%
\pgftext[x=0.830390in,y=0.402469in,,top]{\color{textcolor}{\rmfamily\fontsize{10.000000}{12.000000}\selectfont\catcode`\^=\active\def^{\ifmmode\sp\else\^{}\fi}\catcode`\%=\active\def
\end{pgfscope}%
\begin{pgfscope}%
\pgfpathrectangle{\pgfqpoint{0.445679in}{0.499691in}}{\pgfqpoint{3.080000in}{3.080000in}}%
\pgfusepath{clip}%
\pgfsetrectcap%
\pgfsetroundjoin%
\pgfsetlinewidth{0.803000pt}%
\definecolor{currentstroke}{rgb}{0.690196,0.690196,0.690196}%
\pgfsetstrokecolor{currentstroke}%
\pgfsetdash{}{0pt}%
\pgfpathmoveto{\pgfqpoint{1.215487in}{0.499691in}}%
\pgfpathlineto{\pgfqpoint{1.215487in}{3.579691in}}%
\pgfusepath{stroke}%
\end{pgfscope}%
\begin{pgfscope}%
\pgfsetbuttcap%
\pgfsetroundjoin%
\definecolor{currentfill}{rgb}{0.000000,0.000000,0.000000}%
\pgfsetfillcolor{currentfill}%
\pgfsetlinewidth{0.803000pt}%
\definecolor{currentstroke}{rgb}{0.000000,0.000000,0.000000}%
\pgfsetstrokecolor{currentstroke}%
\pgfsetdash{}{0pt}%
\pgfsys@defobject{currentmarker}{\pgfqpoint{0.000000in}{-0.048611in}}{\pgfqpoint{0.000000in}{0.000000in}}{%
\pgfpathmoveto{\pgfqpoint{0.000000in}{0.000000in}}%
\pgfpathlineto{\pgfqpoint{0.000000in}{-0.048611in}}%
\pgfusepath{stroke,fill}%
}%
\begin{pgfscope}%
\pgfsys@transformshift{1.215487in}{0.499691in}%
\pgfsys@useobject{currentmarker}{}%
\end{pgfscope}%
\end{pgfscope}%
\begin{pgfscope}%
\definecolor{textcolor}{rgb}{0.000000,0.000000,0.000000}%
\pgfsetstrokecolor{textcolor}%
\pgfsetfillcolor{textcolor}%
\pgftext[x=1.215487in,y=0.402469in,,top]{\color{textcolor}{\rmfamily\fontsize{10.000000}{12.000000}\selectfont\catcode`\^=\active\def^{\ifmmode\sp\else\^{}\fi}\catcode`\%=\active\def
\end{pgfscope}%
\begin{pgfscope}%
\pgfpathrectangle{\pgfqpoint{0.445679in}{0.499691in}}{\pgfqpoint{3.080000in}{3.080000in}}%
\pgfusepath{clip}%
\pgfsetrectcap%
\pgfsetroundjoin%
\pgfsetlinewidth{0.803000pt}%
\definecolor{currentstroke}{rgb}{0.690196,0.690196,0.690196}%
\pgfsetstrokecolor{currentstroke}%
\pgfsetdash{}{0pt}%
\pgfpathmoveto{\pgfqpoint{1.600583in}{0.499691in}}%
\pgfpathlineto{\pgfqpoint{1.600583in}{3.579691in}}%
\pgfusepath{stroke}%
\end{pgfscope}%
\begin{pgfscope}%
\pgfsetbuttcap%
\pgfsetroundjoin%
\definecolor{currentfill}{rgb}{0.000000,0.000000,0.000000}%
\pgfsetfillcolor{currentfill}%
\pgfsetlinewidth{0.803000pt}%
\definecolor{currentstroke}{rgb}{0.000000,0.000000,0.000000}%
\pgfsetstrokecolor{currentstroke}%
\pgfsetdash{}{0pt}%
\pgfsys@defobject{currentmarker}{\pgfqpoint{0.000000in}{-0.048611in}}{\pgfqpoint{0.000000in}{0.000000in}}{%
\pgfpathmoveto{\pgfqpoint{0.000000in}{0.000000in}}%
\pgfpathlineto{\pgfqpoint{0.000000in}{-0.048611in}}%
\pgfusepath{stroke,fill}%
}%
\begin{pgfscope}%
\pgfsys@transformshift{1.600583in}{0.499691in}%
\pgfsys@useobject{currentmarker}{}%
\end{pgfscope}%
\end{pgfscope}%
\begin{pgfscope}%
\definecolor{textcolor}{rgb}{0.000000,0.000000,0.000000}%
\pgfsetstrokecolor{textcolor}%
\pgfsetfillcolor{textcolor}%
\pgftext[x=1.600583in,y=0.402469in,,top]{\color{textcolor}{\rmfamily\fontsize{10.000000}{12.000000}\selectfont\catcode`\^=\active\def^{\ifmmode\sp\else\^{}\fi}\catcode`\%=\active\def
\end{pgfscope}%
\begin{pgfscope}%
\pgfpathrectangle{\pgfqpoint{0.445679in}{0.499691in}}{\pgfqpoint{3.080000in}{3.080000in}}%
\pgfusepath{clip}%
\pgfsetrectcap%
\pgfsetroundjoin%
\pgfsetlinewidth{0.803000pt}%
\definecolor{currentstroke}{rgb}{0.690196,0.690196,0.690196}%
\pgfsetstrokecolor{currentstroke}%
\pgfsetdash{}{0pt}%
\pgfpathmoveto{\pgfqpoint{1.985679in}{0.499691in}}%
\pgfpathlineto{\pgfqpoint{1.985679in}{3.579691in}}%
\pgfusepath{stroke}%
\end{pgfscope}%
\begin{pgfscope}%
\pgfsetbuttcap%
\pgfsetroundjoin%
\definecolor{currentfill}{rgb}{0.000000,0.000000,0.000000}%
\pgfsetfillcolor{currentfill}%
\pgfsetlinewidth{0.803000pt}%
\definecolor{currentstroke}{rgb}{0.000000,0.000000,0.000000}%
\pgfsetstrokecolor{currentstroke}%
\pgfsetdash{}{0pt}%
\pgfsys@defobject{currentmarker}{\pgfqpoint{0.000000in}{-0.048611in}}{\pgfqpoint{0.000000in}{0.000000in}}{%
\pgfpathmoveto{\pgfqpoint{0.000000in}{0.000000in}}%
\pgfpathlineto{\pgfqpoint{0.000000in}{-0.048611in}}%
\pgfusepath{stroke,fill}%
}%
\begin{pgfscope}%
\pgfsys@transformshift{1.985679in}{0.499691in}%
\pgfsys@useobject{currentmarker}{}%
\end{pgfscope}%
\end{pgfscope}%
\begin{pgfscope}%
\definecolor{textcolor}{rgb}{0.000000,0.000000,0.000000}%
\pgfsetstrokecolor{textcolor}%
\pgfsetfillcolor{textcolor}%
\pgftext[x=1.985679in,y=0.402469in,,top]{\color{textcolor}{\rmfamily\fontsize{10.000000}{12.000000}\selectfont\catcode`\^=\active\def^{\ifmmode\sp\else\^{}\fi}\catcode`\%=\active\def
\end{pgfscope}%
\begin{pgfscope}%
\pgfpathrectangle{\pgfqpoint{0.445679in}{0.499691in}}{\pgfqpoint{3.080000in}{3.080000in}}%
\pgfusepath{clip}%
\pgfsetrectcap%
\pgfsetroundjoin%
\pgfsetlinewidth{0.803000pt}%
\definecolor{currentstroke}{rgb}{0.690196,0.690196,0.690196}%
\pgfsetstrokecolor{currentstroke}%
\pgfsetdash{}{0pt}%
\pgfpathmoveto{\pgfqpoint{2.370775in}{0.499691in}}%
\pgfpathlineto{\pgfqpoint{2.370775in}{3.579691in}}%
\pgfusepath{stroke}%
\end{pgfscope}%
\begin{pgfscope}%
\pgfsetbuttcap%
\pgfsetroundjoin%
\definecolor{currentfill}{rgb}{0.000000,0.000000,0.000000}%
\pgfsetfillcolor{currentfill}%
\pgfsetlinewidth{0.803000pt}%
\definecolor{currentstroke}{rgb}{0.000000,0.000000,0.000000}%
\pgfsetstrokecolor{currentstroke}%
\pgfsetdash{}{0pt}%
\pgfsys@defobject{currentmarker}{\pgfqpoint{0.000000in}{-0.048611in}}{\pgfqpoint{0.000000in}{0.000000in}}{%
\pgfpathmoveto{\pgfqpoint{0.000000in}{0.000000in}}%
\pgfpathlineto{\pgfqpoint{0.000000in}{-0.048611in}}%
\pgfusepath{stroke,fill}%
}%
\begin{pgfscope}%
\pgfsys@transformshift{2.370775in}{0.499691in}%
\pgfsys@useobject{currentmarker}{}%
\end{pgfscope}%
\end{pgfscope}%
\begin{pgfscope}%
\definecolor{textcolor}{rgb}{0.000000,0.000000,0.000000}%
\pgfsetstrokecolor{textcolor}%
\pgfsetfillcolor{textcolor}%
\pgftext[x=2.370775in,y=0.402469in,,top]{\color{textcolor}{\rmfamily\fontsize{10.000000}{12.000000}\selectfont\catcode`\^=\active\def^{\ifmmode\sp\else\^{}\fi}\catcode`\%=\active\def
\end{pgfscope}%
\begin{pgfscope}%
\pgfpathrectangle{\pgfqpoint{0.445679in}{0.499691in}}{\pgfqpoint{3.080000in}{3.080000in}}%
\pgfusepath{clip}%
\pgfsetrectcap%
\pgfsetroundjoin%
\pgfsetlinewidth{0.803000pt}%
\definecolor{currentstroke}{rgb}{0.690196,0.690196,0.690196}%
\pgfsetstrokecolor{currentstroke}%
\pgfsetdash{}{0pt}%
\pgfpathmoveto{\pgfqpoint{2.755872in}{0.499691in}}%
\pgfpathlineto{\pgfqpoint{2.755872in}{3.579691in}}%
\pgfusepath{stroke}%
\end{pgfscope}%
\begin{pgfscope}%
\pgfsetbuttcap%
\pgfsetroundjoin%
\definecolor{currentfill}{rgb}{0.000000,0.000000,0.000000}%
\pgfsetfillcolor{currentfill}%
\pgfsetlinewidth{0.803000pt}%
\definecolor{currentstroke}{rgb}{0.000000,0.000000,0.000000}%
\pgfsetstrokecolor{currentstroke}%
\pgfsetdash{}{0pt}%
\pgfsys@defobject{currentmarker}{\pgfqpoint{0.000000in}{-0.048611in}}{\pgfqpoint{0.000000in}{0.000000in}}{%
\pgfpathmoveto{\pgfqpoint{0.000000in}{0.000000in}}%
\pgfpathlineto{\pgfqpoint{0.000000in}{-0.048611in}}%
\pgfusepath{stroke,fill}%
}%
\begin{pgfscope}%
\pgfsys@transformshift{2.755872in}{0.499691in}%
\pgfsys@useobject{currentmarker}{}%
\end{pgfscope}%
\end{pgfscope}%
\begin{pgfscope}%
\definecolor{textcolor}{rgb}{0.000000,0.000000,0.000000}%
\pgfsetstrokecolor{textcolor}%
\pgfsetfillcolor{textcolor}%
\pgftext[x=2.755872in,y=0.402469in,,top]{\color{textcolor}{\rmfamily\fontsize{10.000000}{12.000000}\selectfont\catcode`\^=\active\def^{\ifmmode\sp\else\^{}\fi}\catcode`\%=\active\def
\end{pgfscope}%
\begin{pgfscope}%
\pgfpathrectangle{\pgfqpoint{0.445679in}{0.499691in}}{\pgfqpoint{3.080000in}{3.080000in}}%
\pgfusepath{clip}%
\pgfsetrectcap%
\pgfsetroundjoin%
\pgfsetlinewidth{0.803000pt}%
\definecolor{currentstroke}{rgb}{0.690196,0.690196,0.690196}%
\pgfsetstrokecolor{currentstroke}%
\pgfsetdash{}{0pt}%
\pgfpathmoveto{\pgfqpoint{3.140968in}{0.499691in}}%
\pgfpathlineto{\pgfqpoint{3.140968in}{3.579691in}}%
\pgfusepath{stroke}%
\end{pgfscope}%
\begin{pgfscope}%
\pgfsetbuttcap%
\pgfsetroundjoin%
\definecolor{currentfill}{rgb}{0.000000,0.000000,0.000000}%
\pgfsetfillcolor{currentfill}%
\pgfsetlinewidth{0.803000pt}%
\definecolor{currentstroke}{rgb}{0.000000,0.000000,0.000000}%
\pgfsetstrokecolor{currentstroke}%
\pgfsetdash{}{0pt}%
\pgfsys@defobject{currentmarker}{\pgfqpoint{0.000000in}{-0.048611in}}{\pgfqpoint{0.000000in}{0.000000in}}{%
\pgfpathmoveto{\pgfqpoint{0.000000in}{0.000000in}}%
\pgfpathlineto{\pgfqpoint{0.000000in}{-0.048611in}}%
\pgfusepath{stroke,fill}%
}%
\begin{pgfscope}%
\pgfsys@transformshift{3.140968in}{0.499691in}%
\pgfsys@useobject{currentmarker}{}%
\end{pgfscope}%
\end{pgfscope}%
\begin{pgfscope}%
\definecolor{textcolor}{rgb}{0.000000,0.000000,0.000000}%
\pgfsetstrokecolor{textcolor}%
\pgfsetfillcolor{textcolor}%
\pgftext[x=3.140968in,y=0.402469in,,top]{\color{textcolor}{\rmfamily\fontsize{10.000000}{12.000000}\selectfont\catcode`\^=\active\def^{\ifmmode\sp\else\^{}\fi}\catcode`\%=\active\def
\end{pgfscope}%
\begin{pgfscope}%
\definecolor{textcolor}{rgb}{0.000000,0.000000,0.000000}%
\pgfsetstrokecolor{textcolor}%
\pgfsetfillcolor{textcolor}%
\pgftext[x=1.985679in,y=0.223457in,,top]{\color{textcolor}{\rmfamily\fontsize{10.000000}{12.000000}\selectfont\catcode`\^=\active\def^{\ifmmode\sp\else\^{}\fi}\catcode`\%=\active\def
\end{pgfscope}%
\begin{pgfscope}%
\pgfpathrectangle{\pgfqpoint{0.445679in}{0.499691in}}{\pgfqpoint{3.080000in}{3.080000in}}%
\pgfusepath{clip}%
\pgfsetrectcap%
\pgfsetroundjoin%
\pgfsetlinewidth{0.803000pt}%
\definecolor{currentstroke}{rgb}{0.690196,0.690196,0.690196}%
\pgfsetstrokecolor{currentstroke}%
\pgfsetdash{}{0pt}%
\pgfpathmoveto{\pgfqpoint{0.445679in}{2.039691in}}%
\pgfpathlineto{\pgfqpoint{3.525679in}{2.039691in}}%
\pgfusepath{stroke}%
\end{pgfscope}%
\begin{pgfscope}%
\pgfsetbuttcap%
\pgfsetroundjoin%
\definecolor{currentfill}{rgb}{0.000000,0.000000,0.000000}%
\pgfsetfillcolor{currentfill}%
\pgfsetlinewidth{0.803000pt}%
\definecolor{currentstroke}{rgb}{0.000000,0.000000,0.000000}%
\pgfsetstrokecolor{currentstroke}%
\pgfsetdash{}{0pt}%
\pgfsys@defobject{currentmarker}{\pgfqpoint{-0.048611in}{0.000000in}}{\pgfqpoint{-0.000000in}{0.000000in}}{%
\pgfpathmoveto{\pgfqpoint{-0.000000in}{0.000000in}}%
\pgfpathlineto{\pgfqpoint{-0.048611in}{0.000000in}}%
\pgfusepath{stroke,fill}%
}%
\begin{pgfscope}%
\pgfsys@transformshift{0.445679in}{2.039691in}%
\pgfsys@useobject{currentmarker}{}%
\end{pgfscope}%
\end{pgfscope}%
\begin{pgfscope}%
\definecolor{textcolor}{rgb}{0.000000,0.000000,0.000000}%
\pgfsetstrokecolor{textcolor}%
\pgfsetfillcolor{textcolor}%
\pgftext[x=0.279012in, y=1.991466in, left, base]{\color{textcolor}{\rmfamily\fontsize{10.000000}{12.000000}\selectfont\catcode`\^=\active\def^{\ifmmode\sp\else\^{}\fi}\catcode`\%=\active\def
\end{pgfscope}%
\begin{pgfscope}%
\pgfpathrectangle{\pgfqpoint{0.445679in}{0.499691in}}{\pgfqpoint{3.080000in}{3.080000in}}%
\pgfusepath{clip}%
\pgfsetrectcap%
\pgfsetroundjoin%
\pgfsetlinewidth{0.803000pt}%
\definecolor{currentstroke}{rgb}{0.690196,0.690196,0.690196}%
\pgfsetstrokecolor{currentstroke}%
\pgfsetdash{}{0pt}%
\pgfpathmoveto{\pgfqpoint{0.445679in}{2.424787in}}%
\pgfpathlineto{\pgfqpoint{3.525679in}{2.424787in}}%
\pgfusepath{stroke}%
\end{pgfscope}%
\begin{pgfscope}%
\pgfsetbuttcap%
\pgfsetroundjoin%
\definecolor{currentfill}{rgb}{0.000000,0.000000,0.000000}%
\pgfsetfillcolor{currentfill}%
\pgfsetlinewidth{0.803000pt}%
\definecolor{currentstroke}{rgb}{0.000000,0.000000,0.000000}%
\pgfsetstrokecolor{currentstroke}%
\pgfsetdash{}{0pt}%
\pgfsys@defobject{currentmarker}{\pgfqpoint{-0.048611in}{0.000000in}}{\pgfqpoint{-0.000000in}{0.000000in}}{%
\pgfpathmoveto{\pgfqpoint{-0.000000in}{0.000000in}}%
\pgfpathlineto{\pgfqpoint{-0.048611in}{0.000000in}}%
\pgfusepath{stroke,fill}%
}%
\begin{pgfscope}%
\pgfsys@transformshift{0.445679in}{2.424787in}%
\pgfsys@useobject{currentmarker}{}%
\end{pgfscope}%
\end{pgfscope}%
\begin{pgfscope}%
\definecolor{textcolor}{rgb}{0.000000,0.000000,0.000000}%
\pgfsetstrokecolor{textcolor}%
\pgfsetfillcolor{textcolor}%
\pgftext[x=0.279012in, y=2.376562in, left, base]{\color{textcolor}{\rmfamily\fontsize{10.000000}{12.000000}\selectfont\catcode`\^=\active\def^{\ifmmode\sp\else\^{}\fi}\catcode`\%=\active\def
\end{pgfscope}%
\begin{pgfscope}%
\pgfpathrectangle{\pgfqpoint{0.445679in}{0.499691in}}{\pgfqpoint{3.080000in}{3.080000in}}%
\pgfusepath{clip}%
\pgfsetrectcap%
\pgfsetroundjoin%
\pgfsetlinewidth{0.803000pt}%
\definecolor{currentstroke}{rgb}{0.690196,0.690196,0.690196}%
\pgfsetstrokecolor{currentstroke}%
\pgfsetdash{}{0pt}%
\pgfpathmoveto{\pgfqpoint{0.445679in}{0.884402in}}%
\pgfpathlineto{\pgfqpoint{3.525679in}{0.884402in}}%
\pgfusepath{stroke}%
\end{pgfscope}%
\begin{pgfscope}%
\pgfsetbuttcap%
\pgfsetroundjoin%
\definecolor{currentfill}{rgb}{0.000000,0.000000,0.000000}%
\pgfsetfillcolor{currentfill}%
\pgfsetlinewidth{0.803000pt}%
\definecolor{currentstroke}{rgb}{0.000000,0.000000,0.000000}%
\pgfsetstrokecolor{currentstroke}%
\pgfsetdash{}{0pt}%
\pgfsys@defobject{currentmarker}{\pgfqpoint{-0.048611in}{0.000000in}}{\pgfqpoint{-0.000000in}{0.000000in}}{%
\pgfpathmoveto{\pgfqpoint{-0.000000in}{0.000000in}}%
\pgfpathlineto{\pgfqpoint{-0.048611in}{0.000000in}}%
\pgfusepath{stroke,fill}%
}%
\begin{pgfscope}%
\pgfsys@transformshift{0.445679in}{0.884402in}%
\pgfsys@useobject{currentmarker}{}%
\end{pgfscope}%
\end{pgfscope}%
\begin{pgfscope}%
\definecolor{textcolor}{rgb}{0.000000,0.000000,0.000000}%
\pgfsetstrokecolor{textcolor}%
\pgfsetfillcolor{textcolor}%
\pgftext[x=0.279012in, y=0.836177in, left, base]{\color{textcolor}{\rmfamily\fontsize{10.000000}{12.000000}\selectfont\catcode`\^=\active\def^{\ifmmode\sp\else\^{}\fi}\catcode`\%=\active\def
\end{pgfscope}%
\begin{pgfscope}%
\pgfpathrectangle{\pgfqpoint{0.445679in}{0.499691in}}{\pgfqpoint{3.080000in}{3.080000in}}%
\pgfusepath{clip}%
\pgfsetrectcap%
\pgfsetroundjoin%
\pgfsetlinewidth{0.803000pt}%
\definecolor{currentstroke}{rgb}{0.690196,0.690196,0.690196}%
\pgfsetstrokecolor{currentstroke}%
\pgfsetdash{}{0pt}%
\pgfpathmoveto{\pgfqpoint{0.445679in}{1.654595in}}%
\pgfpathlineto{\pgfqpoint{3.525679in}{1.654595in}}%
\pgfusepath{stroke}%
\end{pgfscope}%
\begin{pgfscope}%
\pgfsetbuttcap%
\pgfsetroundjoin%
\definecolor{currentfill}{rgb}{0.000000,0.000000,0.000000}%
\pgfsetfillcolor{currentfill}%
\pgfsetlinewidth{0.803000pt}%
\definecolor{currentstroke}{rgb}{0.000000,0.000000,0.000000}%
\pgfsetstrokecolor{currentstroke}%
\pgfsetdash{}{0pt}%
\pgfsys@defobject{currentmarker}{\pgfqpoint{-0.048611in}{0.000000in}}{\pgfqpoint{-0.000000in}{0.000000in}}{%
\pgfpathmoveto{\pgfqpoint{-0.000000in}{0.000000in}}%
\pgfpathlineto{\pgfqpoint{-0.048611in}{0.000000in}}%
\pgfusepath{stroke,fill}%
}%
\begin{pgfscope}%
\pgfsys@transformshift{0.445679in}{1.654595in}%
\pgfsys@useobject{currentmarker}{}%
\end{pgfscope}%
\end{pgfscope}%
\begin{pgfscope}%
\definecolor{textcolor}{rgb}{0.000000,0.000000,0.000000}%
\pgfsetstrokecolor{textcolor}%
\pgfsetfillcolor{textcolor}%
\pgftext[x=0.279012in, y=1.606370in, left, base]{\color{textcolor}{\rmfamily\fontsize{10.000000}{12.000000}\selectfont\catcode`\^=\active\def^{\ifmmode\sp\else\^{}\fi}\catcode`\%=\active\def
\end{pgfscope}%
\begin{pgfscope}%
\pgfpathrectangle{\pgfqpoint{0.445679in}{0.499691in}}{\pgfqpoint{3.080000in}{3.080000in}}%
\pgfusepath{clip}%
\pgfsetrectcap%
\pgfsetroundjoin%
\pgfsetlinewidth{0.803000pt}%
\definecolor{currentstroke}{rgb}{0.690196,0.690196,0.690196}%
\pgfsetstrokecolor{currentstroke}%
\pgfsetdash{}{0pt}%
\pgfpathmoveto{\pgfqpoint{0.445679in}{3.194980in}}%
\pgfpathlineto{\pgfqpoint{3.525679in}{3.194980in}}%
\pgfusepath{stroke}%
\end{pgfscope}%
\begin{pgfscope}%
\pgfsetbuttcap%
\pgfsetroundjoin%
\definecolor{currentfill}{rgb}{0.000000,0.000000,0.000000}%
\pgfsetfillcolor{currentfill}%
\pgfsetlinewidth{0.803000pt}%
\definecolor{currentstroke}{rgb}{0.000000,0.000000,0.000000}%
\pgfsetstrokecolor{currentstroke}%
\pgfsetdash{}{0pt}%
\pgfsys@defobject{currentmarker}{\pgfqpoint{-0.048611in}{0.000000in}}{\pgfqpoint{-0.000000in}{0.000000in}}{%
\pgfpathmoveto{\pgfqpoint{-0.000000in}{0.000000in}}%
\pgfpathlineto{\pgfqpoint{-0.048611in}{0.000000in}}%
\pgfusepath{stroke,fill}%
}%
\begin{pgfscope}%
\pgfsys@transformshift{0.445679in}{3.194980in}%
\pgfsys@useobject{currentmarker}{}%
\end{pgfscope}%
\end{pgfscope}%
\begin{pgfscope}%
\definecolor{textcolor}{rgb}{0.000000,0.000000,0.000000}%
\pgfsetstrokecolor{textcolor}%
\pgfsetfillcolor{textcolor}%
\pgftext[x=0.279012in, y=3.146755in, left, base]{\color{textcolor}{\rmfamily\fontsize{10.000000}{12.000000}\selectfont\catcode`\^=\active\def^{\ifmmode\sp\else\^{}\fi}\catcode`\%=\active\def
\end{pgfscope}%
\begin{pgfscope}%
\pgfpathrectangle{\pgfqpoint{0.445679in}{0.499691in}}{\pgfqpoint{3.080000in}{3.080000in}}%
\pgfusepath{clip}%
\pgfsetrectcap%
\pgfsetroundjoin%
\pgfsetlinewidth{0.803000pt}%
\definecolor{currentstroke}{rgb}{0.690196,0.690196,0.690196}%
\pgfsetstrokecolor{currentstroke}%
\pgfsetdash{}{0pt}%
\pgfpathmoveto{\pgfqpoint{0.445679in}{2.809884in}}%
\pgfpathlineto{\pgfqpoint{3.525679in}{2.809884in}}%
\pgfusepath{stroke}%
\end{pgfscope}%
\begin{pgfscope}%
\pgfsetbuttcap%
\pgfsetroundjoin%
\definecolor{currentfill}{rgb}{0.000000,0.000000,0.000000}%
\pgfsetfillcolor{currentfill}%
\pgfsetlinewidth{0.803000pt}%
\definecolor{currentstroke}{rgb}{0.000000,0.000000,0.000000}%
\pgfsetstrokecolor{currentstroke}%
\pgfsetdash{}{0pt}%
\pgfsys@defobject{currentmarker}{\pgfqpoint{-0.048611in}{0.000000in}}{\pgfqpoint{-0.000000in}{0.000000in}}{%
\pgfpathmoveto{\pgfqpoint{-0.000000in}{0.000000in}}%
\pgfpathlineto{\pgfqpoint{-0.048611in}{0.000000in}}%
\pgfusepath{stroke,fill}%
}%
\begin{pgfscope}%
\pgfsys@transformshift{0.445679in}{2.809884in}%
\pgfsys@useobject{currentmarker}{}%
\end{pgfscope}%
\end{pgfscope}%
\begin{pgfscope}%
\definecolor{textcolor}{rgb}{0.000000,0.000000,0.000000}%
\pgfsetstrokecolor{textcolor}%
\pgfsetfillcolor{textcolor}%
\pgftext[x=0.279012in, y=2.761658in, left, base]{\color{textcolor}{\rmfamily\fontsize{10.000000}{12.000000}\selectfont\catcode`\^=\active\def^{\ifmmode\sp\else\^{}\fi}\catcode`\%=\active\def
\end{pgfscope}%
\begin{pgfscope}%
\pgfpathrectangle{\pgfqpoint{0.445679in}{0.499691in}}{\pgfqpoint{3.080000in}{3.080000in}}%
\pgfusepath{clip}%
\pgfsetrectcap%
\pgfsetroundjoin%
\pgfsetlinewidth{0.803000pt}%
\definecolor{currentstroke}{rgb}{0.690196,0.690196,0.690196}%
\pgfsetstrokecolor{currentstroke}%
\pgfsetdash{}{0pt}%
\pgfpathmoveto{\pgfqpoint{0.445679in}{1.269499in}}%
\pgfpathlineto{\pgfqpoint{3.525679in}{1.269499in}}%
\pgfusepath{stroke}%
\end{pgfscope}%
\begin{pgfscope}%
\pgfsetbuttcap%
\pgfsetroundjoin%
\definecolor{currentfill}{rgb}{0.000000,0.000000,0.000000}%
\pgfsetfillcolor{currentfill}%
\pgfsetlinewidth{0.803000pt}%
\definecolor{currentstroke}{rgb}{0.000000,0.000000,0.000000}%
\pgfsetstrokecolor{currentstroke}%
\pgfsetdash{}{0pt}%
\pgfsys@defobject{currentmarker}{\pgfqpoint{-0.048611in}{0.000000in}}{\pgfqpoint{-0.000000in}{0.000000in}}{%
\pgfpathmoveto{\pgfqpoint{-0.000000in}{0.000000in}}%
\pgfpathlineto{\pgfqpoint{-0.048611in}{0.000000in}}%
\pgfusepath{stroke,fill}%
}%
\begin{pgfscope}%
\pgfsys@transformshift{0.445679in}{1.269499in}%
\pgfsys@useobject{currentmarker}{}%
\end{pgfscope}%
\end{pgfscope}%
\begin{pgfscope}%
\definecolor{textcolor}{rgb}{0.000000,0.000000,0.000000}%
\pgfsetstrokecolor{textcolor}%
\pgfsetfillcolor{textcolor}%
\pgftext[x=0.279012in, y=1.221273in, left, base]{\color{textcolor}{\rmfamily\fontsize{10.000000}{12.000000}\selectfont\catcode`\^=\active\def^{\ifmmode\sp\else\^{}\fi}\catcode`\%=\active\def
\end{pgfscope}%
\begin{pgfscope}%
\definecolor{textcolor}{rgb}{0.000000,0.000000,0.000000}%
\pgfsetstrokecolor{textcolor}%
\pgfsetfillcolor{textcolor}%
\pgftext[x=0.223457in,y=2.039691in,,bottom,rotate=90.000000]{\color{textcolor}{\rmfamily\fontsize{10.000000}{12.000000}\selectfont\catcode`\^=\active\def^{\ifmmode\sp\else\^{}\fi}\catcode`\%=\active\def
\end{pgfscope}%
\begin{pgfscope}%
\pgfpathrectangle{\pgfqpoint{0.445679in}{0.499691in}}{\pgfqpoint{3.080000in}{3.080000in}}%
\pgfusepath{clip}%
\pgfsetrectcap%
\pgfsetroundjoin%
\pgfsetlinewidth{3.011250pt}%
\definecolor{currentstroke}{rgb}{0.121569,0.188235,0.501961}%
\pgfsetstrokecolor{currentstroke}%
\pgfsetdash{}{0pt}%
\pgfpathmoveto{\pgfqpoint{1.215487in}{2.424787in}}%
\pgfpathlineto{\pgfqpoint{0.830390in}{2.039691in}}%
\pgfusepath{stroke}%
\end{pgfscope}%
\begin{pgfscope}%
\pgfpathrectangle{\pgfqpoint{0.445679in}{0.499691in}}{\pgfqpoint{3.080000in}{3.080000in}}%
\pgfusepath{clip}%
\pgfsetrectcap%
\pgfsetroundjoin%
\pgfsetlinewidth{3.011250pt}%
\definecolor{currentstroke}{rgb}{0.121569,0.188235,0.501961}%
\pgfsetstrokecolor{currentstroke}%
\pgfsetdash{}{0pt}%
\pgfpathmoveto{\pgfqpoint{1.985679in}{1.654595in}}%
\pgfpathlineto{\pgfqpoint{1.600583in}{0.884402in}}%
\pgfusepath{stroke}%
\end{pgfscope}%
\begin{pgfscope}%
\pgfpathrectangle{\pgfqpoint{0.445679in}{0.499691in}}{\pgfqpoint{3.080000in}{3.080000in}}%
\pgfusepath{clip}%
\pgfsetrectcap%
\pgfsetroundjoin%
\pgfsetlinewidth{3.011250pt}%
\definecolor{currentstroke}{rgb}{0.121569,0.188235,0.501961}%
\pgfsetstrokecolor{currentstroke}%
\pgfsetdash{}{0pt}%
\pgfpathmoveto{\pgfqpoint{2.370775in}{3.194980in}}%
\pgfpathlineto{\pgfqpoint{1.215487in}{2.424787in}}%
\pgfusepath{stroke}%
\end{pgfscope}%
\begin{pgfscope}%
\pgfpathrectangle{\pgfqpoint{0.445679in}{0.499691in}}{\pgfqpoint{3.080000in}{3.080000in}}%
\pgfusepath{clip}%
\pgfsetrectcap%
\pgfsetroundjoin%
\pgfsetlinewidth{3.011250pt}%
\definecolor{currentstroke}{rgb}{0.121569,0.188235,0.501961}%
\pgfsetstrokecolor{currentstroke}%
\pgfsetdash{}{0pt}%
\pgfpathmoveto{\pgfqpoint{2.370775in}{3.194980in}}%
\pgfpathlineto{\pgfqpoint{1.985679in}{1.654595in}}%
\pgfusepath{stroke}%
\end{pgfscope}%
\begin{pgfscope}%
\pgfpathrectangle{\pgfqpoint{0.445679in}{0.499691in}}{\pgfqpoint{3.080000in}{3.080000in}}%
\pgfusepath{clip}%
\pgfsetrectcap%
\pgfsetroundjoin%
\pgfsetlinewidth{3.011250pt}%
\definecolor{currentstroke}{rgb}{0.121569,0.188235,0.501961}%
\pgfsetstrokecolor{currentstroke}%
\pgfsetdash{}{0pt}%
\pgfpathmoveto{\pgfqpoint{2.755872in}{2.809884in}}%
\pgfpathlineto{\pgfqpoint{1.215487in}{2.424787in}}%
\pgfusepath{stroke}%
\end{pgfscope}%
\begin{pgfscope}%
\pgfpathrectangle{\pgfqpoint{0.445679in}{0.499691in}}{\pgfqpoint{3.080000in}{3.080000in}}%
\pgfusepath{clip}%
\pgfsetrectcap%
\pgfsetroundjoin%
\pgfsetlinewidth{3.011250pt}%
\definecolor{currentstroke}{rgb}{0.121569,0.188235,0.501961}%
\pgfsetstrokecolor{currentstroke}%
\pgfsetdash{}{0pt}%
\pgfpathmoveto{\pgfqpoint{2.755872in}{2.809884in}}%
\pgfpathlineto{\pgfqpoint{1.985679in}{1.654595in}}%
\pgfusepath{stroke}%
\end{pgfscope}%
\begin{pgfscope}%
\pgfpathrectangle{\pgfqpoint{0.445679in}{0.499691in}}{\pgfqpoint{3.080000in}{3.080000in}}%
\pgfusepath{clip}%
\pgfsetrectcap%
\pgfsetroundjoin%
\pgfsetlinewidth{3.011250pt}%
\definecolor{currentstroke}{rgb}{0.121569,0.188235,0.501961}%
\pgfsetstrokecolor{currentstroke}%
\pgfsetdash{}{0pt}%
\pgfpathmoveto{\pgfqpoint{3.140968in}{1.269499in}}%
\pgfpathlineto{\pgfqpoint{1.600583in}{0.884402in}}%
\pgfusepath{stroke}%
\end{pgfscope}%
\begin{pgfscope}%
\pgfpathrectangle{\pgfqpoint{0.445679in}{0.499691in}}{\pgfqpoint{3.080000in}{3.080000in}}%
\pgfusepath{clip}%
\pgfsetrectcap%
\pgfsetroundjoin%
\pgfsetlinewidth{3.011250pt}%
\definecolor{currentstroke}{rgb}{0.501961,0.184314,0.133333}%
\pgfsetstrokecolor{currentstroke}%
\pgfsetdash{}{0pt}%
\pgfpathmoveto{\pgfqpoint{1.600583in}{0.884402in}}%
\pgfpathlineto{\pgfqpoint{0.830390in}{2.039691in}}%
\pgfusepath{stroke}%
\end{pgfscope}%
\begin{pgfscope}%
\pgfpathrectangle{\pgfqpoint{0.445679in}{0.499691in}}{\pgfqpoint{3.080000in}{3.080000in}}%
\pgfusepath{clip}%
\pgfsetrectcap%
\pgfsetroundjoin%
\pgfsetlinewidth{3.011250pt}%
\definecolor{currentstroke}{rgb}{0.501961,0.184314,0.133333}%
\pgfsetstrokecolor{currentstroke}%
\pgfsetdash{}{0pt}%
\pgfpathmoveto{\pgfqpoint{1.600583in}{0.884402in}}%
\pgfpathlineto{\pgfqpoint{1.215487in}{2.424787in}}%
\pgfusepath{stroke}%
\end{pgfscope}%
\begin{pgfscope}%
\pgfpathrectangle{\pgfqpoint{0.445679in}{0.499691in}}{\pgfqpoint{3.080000in}{3.080000in}}%
\pgfusepath{clip}%
\pgfsetrectcap%
\pgfsetroundjoin%
\pgfsetlinewidth{3.011250pt}%
\definecolor{currentstroke}{rgb}{0.501961,0.184314,0.133333}%
\pgfsetstrokecolor{currentstroke}%
\pgfsetdash{}{0pt}%
\pgfpathmoveto{\pgfqpoint{1.985679in}{1.654595in}}%
\pgfpathlineto{\pgfqpoint{0.830390in}{2.039691in}}%
\pgfusepath{stroke}%
\end{pgfscope}%
\begin{pgfscope}%
\pgfpathrectangle{\pgfqpoint{0.445679in}{0.499691in}}{\pgfqpoint{3.080000in}{3.080000in}}%
\pgfusepath{clip}%
\pgfsetrectcap%
\pgfsetroundjoin%
\pgfsetlinewidth{3.011250pt}%
\definecolor{currentstroke}{rgb}{0.501961,0.184314,0.133333}%
\pgfsetstrokecolor{currentstroke}%
\pgfsetdash{}{0pt}%
\pgfpathmoveto{\pgfqpoint{1.985679in}{1.654595in}}%
\pgfpathlineto{\pgfqpoint{1.215487in}{2.424787in}}%
\pgfusepath{stroke}%
\end{pgfscope}%
\begin{pgfscope}%
\pgfpathrectangle{\pgfqpoint{0.445679in}{0.499691in}}{\pgfqpoint{3.080000in}{3.080000in}}%
\pgfusepath{clip}%
\pgfsetrectcap%
\pgfsetroundjoin%
\pgfsetlinewidth{3.011250pt}%
\definecolor{currentstroke}{rgb}{0.501961,0.184314,0.133333}%
\pgfsetstrokecolor{currentstroke}%
\pgfsetdash{}{0pt}%
\pgfpathmoveto{\pgfqpoint{2.755872in}{2.809884in}}%
\pgfpathlineto{\pgfqpoint{2.370775in}{3.194980in}}%
\pgfusepath{stroke}%
\end{pgfscope}%
\begin{pgfscope}%
\pgfpathrectangle{\pgfqpoint{0.445679in}{0.499691in}}{\pgfqpoint{3.080000in}{3.080000in}}%
\pgfusepath{clip}%
\pgfsetrectcap%
\pgfsetroundjoin%
\pgfsetlinewidth{3.011250pt}%
\definecolor{currentstroke}{rgb}{0.501961,0.184314,0.133333}%
\pgfsetstrokecolor{currentstroke}%
\pgfsetdash{}{0pt}%
\pgfpathmoveto{\pgfqpoint{3.140968in}{1.269499in}}%
\pgfpathlineto{\pgfqpoint{1.985679in}{1.654595in}}%
\pgfusepath{stroke}%
\end{pgfscope}%
\begin{pgfscope}%
\pgfpathrectangle{\pgfqpoint{0.445679in}{0.499691in}}{\pgfqpoint{3.080000in}{3.080000in}}%
\pgfusepath{clip}%
\pgfsetrectcap%
\pgfsetroundjoin%
\pgfsetlinewidth{3.011250pt}%
\definecolor{currentstroke}{rgb}{0.501961,0.184314,0.133333}%
\pgfsetstrokecolor{currentstroke}%
\pgfsetdash{}{0pt}%
\pgfpathmoveto{\pgfqpoint{3.140968in}{1.269499in}}%
\pgfpathlineto{\pgfqpoint{2.755872in}{2.809884in}}%
\pgfusepath{stroke}%
\end{pgfscope}%
\begin{pgfscope}%
\pgfsetrectcap%
\pgfsetmiterjoin%
\pgfsetlinewidth{0.803000pt}%
\definecolor{currentstroke}{rgb}{0.000000,0.000000,0.000000}%
\pgfsetstrokecolor{currentstroke}%
\pgfsetdash{}{0pt}%
\pgfpathmoveto{\pgfqpoint{0.445679in}{0.499691in}}%
\pgfpathlineto{\pgfqpoint{0.445679in}{3.579691in}}%
\pgfusepath{stroke}%
\end{pgfscope}%
\begin{pgfscope}%
\pgfsetrectcap%
\pgfsetmiterjoin%
\pgfsetlinewidth{0.803000pt}%
\definecolor{currentstroke}{rgb}{0.000000,0.000000,0.000000}%
\pgfsetstrokecolor{currentstroke}%
\pgfsetdash{}{0pt}%
\pgfpathmoveto{\pgfqpoint{3.525679in}{0.499691in}}%
\pgfpathlineto{\pgfqpoint{3.525679in}{3.579691in}}%
\pgfusepath{stroke}%
\end{pgfscope}%
\begin{pgfscope}%
\pgfsetrectcap%
\pgfsetmiterjoin%
\pgfsetlinewidth{0.803000pt}%
\definecolor{currentstroke}{rgb}{0.000000,0.000000,0.000000}%
\pgfsetstrokecolor{currentstroke}%
\pgfsetdash{}{0pt}%
\pgfpathmoveto{\pgfqpoint{0.445679in}{0.499691in}}%
\pgfpathlineto{\pgfqpoint{3.525679in}{0.499691in}}%
\pgfusepath{stroke}%
\end{pgfscope}%
\begin{pgfscope}%
\pgfsetrectcap%
\pgfsetmiterjoin%
\pgfsetlinewidth{0.803000pt}%
\definecolor{currentstroke}{rgb}{0.000000,0.000000,0.000000}%
\pgfsetstrokecolor{currentstroke}%
\pgfsetdash{}{0pt}%
\pgfpathmoveto{\pgfqpoint{0.445679in}{3.579691in}}%
\pgfpathlineto{\pgfqpoint{3.525679in}{3.579691in}}%
\pgfusepath{stroke}%
\end{pgfscope}%
\begin{pgfscope}%
\pgfsetroundcap%
\pgfsetroundjoin%
\pgfsetlinewidth{1.003750pt}%
\definecolor{currentstroke}{rgb}{0.000000,0.000000,0.000000}%
\pgfsetstrokecolor{currentstroke}%
\pgfsetdash{}{0pt}%
\pgfpathmoveto{\pgfqpoint{0.850017in}{2.059317in}}%
\pgfpathquadraticcurveto{\pgfqpoint{1.022926in}{2.232226in}}{\pgfqpoint{1.184854in}{2.394155in}}%
\pgfusepath{stroke}%
\end{pgfscope}%
\begin{pgfscope}%
\pgfsetroundcap%
\pgfsetroundjoin%
\pgfsetlinewidth{1.003750pt}%
\definecolor{currentstroke}{rgb}{0.000000,0.000000,0.000000}%
\pgfsetstrokecolor{currentstroke}%
\pgfsetdash{}{0pt}%
\pgfpathmoveto{\pgfqpoint{1.125929in}{2.374513in}}%
\pgfpathlineto{\pgfqpoint{1.184854in}{2.394155in}}%
\pgfpathlineto{\pgfqpoint{1.165213in}{2.335230in}}%
\pgfusepath{stroke}%
\end{pgfscope}%
\begin{pgfscope}%
\definecolor{textcolor}{rgb}{0.121569,0.188235,0.501961}%
\pgfsetstrokecolor{textcolor}%
\pgfsetfillcolor{textcolor}%
\pgftext[x=1.022938in,y=2.232239in,left,base]{\color{textcolor}{\rmfamily\fontsize{10.000000}{12.000000}\selectfont\catcode`\^=\active\def^{\ifmmode\sp\else\^{}\fi}\catcode`\%=\active\def
\end{pgfscope}%
\begin{pgfscope}%
\pgfsetroundcap%
\pgfsetroundjoin%
\pgfsetlinewidth{1.003750pt}%
\definecolor{currentstroke}{rgb}{0.000000,0.000000,0.000000}%
\pgfsetstrokecolor{currentstroke}%
\pgfsetdash{}{0pt}%
\pgfpathmoveto{\pgfqpoint{1.613005in}{0.909246in}}%
\pgfpathquadraticcurveto{\pgfqpoint{1.793127in}{1.269490in}}{\pgfqpoint{1.966304in}{1.615844in}}%
\pgfusepath{stroke}%
\end{pgfscope}%
\begin{pgfscope}%
\pgfsetroundcap%
\pgfsetroundjoin%
\pgfsetlinewidth{1.003750pt}%
\definecolor{currentstroke}{rgb}{0.000000,0.000000,0.000000}%
\pgfsetstrokecolor{currentstroke}%
\pgfsetdash{}{0pt}%
\pgfpathmoveto{\pgfqpoint{1.916614in}{1.578577in}}%
\pgfpathlineto{\pgfqpoint{1.966304in}{1.615844in}}%
\pgfpathlineto{\pgfqpoint{1.966304in}{1.553731in}}%
\pgfusepath{stroke}%
\end{pgfscope}%
\begin{pgfscope}%
\definecolor{textcolor}{rgb}{0.121569,0.188235,0.501961}%
\pgfsetstrokecolor{textcolor}%
\pgfsetfillcolor{textcolor}%
\pgftext[x=1.793131in,y=1.269499in,left,base]{\color{textcolor}{\rmfamily\fontsize{10.000000}{12.000000}\selectfont\catcode`\^=\active\def^{\ifmmode\sp\else\^{}\fi}\catcode`\%=\active\def
\end{pgfscope}%
\begin{pgfscope}%
\pgfsetroundcap%
\pgfsetroundjoin%
\pgfsetlinewidth{1.003750pt}%
\definecolor{currentstroke}{rgb}{0.000000,0.000000,0.000000}%
\pgfsetstrokecolor{currentstroke}%
\pgfsetdash{}{0pt}%
\pgfpathmoveto{\pgfqpoint{1.238580in}{2.440183in}}%
\pgfpathquadraticcurveto{\pgfqpoint{1.793120in}{2.809876in}}{\pgfqpoint{2.334740in}{3.170956in}}%
\pgfusepath{stroke}%
\end{pgfscope}%
\begin{pgfscope}%
\pgfsetroundcap%
\pgfsetroundjoin%
\pgfsetlinewidth{1.003750pt}%
\definecolor{currentstroke}{rgb}{0.000000,0.000000,0.000000}%
\pgfsetstrokecolor{currentstroke}%
\pgfsetdash{}{0pt}%
\pgfpathmoveto{\pgfqpoint{2.273107in}{3.163252in}}%
\pgfpathlineto{\pgfqpoint{2.334740in}{3.170956in}}%
\pgfpathlineto{\pgfqpoint{2.303923in}{3.117027in}}%
\pgfusepath{stroke}%
\end{pgfscope}%
\begin{pgfscope}%
\definecolor{textcolor}{rgb}{0.121569,0.188235,0.501961}%
\pgfsetstrokecolor{textcolor}%
\pgfsetfillcolor{textcolor}%
\pgftext[x=1.793131in,y=2.809884in,left,base]{\color{textcolor}{\rmfamily\fontsize{10.000000}{12.000000}\selectfont\catcode`\^=\active\def^{\ifmmode\sp\else\^{}\fi}\catcode`\%=\active\def
\end{pgfscope}%
\begin{pgfscope}%
\pgfsetroundcap%
\pgfsetroundjoin%
\pgfsetlinewidth{1.003750pt}%
\definecolor{currentstroke}{rgb}{0.000000,0.000000,0.000000}%
\pgfsetstrokecolor{currentstroke}%
\pgfsetdash{}{0pt}%
\pgfpathmoveto{\pgfqpoint{1.992413in}{1.681531in}}%
\pgfpathquadraticcurveto{\pgfqpoint{2.178228in}{2.424792in}}{\pgfqpoint{2.360278in}{3.152988in}}%
\pgfusepath{stroke}%
\end{pgfscope}%
\begin{pgfscope}%
\pgfsetroundcap%
\pgfsetroundjoin%
\pgfsetlinewidth{1.003750pt}%
\definecolor{currentstroke}{rgb}{0.000000,0.000000,0.000000}%
\pgfsetstrokecolor{currentstroke}%
\pgfsetdash{}{0pt}%
\pgfpathmoveto{\pgfqpoint{2.319855in}{3.105829in}}%
\pgfpathlineto{\pgfqpoint{2.360278in}{3.152988in}}%
\pgfpathlineto{\pgfqpoint{2.373752in}{3.092355in}}%
\pgfusepath{stroke}%
\end{pgfscope}%
\begin{pgfscope}%
\definecolor{textcolor}{rgb}{0.121569,0.188235,0.501961}%
\pgfsetstrokecolor{textcolor}%
\pgfsetfillcolor{textcolor}%
\pgftext[x=2.178227in,y=2.424787in,left,base]{\color{textcolor}{\rmfamily\fontsize{10.000000}{12.000000}\selectfont\catcode`\^=\active\def^{\ifmmode\sp\else\^{}\fi}\catcode`\%=\active\def
\end{pgfscope}%
\begin{pgfscope}%
\pgfsetroundcap%
\pgfsetroundjoin%
\pgfsetlinewidth{1.003750pt}%
\definecolor{currentstroke}{rgb}{0.000000,0.000000,0.000000}%
\pgfsetstrokecolor{currentstroke}%
\pgfsetdash{}{0pt}%
\pgfpathmoveto{\pgfqpoint{1.242423in}{2.431521in}}%
\pgfpathquadraticcurveto{\pgfqpoint{1.985684in}{2.617337in}}{\pgfqpoint{2.713880in}{2.799386in}}%
\pgfusepath{stroke}%
\end{pgfscope}%
\begin{pgfscope}%
\pgfsetroundcap%
\pgfsetroundjoin%
\pgfsetlinewidth{1.003750pt}%
\definecolor{currentstroke}{rgb}{0.000000,0.000000,0.000000}%
\pgfsetstrokecolor{currentstroke}%
\pgfsetdash{}{0pt}%
\pgfpathmoveto{\pgfqpoint{2.653246in}{2.812860in}}%
\pgfpathlineto{\pgfqpoint{2.713880in}{2.799386in}}%
\pgfpathlineto{\pgfqpoint{2.666720in}{2.758963in}}%
\pgfusepath{stroke}%
\end{pgfscope}%
\begin{pgfscope}%
\definecolor{textcolor}{rgb}{0.121569,0.188235,0.501961}%
\pgfsetstrokecolor{textcolor}%
\pgfsetfillcolor{textcolor}%
\pgftext[x=1.985679in,y=2.617336in,left,base]{\color{textcolor}{\rmfamily\fontsize{10.000000}{12.000000}\selectfont\catcode`\^=\active\def^{\ifmmode\sp\else\^{}\fi}\catcode`\%=\active\def
\end{pgfscope}%
\begin{pgfscope}%
\pgfsetroundcap%
\pgfsetroundjoin%
\pgfsetlinewidth{1.003750pt}%
\definecolor{currentstroke}{rgb}{0.000000,0.000000,0.000000}%
\pgfsetstrokecolor{currentstroke}%
\pgfsetdash{}{0pt}%
\pgfpathmoveto{\pgfqpoint{2.001075in}{1.677688in}}%
\pgfpathquadraticcurveto{\pgfqpoint{2.370768in}{2.232228in}}{\pgfqpoint{2.731848in}{2.773848in}}%
\pgfusepath{stroke}%
\end{pgfscope}%
\begin{pgfscope}%
\pgfsetroundcap%
\pgfsetroundjoin%
\pgfsetlinewidth{1.003750pt}%
\definecolor{currentstroke}{rgb}{0.000000,0.000000,0.000000}%
\pgfsetstrokecolor{currentstroke}%
\pgfsetdash{}{0pt}%
\pgfpathmoveto{\pgfqpoint{2.677919in}{2.743031in}}%
\pgfpathlineto{\pgfqpoint{2.731848in}{2.773848in}}%
\pgfpathlineto{\pgfqpoint{2.724144in}{2.712215in}}%
\pgfusepath{stroke}%
\end{pgfscope}%
\begin{pgfscope}%
\definecolor{textcolor}{rgb}{0.121569,0.188235,0.501961}%
\pgfsetstrokecolor{textcolor}%
\pgfsetfillcolor{textcolor}%
\pgftext[x=2.370775in,y=2.232239in,left,base]{\color{textcolor}{\rmfamily\fontsize{10.000000}{12.000000}\selectfont\catcode`\^=\active\def^{\ifmmode\sp\else\^{}\fi}\catcode`\%=\active\def
\end{pgfscope}%
\begin{pgfscope}%
\pgfsetroundcap%
\pgfsetroundjoin%
\pgfsetlinewidth{1.003750pt}%
\definecolor{currentstroke}{rgb}{0.000000,0.000000,0.000000}%
\pgfsetstrokecolor{currentstroke}%
\pgfsetdash{}{0pt}%
\pgfpathmoveto{\pgfqpoint{1.627519in}{0.891136in}}%
\pgfpathquadraticcurveto{\pgfqpoint{2.370780in}{1.076952in}}{\pgfqpoint{3.098976in}{1.259001in}}%
\pgfusepath{stroke}%
\end{pgfscope}%
\begin{pgfscope}%
\pgfsetroundcap%
\pgfsetroundjoin%
\pgfsetlinewidth{1.003750pt}%
\definecolor{currentstroke}{rgb}{0.000000,0.000000,0.000000}%
\pgfsetstrokecolor{currentstroke}%
\pgfsetdash{}{0pt}%
\pgfpathmoveto{\pgfqpoint{3.038343in}{1.272475in}}%
\pgfpathlineto{\pgfqpoint{3.098976in}{1.259001in}}%
\pgfpathlineto{\pgfqpoint{3.051817in}{1.218578in}}%
\pgfusepath{stroke}%
\end{pgfscope}%
\begin{pgfscope}%
\definecolor{textcolor}{rgb}{0.121569,0.188235,0.501961}%
\pgfsetstrokecolor{textcolor}%
\pgfsetfillcolor{textcolor}%
\pgftext[x=2.370775in,y=1.076950in,left,base]{\color{textcolor}{\rmfamily\fontsize{10.000000}{12.000000}\selectfont\catcode`\^=\active\def^{\ifmmode\sp\else\^{}\fi}\catcode`\%=\active\def
\end{pgfscope}%
\begin{pgfscope}%
\pgfsetroundcap%
\pgfsetroundjoin%
\pgfsetlinewidth{1.003750pt}%
\definecolor{currentstroke}{rgb}{0.000000,0.000000,0.000000}%
\pgfsetstrokecolor{currentstroke}%
\pgfsetdash{}{0pt}%
\pgfpathmoveto{\pgfqpoint{0.845786in}{2.016598in}}%
\pgfpathquadraticcurveto{\pgfqpoint{1.215479in}{1.462058in}}{\pgfqpoint{1.576559in}{0.920438in}}%
\pgfusepath{stroke}%
\end{pgfscope}%
\begin{pgfscope}%
\pgfsetroundcap%
\pgfsetroundjoin%
\pgfsetlinewidth{1.003750pt}%
\definecolor{currentstroke}{rgb}{0.000000,0.000000,0.000000}%
\pgfsetstrokecolor{currentstroke}%
\pgfsetdash{}{0pt}%
\pgfpathmoveto{\pgfqpoint{1.568855in}{0.982071in}}%
\pgfpathlineto{\pgfqpoint{1.576559in}{0.920438in}}%
\pgfpathlineto{\pgfqpoint{1.522630in}{0.951254in}}%
\pgfusepath{stroke}%
\end{pgfscope}%
\begin{pgfscope}%
\definecolor{textcolor}{rgb}{0.501961,0.184314,0.133333}%
\pgfsetstrokecolor{textcolor}%
\pgfsetfillcolor{textcolor}%
\pgftext[x=1.215487in,y=1.462047in,left,base]{\color{textcolor}{\rmfamily\fontsize{10.000000}{12.000000}\selectfont\catcode`\^=\active\def^{\ifmmode\sp\else\^{}\fi}\catcode`\%=\active\def
\end{pgfscope}%
\begin{pgfscope}%
\pgfsetroundcap%
\pgfsetroundjoin%
\pgfsetlinewidth{1.003750pt}%
\definecolor{currentstroke}{rgb}{0.000000,0.000000,0.000000}%
\pgfsetstrokecolor{currentstroke}%
\pgfsetdash{}{0pt}%
\pgfpathmoveto{\pgfqpoint{1.222221in}{2.397851in}}%
\pgfpathquadraticcurveto{\pgfqpoint{1.408036in}{1.654590in}}{\pgfqpoint{1.590085in}{0.926394in}}%
\pgfusepath{stroke}%
\end{pgfscope}%
\begin{pgfscope}%
\pgfsetroundcap%
\pgfsetroundjoin%
\pgfsetlinewidth{1.003750pt}%
\definecolor{currentstroke}{rgb}{0.000000,0.000000,0.000000}%
\pgfsetstrokecolor{currentstroke}%
\pgfsetdash{}{0pt}%
\pgfpathmoveto{\pgfqpoint{1.603559in}{0.987028in}}%
\pgfpathlineto{\pgfqpoint{1.590085in}{0.926394in}}%
\pgfpathlineto{\pgfqpoint{1.549662in}{0.973554in}}%
\pgfusepath{stroke}%
\end{pgfscope}%
\begin{pgfscope}%
\definecolor{textcolor}{rgb}{0.501961,0.184314,0.133333}%
\pgfsetstrokecolor{textcolor}%
\pgfsetfillcolor{textcolor}%
\pgftext[x=1.408035in,y=1.654595in,left,base]{\color{textcolor}{\rmfamily\fontsize{10.000000}{12.000000}\selectfont\catcode`\^=\active\def^{\ifmmode\sp\else\^{}\fi}\catcode`\%=\active\def
\end{pgfscope}%
\begin{pgfscope}%
\pgfsetroundcap%
\pgfsetroundjoin%
\pgfsetlinewidth{1.003750pt}%
\definecolor{currentstroke}{rgb}{0.000000,0.000000,0.000000}%
\pgfsetstrokecolor{currentstroke}%
\pgfsetdash{}{0pt}%
\pgfpathmoveto{\pgfqpoint{0.856727in}{2.030912in}}%
\pgfpathquadraticcurveto{\pgfqpoint{1.408025in}{1.847146in}}{\pgfqpoint{1.944591in}{1.668291in}}%
\pgfusepath{stroke}%
\end{pgfscope}%
\begin{pgfscope}%
\pgfsetroundcap%
\pgfsetroundjoin%
\pgfsetlinewidth{1.003750pt}%
\definecolor{currentstroke}{rgb}{0.000000,0.000000,0.000000}%
\pgfsetstrokecolor{currentstroke}%
\pgfsetdash{}{0pt}%
\pgfpathmoveto{\pgfqpoint{1.900671in}{1.712211in}}%
\pgfpathlineto{\pgfqpoint{1.944591in}{1.668291in}}%
\pgfpathlineto{\pgfqpoint{1.883103in}{1.659507in}}%
\pgfusepath{stroke}%
\end{pgfscope}%
\begin{pgfscope}%
\definecolor{textcolor}{rgb}{0.501961,0.184314,0.133333}%
\pgfsetstrokecolor{textcolor}%
\pgfsetfillcolor{textcolor}%
\pgftext[x=1.408035in,y=1.847143in,left,base]{\color{textcolor}{\rmfamily\fontsize{10.000000}{12.000000}\selectfont\catcode`\^=\active\def^{\ifmmode\sp\else\^{}\fi}\catcode`\%=\active\def
\end{pgfscope}%
\begin{pgfscope}%
\pgfsetroundcap%
\pgfsetroundjoin%
\pgfsetlinewidth{1.003750pt}%
\definecolor{currentstroke}{rgb}{0.000000,0.000000,0.000000}%
\pgfsetstrokecolor{currentstroke}%
\pgfsetdash{}{0pt}%
\pgfpathmoveto{\pgfqpoint{1.235113in}{2.405161in}}%
\pgfpathquadraticcurveto{\pgfqpoint{1.600581in}{2.039693in}}{\pgfqpoint{1.955069in}{1.685205in}}%
\pgfusepath{stroke}%
\end{pgfscope}%
\begin{pgfscope}%
\pgfsetroundcap%
\pgfsetroundjoin%
\pgfsetlinewidth{1.003750pt}%
\definecolor{currentstroke}{rgb}{0.000000,0.000000,0.000000}%
\pgfsetstrokecolor{currentstroke}%
\pgfsetdash{}{0pt}%
\pgfpathmoveto{\pgfqpoint{1.935427in}{1.744131in}}%
\pgfpathlineto{\pgfqpoint{1.955069in}{1.685205in}}%
\pgfpathlineto{\pgfqpoint{1.896143in}{1.704847in}}%
\pgfusepath{stroke}%
\end{pgfscope}%
\begin{pgfscope}%
\definecolor{textcolor}{rgb}{0.501961,0.184314,0.133333}%
\pgfsetstrokecolor{textcolor}%
\pgfsetfillcolor{textcolor}%
\pgftext[x=1.600583in,y=2.039691in,left,base]{\color{textcolor}{\rmfamily\fontsize{10.000000}{12.000000}\selectfont\catcode`\^=\active\def^{\ifmmode\sp\else\^{}\fi}\catcode`\%=\active\def
\end{pgfscope}%
\begin{pgfscope}%
\pgfsetroundcap%
\pgfsetroundjoin%
\pgfsetlinewidth{1.003750pt}%
\definecolor{currentstroke}{rgb}{0.000000,0.000000,0.000000}%
\pgfsetstrokecolor{currentstroke}%
\pgfsetdash{}{0pt}%
\pgfpathmoveto{\pgfqpoint{2.390402in}{3.175354in}}%
\pgfpathquadraticcurveto{\pgfqpoint{2.563311in}{3.002445in}}{\pgfqpoint{2.725240in}{2.840516in}}%
\pgfusepath{stroke}%
\end{pgfscope}%
\begin{pgfscope}%
\pgfsetroundcap%
\pgfsetroundjoin%
\pgfsetlinewidth{1.003750pt}%
\definecolor{currentstroke}{rgb}{0.000000,0.000000,0.000000}%
\pgfsetstrokecolor{currentstroke}%
\pgfsetdash{}{0pt}%
\pgfpathmoveto{\pgfqpoint{2.705598in}{2.899441in}}%
\pgfpathlineto{\pgfqpoint{2.725240in}{2.840516in}}%
\pgfpathlineto{\pgfqpoint{2.666314in}{2.860158in}}%
\pgfusepath{stroke}%
\end{pgfscope}%
\begin{pgfscope}%
\definecolor{textcolor}{rgb}{0.501961,0.184314,0.133333}%
\pgfsetstrokecolor{textcolor}%
\pgfsetfillcolor{textcolor}%
\pgftext[x=2.563324in,y=3.002432in,left,base]{\color{textcolor}{\rmfamily\fontsize{10.000000}{12.000000}\selectfont\catcode`\^=\active\def^{\ifmmode\sp\else\^{}\fi}\catcode`\%=\active\def
\end{pgfscope}%
\begin{pgfscope}%
\pgfsetroundcap%
\pgfsetroundjoin%
\pgfsetlinewidth{1.003750pt}%
\definecolor{currentstroke}{rgb}{0.000000,0.000000,0.000000}%
\pgfsetstrokecolor{currentstroke}%
\pgfsetdash{}{0pt}%
\pgfpathmoveto{\pgfqpoint{2.012016in}{1.645816in}}%
\pgfpathquadraticcurveto{\pgfqpoint{2.563314in}{1.462050in}}{\pgfqpoint{3.099880in}{1.283195in}}%
\pgfusepath{stroke}%
\end{pgfscope}%
\begin{pgfscope}%
\pgfsetroundcap%
\pgfsetroundjoin%
\pgfsetlinewidth{1.003750pt}%
\definecolor{currentstroke}{rgb}{0.000000,0.000000,0.000000}%
\pgfsetstrokecolor{currentstroke}%
\pgfsetdash{}{0pt}%
\pgfpathmoveto{\pgfqpoint{3.055960in}{1.327115in}}%
\pgfpathlineto{\pgfqpoint{3.099880in}{1.283195in}}%
\pgfpathlineto{\pgfqpoint{3.038391in}{1.274410in}}%
\pgfusepath{stroke}%
\end{pgfscope}%
\begin{pgfscope}%
\definecolor{textcolor}{rgb}{0.501961,0.184314,0.133333}%
\pgfsetstrokecolor{textcolor}%
\pgfsetfillcolor{textcolor}%
\pgftext[x=2.563324in,y=1.462047in,left,base]{\color{textcolor}{\rmfamily\fontsize{10.000000}{12.000000}\selectfont\catcode`\^=\active\def^{\ifmmode\sp\else\^{}\fi}\catcode`\%=\active\def
\end{pgfscope}%
\begin{pgfscope}%
\pgfsetroundcap%
\pgfsetroundjoin%
\pgfsetlinewidth{1.003750pt}%
\definecolor{currentstroke}{rgb}{0.000000,0.000000,0.000000}%
\pgfsetstrokecolor{currentstroke}%
\pgfsetdash{}{0pt}%
\pgfpathmoveto{\pgfqpoint{2.762606in}{2.782948in}}%
\pgfpathquadraticcurveto{\pgfqpoint{2.948421in}{2.039687in}}{\pgfqpoint{3.130470in}{1.311490in}}%
\pgfusepath{stroke}%
\end{pgfscope}%
\begin{pgfscope}%
\pgfsetroundcap%
\pgfsetroundjoin%
\pgfsetlinewidth{1.003750pt}%
\definecolor{currentstroke}{rgb}{0.000000,0.000000,0.000000}%
\pgfsetstrokecolor{currentstroke}%
\pgfsetdash{}{0pt}%
\pgfpathmoveto{\pgfqpoint{3.143944in}{1.372124in}}%
\pgfpathlineto{\pgfqpoint{3.130470in}{1.311490in}}%
\pgfpathlineto{\pgfqpoint{3.090048in}{1.358650in}}%
\pgfusepath{stroke}%
\end{pgfscope}%
\begin{pgfscope}%
\definecolor{textcolor}{rgb}{0.501961,0.184314,0.133333}%
\pgfsetstrokecolor{textcolor}%
\pgfsetfillcolor{textcolor}%
\pgftext[x=2.948420in,y=2.039691in,left,base]{\color{textcolor}{\rmfamily\fontsize{10.000000}{12.000000}\selectfont\catcode`\^=\active\def^{\ifmmode\sp\else\^{}\fi}\catcode`\%=\active\def
\end{pgfscope}%
\begin{pgfscope}%
\pgfpathrectangle{\pgfqpoint{0.445679in}{0.499691in}}{\pgfqpoint{3.080000in}{3.080000in}}%
\pgfusepath{clip}%
\pgfsetbuttcap%
\pgfsetroundjoin%
\definecolor{currentfill}{rgb}{0.000000,0.000000,0.000000}%
\pgfsetfillcolor{currentfill}%
\pgfsetlinewidth{1.003750pt}%
\definecolor{currentstroke}{rgb}{0.000000,0.000000,0.000000}%
\pgfsetstrokecolor{currentstroke}%
\pgfsetdash{}{0pt}%
\pgfsys@defobject{currentmarker}{\pgfqpoint{-0.069444in}{-0.069444in}}{\pgfqpoint{0.069444in}{0.069444in}}{%
\pgfpathmoveto{\pgfqpoint{-0.069444in}{-0.069444in}}%
\pgfpathlineto{\pgfqpoint{0.069444in}{-0.069444in}}%
\pgfpathlineto{\pgfqpoint{0.069444in}{0.069444in}}%
\pgfpathlineto{\pgfqpoint{-0.069444in}{0.069444in}}%
\pgfpathlineto{\pgfqpoint{-0.069444in}{-0.069444in}}%
\pgfpathclose%
\pgfusepath{stroke,fill}%
}%
\begin{pgfscope}%
\pgfsys@transformshift{0.830390in}{2.039691in}%
\pgfsys@useobject{currentmarker}{}%
\end{pgfscope}%
\begin{pgfscope}%
\pgfsys@transformshift{1.215487in}{2.424787in}%
\pgfsys@useobject{currentmarker}{}%
\end{pgfscope}%
\begin{pgfscope}%
\pgfsys@transformshift{1.600583in}{0.884402in}%
\pgfsys@useobject{currentmarker}{}%
\end{pgfscope}%
\begin{pgfscope}%
\pgfsys@transformshift{1.985679in}{1.654595in}%
\pgfsys@useobject{currentmarker}{}%
\end{pgfscope}%
\begin{pgfscope}%
\pgfsys@transformshift{2.370775in}{3.194980in}%
\pgfsys@useobject{currentmarker}{}%
\end{pgfscope}%
\begin{pgfscope}%
\pgfsys@transformshift{2.755872in}{2.809884in}%
\pgfsys@useobject{currentmarker}{}%
\end{pgfscope}%
\begin{pgfscope}%
\pgfsys@transformshift{3.140968in}{1.269499in}%
\pgfsys@useobject{currentmarker}{}%
\end{pgfscope}%
\end{pgfscope}%
\begin{pgfscope}%
\definecolor{textcolor}{rgb}{0.000000,0.000000,0.000000}%
\pgfsetstrokecolor{textcolor}%
\pgfsetfillcolor{textcolor}%
\pgftext[x=1.985679in,y=3.663024in,,base]{\color{textcolor}{\rmfamily\fontsize{12.000000}{14.400000}\selectfont\catcode`\^=\active\def^{\ifmmode\sp\else\^{}\fi}\catcode`\%=\active\def
\end{pgfscope}%
\end{pgfpicture}%
\makeatother%
\endgroup%

%% file: figs/fp-lemma31.pgf
\begingroup%
\makeatletter%
\begin{pgfpicture}%
\pgfpathrectangle{\pgfpointorigin}{\pgfqpoint{4.900000in}{4.900000in}}%
\pgfusepath{use as bounding box, clip}%
\begin{pgfscope}%
\pgfsetbuttcap%
\pgfsetmiterjoin%
\definecolor{currentfill}{rgb}{1.000000,1.000000,1.000000}%
\pgfsetfillcolor{currentfill}%
\pgfsetlinewidth{0.000000pt}%
\definecolor{currentstroke}{rgb}{1.000000,1.000000,1.000000}%
\pgfsetstrokecolor{currentstroke}%
\pgfsetdash{}{0pt}%
\pgfpathmoveto{\pgfqpoint{0.000000in}{0.000000in}}%
\pgfpathlineto{\pgfqpoint{4.900000in}{0.000000in}}%
\pgfpathlineto{\pgfqpoint{4.900000in}{4.900000in}}%
\pgfpathlineto{\pgfqpoint{0.000000in}{4.900000in}}%
\pgfpathlineto{\pgfqpoint{0.000000in}{0.000000in}}%
\pgfpathclose%
\pgfusepath{fill}%
\end{pgfscope}%
\begin{pgfscope}%
\pgfsetbuttcap%
\pgfsetmiterjoin%
\definecolor{currentfill}{rgb}{1.000000,1.000000,1.000000}%
\pgfsetfillcolor{currentfill}%
\pgfsetlinewidth{0.000000pt}%
\definecolor{currentstroke}{rgb}{0.000000,0.000000,0.000000}%
\pgfsetstrokecolor{currentstroke}%
\pgfsetstrokeopacity{0.000000}%
\pgfsetdash{}{0pt}%
\pgfpathmoveto{\pgfqpoint{0.100000in}{0.100000in}}%
\pgfpathlineto{\pgfqpoint{4.800000in}{0.100000in}}%
\pgfpathlineto{\pgfqpoint{4.800000in}{4.800000in}}%
\pgfpathlineto{\pgfqpoint{0.100000in}{4.800000in}}%
\pgfpathlineto{\pgfqpoint{0.100000in}{0.100000in}}%
\pgfpathclose%
\pgfusepath{fill}%
\end{pgfscope}%
\begin{pgfscope}%
\pgfpathrectangle{\pgfqpoint{0.100000in}{0.100000in}}{\pgfqpoint{4.700000in}{4.700000in}}%
\pgfusepath{clip}%
\pgfsetbuttcap%
\pgfsetroundjoin%
\definecolor{currentfill}{rgb}{1.000000,1.000000,1.000000}%
\pgfsetfillcolor{currentfill}%
\pgfsetfillopacity{0.100000}%
\pgfsetlinewidth{1.003750pt}%
\definecolor{currentstroke}{rgb}{0.000000,0.000000,0.000000}%
\pgfsetstrokecolor{currentstroke}%
\pgfsetdash{}{0pt}%
\pgfpathmoveto{\pgfqpoint{2.513514in}{2.597116in}}%
\pgfpathlineto{\pgfqpoint{2.898464in}{2.404641in}}%
\pgfpathlineto{\pgfqpoint{2.898464in}{4.290503in}}%
\pgfpathlineto{\pgfqpoint{2.513514in}{4.482978in}}%
\pgfpathlineto{\pgfqpoint{2.513514in}{2.597116in}}%
\pgfpathclose%
\pgfusepath{stroke,fill}%
\end{pgfscope}%
\begin{pgfscope}%
\pgfpathrectangle{\pgfqpoint{0.100000in}{0.100000in}}{\pgfqpoint{4.700000in}{4.700000in}}%
\pgfusepath{clip}%
\pgfsetbuttcap%
\pgfsetroundjoin%
\definecolor{currentfill}{rgb}{1.000000,1.000000,1.000000}%
\pgfsetfillcolor{currentfill}%
\pgfsetfillopacity{0.100000}%
\pgfsetlinewidth{1.003750pt}%
\definecolor{currentstroke}{rgb}{0.000000,0.000000,0.000000}%
\pgfsetstrokecolor{currentstroke}%
\pgfsetdash{}{0pt}%
\pgfpathmoveto{\pgfqpoint{2.513514in}{2.597116in}}%
\pgfpathlineto{\pgfqpoint{0.973713in}{1.827216in}}%
\pgfpathlineto{\pgfqpoint{1.358663in}{1.634741in}}%
\pgfpathlineto{\pgfqpoint{2.898464in}{2.404641in}}%
\pgfpathlineto{\pgfqpoint{2.513514in}{2.597116in}}%
\pgfpathclose%
\pgfusepath{stroke,fill}%
\end{pgfscope}%
\begin{pgfscope}%
\pgfpathrectangle{\pgfqpoint{0.100000in}{0.100000in}}{\pgfqpoint{4.700000in}{4.700000in}}%
\pgfusepath{clip}%
\pgfsetbuttcap%
\pgfsetroundjoin%
\definecolor{currentfill}{rgb}{1.000000,1.000000,1.000000}%
\pgfsetfillcolor{currentfill}%
\pgfsetfillopacity{0.100000}%
\pgfsetlinewidth{1.003750pt}%
\definecolor{currentstroke}{rgb}{0.000000,0.000000,0.000000}%
\pgfsetstrokecolor{currentstroke}%
\pgfsetdash{}{0pt}%
\pgfpathmoveto{\pgfqpoint{2.513514in}{2.597116in}}%
\pgfpathlineto{\pgfqpoint{0.973713in}{1.827216in}}%
\pgfpathlineto{\pgfqpoint{0.973713in}{3.713078in}}%
\pgfpathlineto{\pgfqpoint{2.513514in}{4.482978in}}%
\pgfpathlineto{\pgfqpoint{2.513514in}{2.597116in}}%
\pgfpathclose%
\pgfusepath{stroke,fill}%
\end{pgfscope}%
\begin{pgfscope}%
\pgfpathrectangle{\pgfqpoint{0.100000in}{0.100000in}}{\pgfqpoint{4.700000in}{4.700000in}}%
\pgfusepath{clip}%
\pgfsetbuttcap%
\pgfsetroundjoin%
\definecolor{currentfill}{rgb}{1.000000,1.000000,1.000000}%
\pgfsetfillcolor{currentfill}%
\pgfsetfillopacity{0.100000}%
\pgfsetlinewidth{1.003750pt}%
\definecolor{currentstroke}{rgb}{0.000000,0.000000,0.000000}%
\pgfsetstrokecolor{currentstroke}%
\pgfsetdash{}{0pt}%
\pgfpathmoveto{\pgfqpoint{0.973713in}{1.827216in}}%
\pgfpathlineto{\pgfqpoint{1.358663in}{1.634741in}}%
\pgfpathlineto{\pgfqpoint{1.358663in}{3.520603in}}%
\pgfpathlineto{\pgfqpoint{0.973713in}{3.713078in}}%
\pgfpathlineto{\pgfqpoint{0.973713in}{1.827216in}}%
\pgfpathclose%
\pgfusepath{stroke,fill}%
\end{pgfscope}%
\begin{pgfscope}%
\pgfpathrectangle{\pgfqpoint{0.100000in}{0.100000in}}{\pgfqpoint{4.700000in}{4.700000in}}%
\pgfusepath{clip}%
\pgfsetbuttcap%
\pgfsetroundjoin%
\definecolor{currentfill}{rgb}{1.000000,1.000000,1.000000}%
\pgfsetfillcolor{currentfill}%
\pgfsetfillopacity{0.100000}%
\pgfsetlinewidth{1.003750pt}%
\definecolor{currentstroke}{rgb}{0.000000,0.000000,0.000000}%
\pgfsetstrokecolor{currentstroke}%
\pgfsetdash{}{0pt}%
\pgfpathmoveto{\pgfqpoint{2.898464in}{2.404641in}}%
\pgfpathlineto{\pgfqpoint{1.358663in}{1.634741in}}%
\pgfpathlineto{\pgfqpoint{1.358663in}{3.520603in}}%
\pgfpathlineto{\pgfqpoint{2.898464in}{4.290503in}}%
\pgfpathlineto{\pgfqpoint{2.898464in}{2.404641in}}%
\pgfpathclose%
\pgfusepath{stroke,fill}%
\end{pgfscope}%
\begin{pgfscope}%
\pgfpathrectangle{\pgfqpoint{0.100000in}{0.100000in}}{\pgfqpoint{4.700000in}{4.700000in}}%
\pgfusepath{clip}%
\pgfsetbuttcap%
\pgfsetroundjoin%
\definecolor{currentfill}{rgb}{1.000000,1.000000,1.000000}%
\pgfsetfillcolor{currentfill}%
\pgfsetfillopacity{0.100000}%
\pgfsetlinewidth{1.003750pt}%
\definecolor{currentstroke}{rgb}{0.000000,0.000000,0.000000}%
\pgfsetstrokecolor{currentstroke}%
\pgfsetdash{}{0pt}%
\pgfpathmoveto{\pgfqpoint{2.513514in}{4.482978in}}%
\pgfpathlineto{\pgfqpoint{0.973713in}{3.713078in}}%
\pgfpathlineto{\pgfqpoint{1.358663in}{3.520603in}}%
\pgfpathlineto{\pgfqpoint{2.898464in}{4.290503in}}%
\pgfpathlineto{\pgfqpoint{2.513514in}{4.482978in}}%
\pgfpathclose%
\pgfusepath{stroke,fill}%
\end{pgfscope}%
\begin{pgfscope}%
\pgfpathrectangle{\pgfqpoint{0.100000in}{0.100000in}}{\pgfqpoint{4.700000in}{4.700000in}}%
\pgfusepath{clip}%
\pgfsetbuttcap%
\pgfsetroundjoin%
\definecolor{currentfill}{rgb}{0.000000,0.000000,0.000000}%
\pgfsetfillcolor{currentfill}%
\pgfsetfillopacity{0.100000}%
\pgfsetlinewidth{1.003750pt}%
\definecolor{currentstroke}{rgb}{0.000000,0.000000,0.000000}%
\pgfsetstrokecolor{currentstroke}%
\pgfsetdash{}{0pt}%
\pgfpathmoveto{\pgfqpoint{2.898464in}{2.404641in}}%
\pgfpathlineto{\pgfqpoint{3.668364in}{2.019691in}}%
\pgfpathlineto{\pgfqpoint{3.668364in}{3.905553in}}%
\pgfpathlineto{\pgfqpoint{2.898464in}{4.290503in}}%
\pgfpathlineto{\pgfqpoint{2.898464in}{2.404641in}}%
\pgfpathclose%
\pgfusepath{stroke,fill}%
\end{pgfscope}%
\begin{pgfscope}%
\pgfpathrectangle{\pgfqpoint{0.100000in}{0.100000in}}{\pgfqpoint{4.700000in}{4.700000in}}%
\pgfusepath{clip}%
\pgfsetbuttcap%
\pgfsetroundjoin%
\definecolor{currentfill}{rgb}{0.000000,0.000000,0.000000}%
\pgfsetfillcolor{currentfill}%
\pgfsetfillopacity{0.100000}%
\pgfsetlinewidth{1.003750pt}%
\definecolor{currentstroke}{rgb}{0.000000,0.000000,0.000000}%
\pgfsetstrokecolor{currentstroke}%
\pgfsetdash{}{0pt}%
\pgfpathmoveto{\pgfqpoint{2.898464in}{2.404641in}}%
\pgfpathlineto{\pgfqpoint{1.358663in}{1.634741in}}%
\pgfpathlineto{\pgfqpoint{2.128563in}{1.249791in}}%
\pgfpathlineto{\pgfqpoint{3.668364in}{2.019691in}}%
\pgfpathlineto{\pgfqpoint{2.898464in}{2.404641in}}%
\pgfpathclose%
\pgfusepath{stroke,fill}%
\end{pgfscope}%
\begin{pgfscope}%
\pgfpathrectangle{\pgfqpoint{0.100000in}{0.100000in}}{\pgfqpoint{4.700000in}{4.700000in}}%
\pgfusepath{clip}%
\pgfsetbuttcap%
\pgfsetroundjoin%
\definecolor{currentfill}{rgb}{0.000000,0.000000,0.000000}%
\pgfsetfillcolor{currentfill}%
\pgfsetfillopacity{0.100000}%
\pgfsetlinewidth{1.003750pt}%
\definecolor{currentstroke}{rgb}{0.000000,0.000000,0.000000}%
\pgfsetstrokecolor{currentstroke}%
\pgfsetdash{}{0pt}%
\pgfpathmoveto{\pgfqpoint{2.898464in}{2.404641in}}%
\pgfpathlineto{\pgfqpoint{1.358663in}{1.634741in}}%
\pgfpathlineto{\pgfqpoint{1.358663in}{3.520603in}}%
\pgfpathlineto{\pgfqpoint{2.898464in}{4.290503in}}%
\pgfpathlineto{\pgfqpoint{2.898464in}{2.404641in}}%
\pgfpathclose%
\pgfusepath{stroke,fill}%
\end{pgfscope}%
\begin{pgfscope}%
\pgfpathrectangle{\pgfqpoint{0.100000in}{0.100000in}}{\pgfqpoint{4.700000in}{4.700000in}}%
\pgfusepath{clip}%
\pgfsetbuttcap%
\pgfsetroundjoin%
\definecolor{currentfill}{rgb}{0.000000,0.000000,0.000000}%
\pgfsetfillcolor{currentfill}%
\pgfsetfillopacity{0.100000}%
\pgfsetlinewidth{1.003750pt}%
\definecolor{currentstroke}{rgb}{0.000000,0.000000,0.000000}%
\pgfsetstrokecolor{currentstroke}%
\pgfsetdash{}{0pt}%
\pgfpathmoveto{\pgfqpoint{1.358663in}{1.634741in}}%
\pgfpathlineto{\pgfqpoint{2.128563in}{1.249791in}}%
\pgfpathlineto{\pgfqpoint{2.128563in}{3.135653in}}%
\pgfpathlineto{\pgfqpoint{1.358663in}{3.520603in}}%
\pgfpathlineto{\pgfqpoint{1.358663in}{1.634741in}}%
\pgfpathclose%
\pgfusepath{stroke,fill}%
\end{pgfscope}%
\begin{pgfscope}%
\pgfpathrectangle{\pgfqpoint{0.100000in}{0.100000in}}{\pgfqpoint{4.700000in}{4.700000in}}%
\pgfusepath{clip}%
\pgfsetbuttcap%
\pgfsetroundjoin%
\definecolor{currentfill}{rgb}{0.000000,0.000000,0.000000}%
\pgfsetfillcolor{currentfill}%
\pgfsetfillopacity{0.100000}%
\pgfsetlinewidth{1.003750pt}%
\definecolor{currentstroke}{rgb}{0.000000,0.000000,0.000000}%
\pgfsetstrokecolor{currentstroke}%
\pgfsetdash{}{0pt}%
\pgfpathmoveto{\pgfqpoint{3.668364in}{2.019691in}}%
\pgfpathlineto{\pgfqpoint{2.128563in}{1.249791in}}%
\pgfpathlineto{\pgfqpoint{2.128563in}{3.135653in}}%
\pgfpathlineto{\pgfqpoint{3.668364in}{3.905553in}}%
\pgfpathlineto{\pgfqpoint{3.668364in}{2.019691in}}%
\pgfpathclose%
\pgfusepath{stroke,fill}%
\end{pgfscope}%
\begin{pgfscope}%
\pgfpathrectangle{\pgfqpoint{0.100000in}{0.100000in}}{\pgfqpoint{4.700000in}{4.700000in}}%
\pgfusepath{clip}%
\pgfsetbuttcap%
\pgfsetroundjoin%
\definecolor{currentfill}{rgb}{0.000000,0.000000,0.000000}%
\pgfsetfillcolor{currentfill}%
\pgfsetfillopacity{0.100000}%
\pgfsetlinewidth{1.003750pt}%
\definecolor{currentstroke}{rgb}{0.000000,0.000000,0.000000}%
\pgfsetstrokecolor{currentstroke}%
\pgfsetdash{}{0pt}%
\pgfpathmoveto{\pgfqpoint{2.898464in}{4.290503in}}%
\pgfpathlineto{\pgfqpoint{1.358663in}{3.520603in}}%
\pgfpathlineto{\pgfqpoint{2.128563in}{3.135653in}}%
\pgfpathlineto{\pgfqpoint{3.668364in}{3.905553in}}%
\pgfpathlineto{\pgfqpoint{2.898464in}{4.290503in}}%
\pgfpathclose%
\pgfusepath{stroke,fill}%
\end{pgfscope}%
\begin{pgfscope}%
\pgfpathrectangle{\pgfqpoint{0.100000in}{0.100000in}}{\pgfqpoint{4.700000in}{4.700000in}}%
\pgfusepath{clip}%
\pgfsetbuttcap%
\pgfsetroundjoin%
\definecolor{currentfill}{rgb}{1.000000,1.000000,1.000000}%
\pgfsetfillcolor{currentfill}%
\pgfsetfillopacity{0.100000}%
\pgfsetlinewidth{1.003750pt}%
\definecolor{currentstroke}{rgb}{0.000000,0.000000,0.000000}%
\pgfsetstrokecolor{currentstroke}%
\pgfsetdash{}{0pt}%
\pgfpathmoveto{\pgfqpoint{3.668364in}{2.019691in}}%
\pgfpathlineto{\pgfqpoint{4.053314in}{1.827216in}}%
\pgfpathlineto{\pgfqpoint{4.053314in}{3.713078in}}%
\pgfpathlineto{\pgfqpoint{3.668364in}{3.905553in}}%
\pgfpathlineto{\pgfqpoint{3.668364in}{2.019691in}}%
\pgfpathclose%
\pgfusepath{stroke,fill}%
\end{pgfscope}%
\begin{pgfscope}%
\pgfpathrectangle{\pgfqpoint{0.100000in}{0.100000in}}{\pgfqpoint{4.700000in}{4.700000in}}%
\pgfusepath{clip}%
\pgfsetbuttcap%
\pgfsetroundjoin%
\definecolor{currentfill}{rgb}{1.000000,1.000000,1.000000}%
\pgfsetfillcolor{currentfill}%
\pgfsetfillopacity{0.100000}%
\pgfsetlinewidth{1.003750pt}%
\definecolor{currentstroke}{rgb}{0.000000,0.000000,0.000000}%
\pgfsetstrokecolor{currentstroke}%
\pgfsetdash{}{0pt}%
\pgfpathmoveto{\pgfqpoint{3.668364in}{2.019691in}}%
\pgfpathlineto{\pgfqpoint{2.128563in}{1.249791in}}%
\pgfpathlineto{\pgfqpoint{2.513514in}{1.057316in}}%
\pgfpathlineto{\pgfqpoint{4.053314in}{1.827216in}}%
\pgfpathlineto{\pgfqpoint{3.668364in}{2.019691in}}%
\pgfpathclose%
\pgfusepath{stroke,fill}%
\end{pgfscope}%
\begin{pgfscope}%
\pgfpathrectangle{\pgfqpoint{0.100000in}{0.100000in}}{\pgfqpoint{4.700000in}{4.700000in}}%
\pgfusepath{clip}%
\pgfsetbuttcap%
\pgfsetroundjoin%
\definecolor{currentfill}{rgb}{1.000000,1.000000,1.000000}%
\pgfsetfillcolor{currentfill}%
\pgfsetfillopacity{0.100000}%
\pgfsetlinewidth{1.003750pt}%
\definecolor{currentstroke}{rgb}{0.000000,0.000000,0.000000}%
\pgfsetstrokecolor{currentstroke}%
\pgfsetdash{}{0pt}%
\pgfpathmoveto{\pgfqpoint{3.668364in}{2.019691in}}%
\pgfpathlineto{\pgfqpoint{2.128563in}{1.249791in}}%
\pgfpathlineto{\pgfqpoint{2.128563in}{3.135653in}}%
\pgfpathlineto{\pgfqpoint{3.668364in}{3.905553in}}%
\pgfpathlineto{\pgfqpoint{3.668364in}{2.019691in}}%
\pgfpathclose%
\pgfusepath{stroke,fill}%
\end{pgfscope}%
\begin{pgfscope}%
\pgfpathrectangle{\pgfqpoint{0.100000in}{0.100000in}}{\pgfqpoint{4.700000in}{4.700000in}}%
\pgfusepath{clip}%
\pgfsetbuttcap%
\pgfsetroundjoin%
\definecolor{currentfill}{rgb}{1.000000,1.000000,1.000000}%
\pgfsetfillcolor{currentfill}%
\pgfsetfillopacity{0.100000}%
\pgfsetlinewidth{1.003750pt}%
\definecolor{currentstroke}{rgb}{0.000000,0.000000,0.000000}%
\pgfsetstrokecolor{currentstroke}%
\pgfsetdash{}{0pt}%
\pgfpathmoveto{\pgfqpoint{2.128563in}{1.249791in}}%
\pgfpathlineto{\pgfqpoint{2.513514in}{1.057316in}}%
\pgfpathlineto{\pgfqpoint{2.513514in}{2.943178in}}%
\pgfpathlineto{\pgfqpoint{2.128563in}{3.135653in}}%
\pgfpathlineto{\pgfqpoint{2.128563in}{1.249791in}}%
\pgfpathclose%
\pgfusepath{stroke,fill}%
\end{pgfscope}%
\begin{pgfscope}%
\pgfpathrectangle{\pgfqpoint{0.100000in}{0.100000in}}{\pgfqpoint{4.700000in}{4.700000in}}%
\pgfusepath{clip}%
\pgfsetbuttcap%
\pgfsetroundjoin%
\definecolor{currentfill}{rgb}{1.000000,1.000000,1.000000}%
\pgfsetfillcolor{currentfill}%
\pgfsetfillopacity{0.100000}%
\pgfsetlinewidth{1.003750pt}%
\definecolor{currentstroke}{rgb}{0.000000,0.000000,0.000000}%
\pgfsetstrokecolor{currentstroke}%
\pgfsetdash{}{0pt}%
\pgfpathmoveto{\pgfqpoint{4.053314in}{1.827216in}}%
\pgfpathlineto{\pgfqpoint{2.513514in}{1.057316in}}%
\pgfpathlineto{\pgfqpoint{2.513514in}{2.943178in}}%
\pgfpathlineto{\pgfqpoint{4.053314in}{3.713078in}}%
\pgfpathlineto{\pgfqpoint{4.053314in}{1.827216in}}%
\pgfpathclose%
\pgfusepath{stroke,fill}%
\end{pgfscope}%
\begin{pgfscope}%
\pgfpathrectangle{\pgfqpoint{0.100000in}{0.100000in}}{\pgfqpoint{4.700000in}{4.700000in}}%
\pgfusepath{clip}%
\pgfsetbuttcap%
\pgfsetroundjoin%
\definecolor{currentfill}{rgb}{1.000000,1.000000,1.000000}%
\pgfsetfillcolor{currentfill}%
\pgfsetfillopacity{0.100000}%
\pgfsetlinewidth{1.003750pt}%
\definecolor{currentstroke}{rgb}{0.000000,0.000000,0.000000}%
\pgfsetstrokecolor{currentstroke}%
\pgfsetdash{}{0pt}%
\pgfpathmoveto{\pgfqpoint{3.668364in}{3.905553in}}%
\pgfpathlineto{\pgfqpoint{2.128563in}{3.135653in}}%
\pgfpathlineto{\pgfqpoint{2.513514in}{2.943178in}}%
\pgfpathlineto{\pgfqpoint{4.053314in}{3.713078in}}%
\pgfpathlineto{\pgfqpoint{3.668364in}{3.905553in}}%
\pgfpathclose%
\pgfusepath{stroke,fill}%
\end{pgfscope}%
\begin{pgfscope}%
\definecolor{textcolor}{rgb}{0.000000,0.000000,0.000000}%
\pgfsetstrokecolor{textcolor}%
\pgfsetfillcolor{textcolor}%
\pgftext[x=1.166188in,y=2.673909in,left,base]{\color{textcolor}{\rmfamily\fontsize{20.000000}{24.000000}\selectfont\catcode`\^=\active\def^{\ifmmode\sp\else\^{}\fi}\catcode`\%=\active\def
\end{pgfscope}%
\begin{pgfscope}%
\definecolor{textcolor}{rgb}{0.000000,0.000000,0.000000}%
\pgfsetstrokecolor{textcolor}%
\pgfsetfillcolor{textcolor}%
\pgftext[x=1.936088in,y=4.001791in,left,base]{\color{textcolor}{\rmfamily\fontsize{20.000000}{24.000000}\selectfont\catcode`\^=\active\def^{\ifmmode\sp\else\^{}\fi}\catcode`\%=\active\def
\end{pgfscope}%
\begin{pgfscope}%
\definecolor{textcolor}{rgb}{0.000000,0.000000,0.000000}%
\pgfsetstrokecolor{textcolor}%
\pgfsetfillcolor{textcolor}%
\pgftext[x=1.743613in,y=2.385197in,left,base]{\color{textcolor}{\rmfamily\fontsize{20.000000}{24.000000}\selectfont\catcode`\^=\active\def^{\ifmmode\sp\else\^{}\fi}\catcode`\%=\active\def
\end{pgfscope}%
\begin{pgfscope}%
\definecolor{textcolor}{rgb}{0.000000,0.000000,0.000000}%
\pgfsetstrokecolor{textcolor}%
\pgfsetfillcolor{textcolor}%
\pgftext[x=2.513514in,y=3.713078in,left,base]{\color{textcolor}{\rmfamily\fontsize{20.000000}{24.000000}\selectfont\catcode`\^=\active\def^{\ifmmode\sp\else\^{}\fi}\catcode`\%=\active\def
\end{pgfscope}%
\begin{pgfscope}%
\definecolor{textcolor}{rgb}{0.000000,0.000000,0.000000}%
\pgfsetstrokecolor{textcolor}%
\pgfsetfillcolor{textcolor}%
\pgftext[x=2.321039in,y=2.096484in,left,base]{\color{textcolor}{\rmfamily\fontsize{20.000000}{24.000000}\selectfont\catcode`\^=\active\def^{\ifmmode\sp\else\^{}\fi}\catcode`\%=\active\def
\end{pgfscope}%
\begin{pgfscope}%
\definecolor{textcolor}{rgb}{0.000000,0.000000,0.000000}%
\pgfsetstrokecolor{textcolor}%
\pgfsetfillcolor{textcolor}%
\pgftext[x=3.283414in,y=2.385197in,left,base]{\color{textcolor}{\rmfamily\fontsize{20.000000}{24.000000}\selectfont\catcode`\^=\active\def^{\ifmmode\sp\else\^{}\fi}\catcode`\%=\active\def
\end{pgfscope}%
\begin{pgfscope}%
\definecolor{textcolor}{rgb}{0.000000,0.000000,0.000000}%
\pgfsetstrokecolor{textcolor}%
\pgfsetfillcolor{textcolor}%
\pgftext[x=3.090939in,y=3.424365in,left,base]{\color{textcolor}{\rmfamily\fontsize{20.000000}{24.000000}\selectfont\catcode`\^=\active\def^{\ifmmode\sp\else\^{}\fi}\catcode`\%=\active\def
\end{pgfscope}%
\end{pgfpicture}%
\makeatother%
\endgroup%

%% file: figs/fp-bis-patt1.pgf
\begingroup%
\makeatletter%
\begin{pgfpicture}%
\pgfpathrectangle{\pgfpointorigin}{\pgfqpoint{4.900000in}{4.900000in}}%
\pgfusepath{use as bounding box, clip}%
\begin{pgfscope}%
\pgfsetbuttcap%
\pgfsetmiterjoin%
\definecolor{currentfill}{rgb}{1.000000,1.000000,1.000000}%
\pgfsetfillcolor{currentfill}%
\pgfsetlinewidth{0.000000pt}%
\definecolor{currentstroke}{rgb}{1.000000,1.000000,1.000000}%
\pgfsetstrokecolor{currentstroke}%
\pgfsetdash{}{0pt}%
\pgfpathmoveto{\pgfqpoint{0.000000in}{0.000000in}}%
\pgfpathlineto{\pgfqpoint{4.900000in}{0.000000in}}%
\pgfpathlineto{\pgfqpoint{4.900000in}{4.900000in}}%
\pgfpathlineto{\pgfqpoint{0.000000in}{4.900000in}}%
\pgfpathlineto{\pgfqpoint{0.000000in}{0.000000in}}%
\pgfpathclose%
\pgfusepath{fill}%
\end{pgfscope}%
\begin{pgfscope}%
\pgfsetbuttcap%
\pgfsetmiterjoin%
\definecolor{currentfill}{rgb}{1.000000,1.000000,1.000000}%
\pgfsetfillcolor{currentfill}%
\pgfsetlinewidth{0.000000pt}%
\definecolor{currentstroke}{rgb}{0.000000,0.000000,0.000000}%
\pgfsetstrokecolor{currentstroke}%
\pgfsetstrokeopacity{0.000000}%
\pgfsetdash{}{0pt}%
\pgfpathmoveto{\pgfqpoint{0.100000in}{0.100000in}}%
\pgfpathlineto{\pgfqpoint{4.800000in}{0.100000in}}%
\pgfpathlineto{\pgfqpoint{4.800000in}{4.800000in}}%
\pgfpathlineto{\pgfqpoint{0.100000in}{4.800000in}}%
\pgfpathlineto{\pgfqpoint{0.100000in}{0.100000in}}%
\pgfpathclose%
\pgfusepath{fill}%
\end{pgfscope}%
\begin{pgfscope}%
\pgfpathrectangle{\pgfqpoint{0.100000in}{0.100000in}}{\pgfqpoint{4.700000in}{4.700000in}}%
\pgfusepath{clip}%
\pgfsetbuttcap%
\pgfsetroundjoin%
\definecolor{currentfill}{rgb}{0.000000,0.000000,0.000000}%
\pgfsetfillcolor{currentfill}%
\pgfsetfillopacity{0.200000}%
\pgfsetlinewidth{1.003750pt}%
\definecolor{currentstroke}{rgb}{0.000000,0.000000,0.000000}%
\pgfsetstrokecolor{currentstroke}%
\pgfsetdash{}{0pt}%
\pgfpathmoveto{\pgfqpoint{2.513514in}{4.011512in}}%
\pgfpathlineto{\pgfqpoint{2.898464in}{3.819037in}}%
\pgfpathlineto{\pgfqpoint{2.898464in}{4.290503in}}%
\pgfpathlineto{\pgfqpoint{2.513514in}{4.482978in}}%
\pgfpathlineto{\pgfqpoint{2.513514in}{4.011512in}}%
\pgfpathclose%
\pgfusepath{stroke,fill}%
\end{pgfscope}%
\begin{pgfscope}%
\pgfpathrectangle{\pgfqpoint{0.100000in}{0.100000in}}{\pgfqpoint{4.700000in}{4.700000in}}%
\pgfusepath{clip}%
\pgfsetbuttcap%
\pgfsetroundjoin%
\definecolor{currentfill}{rgb}{0.000000,0.000000,0.000000}%
\pgfsetfillcolor{currentfill}%
\pgfsetfillopacity{0.200000}%
\pgfsetlinewidth{1.003750pt}%
\definecolor{currentstroke}{rgb}{0.000000,0.000000,0.000000}%
\pgfsetstrokecolor{currentstroke}%
\pgfsetdash{}{0pt}%
\pgfpathmoveto{\pgfqpoint{2.513514in}{4.011512in}}%
\pgfpathlineto{\pgfqpoint{1.743613in}{3.626562in}}%
\pgfpathlineto{\pgfqpoint{1.743613in}{4.098028in}}%
\pgfpathlineto{\pgfqpoint{2.513514in}{4.482978in}}%
\pgfpathlineto{\pgfqpoint{2.513514in}{4.011512in}}%
\pgfpathclose%
\pgfusepath{stroke,fill}%
\end{pgfscope}%
\begin{pgfscope}%
\pgfpathrectangle{\pgfqpoint{0.100000in}{0.100000in}}{\pgfqpoint{4.700000in}{4.700000in}}%
\pgfusepath{clip}%
\pgfsetbuttcap%
\pgfsetroundjoin%
\definecolor{currentfill}{rgb}{0.000000,0.000000,0.000000}%
\pgfsetfillcolor{currentfill}%
\pgfsetfillopacity{0.200000}%
\pgfsetlinewidth{1.003750pt}%
\definecolor{currentstroke}{rgb}{0.000000,0.000000,0.000000}%
\pgfsetstrokecolor{currentstroke}%
\pgfsetdash{}{0pt}%
\pgfpathmoveto{\pgfqpoint{2.513514in}{4.011512in}}%
\pgfpathlineto{\pgfqpoint{1.743613in}{3.626562in}}%
\pgfpathlineto{\pgfqpoint{2.128563in}{3.434087in}}%
\pgfpathlineto{\pgfqpoint{2.898464in}{3.819037in}}%
\pgfpathlineto{\pgfqpoint{2.513514in}{4.011512in}}%
\pgfpathclose%
\pgfusepath{stroke,fill}%
\end{pgfscope}%
\begin{pgfscope}%
\pgfpathrectangle{\pgfqpoint{0.100000in}{0.100000in}}{\pgfqpoint{4.700000in}{4.700000in}}%
\pgfusepath{clip}%
\pgfsetbuttcap%
\pgfsetroundjoin%
\definecolor{currentfill}{rgb}{0.000000,0.000000,0.000000}%
\pgfsetfillcolor{currentfill}%
\pgfsetfillopacity{0.200000}%
\pgfsetlinewidth{1.003750pt}%
\definecolor{currentstroke}{rgb}{0.000000,0.000000,0.000000}%
\pgfsetstrokecolor{currentstroke}%
\pgfsetdash{}{0pt}%
\pgfpathmoveto{\pgfqpoint{1.743613in}{3.626562in}}%
\pgfpathlineto{\pgfqpoint{2.128563in}{3.434087in}}%
\pgfpathlineto{\pgfqpoint{2.128563in}{3.905553in}}%
\pgfpathlineto{\pgfqpoint{1.743613in}{4.098028in}}%
\pgfpathlineto{\pgfqpoint{1.743613in}{3.626562in}}%
\pgfpathclose%
\pgfusepath{stroke,fill}%
\end{pgfscope}%
\begin{pgfscope}%
\pgfpathrectangle{\pgfqpoint{0.100000in}{0.100000in}}{\pgfqpoint{4.700000in}{4.700000in}}%
\pgfusepath{clip}%
\pgfsetbuttcap%
\pgfsetroundjoin%
\definecolor{currentfill}{rgb}{0.000000,0.000000,0.000000}%
\pgfsetfillcolor{currentfill}%
\pgfsetfillopacity{0.200000}%
\pgfsetlinewidth{1.003750pt}%
\definecolor{currentstroke}{rgb}{0.000000,0.000000,0.000000}%
\pgfsetstrokecolor{currentstroke}%
\pgfsetdash{}{0pt}%
\pgfpathmoveto{\pgfqpoint{2.898464in}{3.819037in}}%
\pgfpathlineto{\pgfqpoint{2.128563in}{3.434087in}}%
\pgfpathlineto{\pgfqpoint{2.128563in}{3.905553in}}%
\pgfpathlineto{\pgfqpoint{2.898464in}{4.290503in}}%
\pgfpathlineto{\pgfqpoint{2.898464in}{3.819037in}}%
\pgfpathclose%
\pgfusepath{stroke,fill}%
\end{pgfscope}%
\begin{pgfscope}%
\pgfpathrectangle{\pgfqpoint{0.100000in}{0.100000in}}{\pgfqpoint{4.700000in}{4.700000in}}%
\pgfusepath{clip}%
\pgfsetbuttcap%
\pgfsetroundjoin%
\definecolor{currentfill}{rgb}{0.000000,0.000000,0.000000}%
\pgfsetfillcolor{currentfill}%
\pgfsetfillopacity{0.200000}%
\pgfsetlinewidth{1.003750pt}%
\definecolor{currentstroke}{rgb}{0.000000,0.000000,0.000000}%
\pgfsetstrokecolor{currentstroke}%
\pgfsetdash{}{0pt}%
\pgfpathmoveto{\pgfqpoint{2.513514in}{4.482978in}}%
\pgfpathlineto{\pgfqpoint{1.743613in}{4.098028in}}%
\pgfpathlineto{\pgfqpoint{2.128563in}{3.905553in}}%
\pgfpathlineto{\pgfqpoint{2.898464in}{4.290503in}}%
\pgfpathlineto{\pgfqpoint{2.513514in}{4.482978in}}%
\pgfpathclose%
\pgfusepath{stroke,fill}%
\end{pgfscope}%
\begin{pgfscope}%
\pgfpathrectangle{\pgfqpoint{0.100000in}{0.100000in}}{\pgfqpoint{4.700000in}{4.700000in}}%
\pgfusepath{clip}%
\pgfsetbuttcap%
\pgfsetroundjoin%
\definecolor{currentfill}{rgb}{0.000000,0.000000,0.000000}%
\pgfsetfillcolor{currentfill}%
\pgfsetfillopacity{0.200000}%
\pgfsetlinewidth{1.003750pt}%
\definecolor{currentstroke}{rgb}{0.000000,0.000000,0.000000}%
\pgfsetstrokecolor{currentstroke}%
\pgfsetdash{}{0pt}%
\pgfpathmoveto{\pgfqpoint{1.358663in}{3.434087in}}%
\pgfpathlineto{\pgfqpoint{0.973713in}{3.241612in}}%
\pgfpathlineto{\pgfqpoint{0.973713in}{3.713078in}}%
\pgfpathlineto{\pgfqpoint{1.358663in}{3.905553in}}%
\pgfpathlineto{\pgfqpoint{1.358663in}{3.434087in}}%
\pgfpathclose%
\pgfusepath{stroke,fill}%
\end{pgfscope}%
\begin{pgfscope}%
\pgfpathrectangle{\pgfqpoint{0.100000in}{0.100000in}}{\pgfqpoint{4.700000in}{4.700000in}}%
\pgfusepath{clip}%
\pgfsetbuttcap%
\pgfsetroundjoin%
\definecolor{currentfill}{rgb}{0.000000,0.000000,0.000000}%
\pgfsetfillcolor{currentfill}%
\pgfsetfillopacity{0.200000}%
\pgfsetlinewidth{1.003750pt}%
\definecolor{currentstroke}{rgb}{0.000000,0.000000,0.000000}%
\pgfsetstrokecolor{currentstroke}%
\pgfsetdash{}{0pt}%
\pgfpathmoveto{\pgfqpoint{1.358663in}{3.434087in}}%
\pgfpathlineto{\pgfqpoint{2.128563in}{3.049137in}}%
\pgfpathlineto{\pgfqpoint{2.128563in}{3.520603in}}%
\pgfpathlineto{\pgfqpoint{1.358663in}{3.905553in}}%
\pgfpathlineto{\pgfqpoint{1.358663in}{3.434087in}}%
\pgfpathclose%
\pgfusepath{stroke,fill}%
\end{pgfscope}%
\begin{pgfscope}%
\pgfpathrectangle{\pgfqpoint{0.100000in}{0.100000in}}{\pgfqpoint{4.700000in}{4.700000in}}%
\pgfusepath{clip}%
\pgfsetbuttcap%
\pgfsetroundjoin%
\definecolor{currentfill}{rgb}{0.000000,0.000000,0.000000}%
\pgfsetfillcolor{currentfill}%
\pgfsetfillopacity{0.200000}%
\pgfsetlinewidth{1.003750pt}%
\definecolor{currentstroke}{rgb}{0.000000,0.000000,0.000000}%
\pgfsetstrokecolor{currentstroke}%
\pgfsetdash{}{0pt}%
\pgfpathmoveto{\pgfqpoint{1.358663in}{3.434087in}}%
\pgfpathlineto{\pgfqpoint{0.973713in}{3.241612in}}%
\pgfpathlineto{\pgfqpoint{1.743613in}{2.856662in}}%
\pgfpathlineto{\pgfqpoint{2.128563in}{3.049137in}}%
\pgfpathlineto{\pgfqpoint{1.358663in}{3.434087in}}%
\pgfpathclose%
\pgfusepath{stroke,fill}%
\end{pgfscope}%
\begin{pgfscope}%
\pgfpathrectangle{\pgfqpoint{0.100000in}{0.100000in}}{\pgfqpoint{4.700000in}{4.700000in}}%
\pgfusepath{clip}%
\pgfsetbuttcap%
\pgfsetroundjoin%
\definecolor{currentfill}{rgb}{0.000000,0.000000,0.000000}%
\pgfsetfillcolor{currentfill}%
\pgfsetfillopacity{0.200000}%
\pgfsetlinewidth{1.003750pt}%
\definecolor{currentstroke}{rgb}{0.000000,0.000000,0.000000}%
\pgfsetstrokecolor{currentstroke}%
\pgfsetdash{}{0pt}%
\pgfpathmoveto{\pgfqpoint{0.973713in}{3.241612in}}%
\pgfpathlineto{\pgfqpoint{1.743613in}{2.856662in}}%
\pgfpathlineto{\pgfqpoint{1.743613in}{3.328128in}}%
\pgfpathlineto{\pgfqpoint{0.973713in}{3.713078in}}%
\pgfpathlineto{\pgfqpoint{0.973713in}{3.241612in}}%
\pgfpathclose%
\pgfusepath{stroke,fill}%
\end{pgfscope}%
\begin{pgfscope}%
\pgfpathrectangle{\pgfqpoint{0.100000in}{0.100000in}}{\pgfqpoint{4.700000in}{4.700000in}}%
\pgfusepath{clip}%
\pgfsetbuttcap%
\pgfsetroundjoin%
\definecolor{currentfill}{rgb}{0.000000,0.000000,0.000000}%
\pgfsetfillcolor{currentfill}%
\pgfsetfillopacity{0.200000}%
\pgfsetlinewidth{1.003750pt}%
\definecolor{currentstroke}{rgb}{0.000000,0.000000,0.000000}%
\pgfsetstrokecolor{currentstroke}%
\pgfsetdash{}{0pt}%
\pgfpathmoveto{\pgfqpoint{2.128563in}{3.049137in}}%
\pgfpathlineto{\pgfqpoint{1.743613in}{2.856662in}}%
\pgfpathlineto{\pgfqpoint{1.743613in}{3.328128in}}%
\pgfpathlineto{\pgfqpoint{2.128563in}{3.520603in}}%
\pgfpathlineto{\pgfqpoint{2.128563in}{3.049137in}}%
\pgfpathclose%
\pgfusepath{stroke,fill}%
\end{pgfscope}%
\begin{pgfscope}%
\pgfpathrectangle{\pgfqpoint{0.100000in}{0.100000in}}{\pgfqpoint{4.700000in}{4.700000in}}%
\pgfusepath{clip}%
\pgfsetbuttcap%
\pgfsetroundjoin%
\definecolor{currentfill}{rgb}{0.000000,0.000000,0.000000}%
\pgfsetfillcolor{currentfill}%
\pgfsetfillopacity{0.200000}%
\pgfsetlinewidth{1.003750pt}%
\definecolor{currentstroke}{rgb}{0.000000,0.000000,0.000000}%
\pgfsetstrokecolor{currentstroke}%
\pgfsetdash{}{0pt}%
\pgfpathmoveto{\pgfqpoint{1.358663in}{3.905553in}}%
\pgfpathlineto{\pgfqpoint{0.973713in}{3.713078in}}%
\pgfpathlineto{\pgfqpoint{1.743613in}{3.328128in}}%
\pgfpathlineto{\pgfqpoint{2.128563in}{3.520603in}}%
\pgfpathlineto{\pgfqpoint{1.358663in}{3.905553in}}%
\pgfpathclose%
\pgfusepath{stroke,fill}%
\end{pgfscope}%
\begin{pgfscope}%
\pgfpathrectangle{\pgfqpoint{0.100000in}{0.100000in}}{\pgfqpoint{4.700000in}{4.700000in}}%
\pgfusepath{clip}%
\pgfsetbuttcap%
\pgfsetroundjoin%
\definecolor{currentfill}{rgb}{0.000000,0.000000,0.000000}%
\pgfsetfillcolor{currentfill}%
\pgfsetfillopacity{0.200000}%
\pgfsetlinewidth{1.003750pt}%
\definecolor{currentstroke}{rgb}{0.000000,0.000000,0.000000}%
\pgfsetstrokecolor{currentstroke}%
\pgfsetdash{}{0pt}%
\pgfpathmoveto{\pgfqpoint{3.283414in}{2.212166in}}%
\pgfpathlineto{\pgfqpoint{2.898464in}{2.019691in}}%
\pgfpathlineto{\pgfqpoint{2.898464in}{2.491156in}}%
\pgfpathlineto{\pgfqpoint{3.283414in}{2.683631in}}%
\pgfpathlineto{\pgfqpoint{3.283414in}{2.212166in}}%
\pgfpathclose%
\pgfusepath{stroke,fill}%
\end{pgfscope}%
\begin{pgfscope}%
\pgfpathrectangle{\pgfqpoint{0.100000in}{0.100000in}}{\pgfqpoint{4.700000in}{4.700000in}}%
\pgfusepath{clip}%
\pgfsetbuttcap%
\pgfsetroundjoin%
\definecolor{currentfill}{rgb}{0.000000,0.000000,0.000000}%
\pgfsetfillcolor{currentfill}%
\pgfsetfillopacity{0.200000}%
\pgfsetlinewidth{1.003750pt}%
\definecolor{currentstroke}{rgb}{0.000000,0.000000,0.000000}%
\pgfsetstrokecolor{currentstroke}%
\pgfsetdash{}{0pt}%
\pgfpathmoveto{\pgfqpoint{3.283414in}{2.212166in}}%
\pgfpathlineto{\pgfqpoint{4.053314in}{1.827216in}}%
\pgfpathlineto{\pgfqpoint{4.053314in}{2.298681in}}%
\pgfpathlineto{\pgfqpoint{3.283414in}{2.683631in}}%
\pgfpathlineto{\pgfqpoint{3.283414in}{2.212166in}}%
\pgfpathclose%
\pgfusepath{stroke,fill}%
\end{pgfscope}%
\begin{pgfscope}%
\pgfpathrectangle{\pgfqpoint{0.100000in}{0.100000in}}{\pgfqpoint{4.700000in}{4.700000in}}%
\pgfusepath{clip}%
\pgfsetbuttcap%
\pgfsetroundjoin%
\definecolor{currentfill}{rgb}{0.000000,0.000000,0.000000}%
\pgfsetfillcolor{currentfill}%
\pgfsetfillopacity{0.200000}%
\pgfsetlinewidth{1.003750pt}%
\definecolor{currentstroke}{rgb}{0.000000,0.000000,0.000000}%
\pgfsetstrokecolor{currentstroke}%
\pgfsetdash{}{0pt}%
\pgfpathmoveto{\pgfqpoint{3.283414in}{2.212166in}}%
\pgfpathlineto{\pgfqpoint{2.898464in}{2.019691in}}%
\pgfpathlineto{\pgfqpoint{3.668364in}{1.634741in}}%
\pgfpathlineto{\pgfqpoint{4.053314in}{1.827216in}}%
\pgfpathlineto{\pgfqpoint{3.283414in}{2.212166in}}%
\pgfpathclose%
\pgfusepath{stroke,fill}%
\end{pgfscope}%
\begin{pgfscope}%
\pgfpathrectangle{\pgfqpoint{0.100000in}{0.100000in}}{\pgfqpoint{4.700000in}{4.700000in}}%
\pgfusepath{clip}%
\pgfsetbuttcap%
\pgfsetroundjoin%
\definecolor{currentfill}{rgb}{0.000000,0.000000,0.000000}%
\pgfsetfillcolor{currentfill}%
\pgfsetfillopacity{0.200000}%
\pgfsetlinewidth{1.003750pt}%
\definecolor{currentstroke}{rgb}{0.000000,0.000000,0.000000}%
\pgfsetstrokecolor{currentstroke}%
\pgfsetdash{}{0pt}%
\pgfpathmoveto{\pgfqpoint{2.898464in}{2.019691in}}%
\pgfpathlineto{\pgfqpoint{3.668364in}{1.634741in}}%
\pgfpathlineto{\pgfqpoint{3.668364in}{2.106206in}}%
\pgfpathlineto{\pgfqpoint{2.898464in}{2.491156in}}%
\pgfpathlineto{\pgfqpoint{2.898464in}{2.019691in}}%
\pgfpathclose%
\pgfusepath{stroke,fill}%
\end{pgfscope}%
\begin{pgfscope}%
\pgfpathrectangle{\pgfqpoint{0.100000in}{0.100000in}}{\pgfqpoint{4.700000in}{4.700000in}}%
\pgfusepath{clip}%
\pgfsetbuttcap%
\pgfsetroundjoin%
\definecolor{currentfill}{rgb}{0.000000,0.000000,0.000000}%
\pgfsetfillcolor{currentfill}%
\pgfsetfillopacity{0.200000}%
\pgfsetlinewidth{1.003750pt}%
\definecolor{currentstroke}{rgb}{0.000000,0.000000,0.000000}%
\pgfsetstrokecolor{currentstroke}%
\pgfsetdash{}{0pt}%
\pgfpathmoveto{\pgfqpoint{4.053314in}{1.827216in}}%
\pgfpathlineto{\pgfqpoint{3.668364in}{1.634741in}}%
\pgfpathlineto{\pgfqpoint{3.668364in}{2.106206in}}%
\pgfpathlineto{\pgfqpoint{4.053314in}{2.298681in}}%
\pgfpathlineto{\pgfqpoint{4.053314in}{1.827216in}}%
\pgfpathclose%
\pgfusepath{stroke,fill}%
\end{pgfscope}%
\begin{pgfscope}%
\pgfpathrectangle{\pgfqpoint{0.100000in}{0.100000in}}{\pgfqpoint{4.700000in}{4.700000in}}%
\pgfusepath{clip}%
\pgfsetbuttcap%
\pgfsetroundjoin%
\definecolor{currentfill}{rgb}{0.000000,0.000000,0.000000}%
\pgfsetfillcolor{currentfill}%
\pgfsetfillopacity{0.200000}%
\pgfsetlinewidth{1.003750pt}%
\definecolor{currentstroke}{rgb}{0.000000,0.000000,0.000000}%
\pgfsetstrokecolor{currentstroke}%
\pgfsetdash{}{0pt}%
\pgfpathmoveto{\pgfqpoint{3.283414in}{2.683631in}}%
\pgfpathlineto{\pgfqpoint{2.898464in}{2.491156in}}%
\pgfpathlineto{\pgfqpoint{3.668364in}{2.106206in}}%
\pgfpathlineto{\pgfqpoint{4.053314in}{2.298681in}}%
\pgfpathlineto{\pgfqpoint{3.283414in}{2.683631in}}%
\pgfpathclose%
\pgfusepath{stroke,fill}%
\end{pgfscope}%
\begin{pgfscope}%
\pgfpathrectangle{\pgfqpoint{0.100000in}{0.100000in}}{\pgfqpoint{4.700000in}{4.700000in}}%
\pgfusepath{clip}%
\pgfsetbuttcap%
\pgfsetroundjoin%
\definecolor{currentfill}{rgb}{0.000000,0.000000,0.000000}%
\pgfsetfillcolor{currentfill}%
\pgfsetfillopacity{0.200000}%
\pgfsetlinewidth{1.003750pt}%
\definecolor{currentstroke}{rgb}{0.000000,0.000000,0.000000}%
\pgfsetstrokecolor{currentstroke}%
\pgfsetdash{}{0pt}%
\pgfpathmoveto{\pgfqpoint{2.898464in}{1.634741in}}%
\pgfpathlineto{\pgfqpoint{3.283414in}{1.442266in}}%
\pgfpathlineto{\pgfqpoint{3.283414in}{1.913731in}}%
\pgfpathlineto{\pgfqpoint{2.898464in}{2.106206in}}%
\pgfpathlineto{\pgfqpoint{2.898464in}{1.634741in}}%
\pgfpathclose%
\pgfusepath{stroke,fill}%
\end{pgfscope}%
\begin{pgfscope}%
\pgfpathrectangle{\pgfqpoint{0.100000in}{0.100000in}}{\pgfqpoint{4.700000in}{4.700000in}}%
\pgfusepath{clip}%
\pgfsetbuttcap%
\pgfsetroundjoin%
\definecolor{currentfill}{rgb}{0.000000,0.000000,0.000000}%
\pgfsetfillcolor{currentfill}%
\pgfsetfillopacity{0.200000}%
\pgfsetlinewidth{1.003750pt}%
\definecolor{currentstroke}{rgb}{0.000000,0.000000,0.000000}%
\pgfsetstrokecolor{currentstroke}%
\pgfsetdash{}{0pt}%
\pgfpathmoveto{\pgfqpoint{2.898464in}{1.634741in}}%
\pgfpathlineto{\pgfqpoint{2.128563in}{1.249791in}}%
\pgfpathlineto{\pgfqpoint{2.128563in}{1.721256in}}%
\pgfpathlineto{\pgfqpoint{2.898464in}{2.106206in}}%
\pgfpathlineto{\pgfqpoint{2.898464in}{1.634741in}}%
\pgfpathclose%
\pgfusepath{stroke,fill}%
\end{pgfscope}%
\begin{pgfscope}%
\pgfpathrectangle{\pgfqpoint{0.100000in}{0.100000in}}{\pgfqpoint{4.700000in}{4.700000in}}%
\pgfusepath{clip}%
\pgfsetbuttcap%
\pgfsetroundjoin%
\definecolor{currentfill}{rgb}{0.000000,0.000000,0.000000}%
\pgfsetfillcolor{currentfill}%
\pgfsetfillopacity{0.200000}%
\pgfsetlinewidth{1.003750pt}%
\definecolor{currentstroke}{rgb}{0.000000,0.000000,0.000000}%
\pgfsetstrokecolor{currentstroke}%
\pgfsetdash{}{0pt}%
\pgfpathmoveto{\pgfqpoint{2.898464in}{1.634741in}}%
\pgfpathlineto{\pgfqpoint{2.128563in}{1.249791in}}%
\pgfpathlineto{\pgfqpoint{2.513514in}{1.057316in}}%
\pgfpathlineto{\pgfqpoint{3.283414in}{1.442266in}}%
\pgfpathlineto{\pgfqpoint{2.898464in}{1.634741in}}%
\pgfpathclose%
\pgfusepath{stroke,fill}%
\end{pgfscope}%
\begin{pgfscope}%
\pgfpathrectangle{\pgfqpoint{0.100000in}{0.100000in}}{\pgfqpoint{4.700000in}{4.700000in}}%
\pgfusepath{clip}%
\pgfsetbuttcap%
\pgfsetroundjoin%
\definecolor{currentfill}{rgb}{0.000000,0.000000,0.000000}%
\pgfsetfillcolor{currentfill}%
\pgfsetfillopacity{0.200000}%
\pgfsetlinewidth{1.003750pt}%
\definecolor{currentstroke}{rgb}{0.000000,0.000000,0.000000}%
\pgfsetstrokecolor{currentstroke}%
\pgfsetdash{}{0pt}%
\pgfpathmoveto{\pgfqpoint{2.128563in}{1.249791in}}%
\pgfpathlineto{\pgfqpoint{2.513514in}{1.057316in}}%
\pgfpathlineto{\pgfqpoint{2.513514in}{1.528781in}}%
\pgfpathlineto{\pgfqpoint{2.128563in}{1.721256in}}%
\pgfpathlineto{\pgfqpoint{2.128563in}{1.249791in}}%
\pgfpathclose%
\pgfusepath{stroke,fill}%
\end{pgfscope}%
\begin{pgfscope}%
\pgfpathrectangle{\pgfqpoint{0.100000in}{0.100000in}}{\pgfqpoint{4.700000in}{4.700000in}}%
\pgfusepath{clip}%
\pgfsetbuttcap%
\pgfsetroundjoin%
\definecolor{currentfill}{rgb}{0.000000,0.000000,0.000000}%
\pgfsetfillcolor{currentfill}%
\pgfsetfillopacity{0.200000}%
\pgfsetlinewidth{1.003750pt}%
\definecolor{currentstroke}{rgb}{0.000000,0.000000,0.000000}%
\pgfsetstrokecolor{currentstroke}%
\pgfsetdash{}{0pt}%
\pgfpathmoveto{\pgfqpoint{3.283414in}{1.442266in}}%
\pgfpathlineto{\pgfqpoint{2.513514in}{1.057316in}}%
\pgfpathlineto{\pgfqpoint{2.513514in}{1.528781in}}%
\pgfpathlineto{\pgfqpoint{3.283414in}{1.913731in}}%
\pgfpathlineto{\pgfqpoint{3.283414in}{1.442266in}}%
\pgfpathclose%
\pgfusepath{stroke,fill}%
\end{pgfscope}%
\begin{pgfscope}%
\pgfpathrectangle{\pgfqpoint{0.100000in}{0.100000in}}{\pgfqpoint{4.700000in}{4.700000in}}%
\pgfusepath{clip}%
\pgfsetbuttcap%
\pgfsetroundjoin%
\definecolor{currentfill}{rgb}{0.000000,0.000000,0.000000}%
\pgfsetfillcolor{currentfill}%
\pgfsetfillopacity{0.200000}%
\pgfsetlinewidth{1.003750pt}%
\definecolor{currentstroke}{rgb}{0.000000,0.000000,0.000000}%
\pgfsetstrokecolor{currentstroke}%
\pgfsetdash{}{0pt}%
\pgfpathmoveto{\pgfqpoint{2.898464in}{2.106206in}}%
\pgfpathlineto{\pgfqpoint{2.128563in}{1.721256in}}%
\pgfpathlineto{\pgfqpoint{2.513514in}{1.528781in}}%
\pgfpathlineto{\pgfqpoint{3.283414in}{1.913731in}}%
\pgfpathlineto{\pgfqpoint{2.898464in}{2.106206in}}%
\pgfpathclose%
\pgfusepath{stroke,fill}%
\end{pgfscope}%
\begin{pgfscope}%
\pgfpathrectangle{\pgfqpoint{0.100000in}{0.100000in}}{\pgfqpoint{4.700000in}{4.700000in}}%
\pgfusepath{clip}%
\pgfsetbuttcap%
\pgfsetroundjoin%
\definecolor{currentfill}{rgb}{0.281250,0.721569,0.215686}%
\pgfsetfillcolor{currentfill}%
\pgfsetfillopacity{0.200000}%
\pgfsetlinewidth{1.003750pt}%
\definecolor{currentstroke}{rgb}{0.000000,0.000000,0.000000}%
\pgfsetstrokecolor{currentstroke}%
\pgfsetdash{}{0pt}%
\pgfsys@defobject{currentmarker}{\pgfqpoint{1.743613in}{1.442266in}}{\pgfqpoint{3.283414in}{4.098028in}}{%
\pgfpathmoveto{\pgfqpoint{1.743613in}{2.212166in}}%
\pgfpathlineto{\pgfqpoint{3.283414in}{1.442266in}}%
\pgfpathlineto{\pgfqpoint{3.283414in}{3.328128in}}%
\pgfpathlineto{\pgfqpoint{1.743613in}{4.098028in}}%
\pgfpathlineto{\pgfqpoint{1.743613in}{2.212166in}}%
\pgfpathclose%
\pgfusepath{stroke,fill}%
}%
\begin{pgfscope}%
\pgfsys@transformshift{0.000000in}{0.000000in}%
\pgfsys@useobject{currentmarker}{}%
\end{pgfscope}%
\end{pgfscope}%
\begin{pgfscope}%
\pgfpathrectangle{\pgfqpoint{0.100000in}{0.100000in}}{\pgfqpoint{4.700000in}{4.700000in}}%
\pgfusepath{clip}%
\pgfsetbuttcap%
\pgfsetroundjoin%
\definecolor{currentfill}{rgb}{0.613281,0.125000,0.960784}%
\pgfsetfillcolor{currentfill}%
\pgfsetfillopacity{0.200000}%
\pgfsetlinewidth{1.003750pt}%
\definecolor{currentstroke}{rgb}{0.000000,0.000000,0.000000}%
\pgfsetstrokecolor{currentstroke}%
\pgfsetdash{}{0pt}%
\pgfsys@defobject{currentmarker}{\pgfqpoint{1.743613in}{1.442266in}}{\pgfqpoint{3.283414in}{4.098028in}}{%
\pgfpathmoveto{\pgfqpoint{3.283414in}{2.212166in}}%
\pgfpathlineto{\pgfqpoint{1.743613in}{1.442266in}}%
\pgfpathlineto{\pgfqpoint{1.743613in}{3.328128in}}%
\pgfpathlineto{\pgfqpoint{3.283414in}{4.098028in}}%
\pgfpathlineto{\pgfqpoint{3.283414in}{2.212166in}}%
\pgfpathclose%
\pgfusepath{stroke,fill}%
}%
\begin{pgfscope}%
\pgfsys@transformshift{0.000000in}{0.000000in}%
\pgfsys@useobject{currentmarker}{}%
\end{pgfscope}%
\end{pgfscope}%
\begin{pgfscope}%
\pgfpathrectangle{\pgfqpoint{0.100000in}{0.100000in}}{\pgfqpoint{4.700000in}{4.700000in}}%
\pgfusepath{clip}%
\pgfsetbuttcap%
\pgfsetroundjoin%
\definecolor{currentfill}{rgb}{0.831373,0.168627,0.043137}%
\pgfsetfillcolor{currentfill}%
\pgfsetlinewidth{1.003750pt}%
\definecolor{currentstroke}{rgb}{0.831373,0.168627,0.043137}%
\pgfsetstrokecolor{currentstroke}%
\pgfsetdash{}{0pt}%
\pgfsys@defobject{currentmarker}{\pgfqpoint{-0.069444in}{-0.069444in}}{\pgfqpoint{0.069444in}{0.069444in}}{%
\pgfpathmoveto{\pgfqpoint{0.000000in}{-0.069444in}}%
\pgfpathcurveto{\pgfqpoint{0.018417in}{-0.069444in}}{\pgfqpoint{0.036082in}{-0.062127in}}{\pgfqpoint{0.049105in}{-0.049105in}}%
\pgfpathcurveto{\pgfqpoint{0.062127in}{-0.036082in}}{\pgfqpoint{0.069444in}{-0.018417in}}{\pgfqpoint{0.069444in}{0.000000in}}%
\pgfpathcurveto{\pgfqpoint{0.069444in}{0.018417in}}{\pgfqpoint{0.062127in}{0.036082in}}{\pgfqpoint{0.049105in}{0.049105in}}%
\pgfpathcurveto{\pgfqpoint{0.036082in}{0.062127in}}{\pgfqpoint{0.018417in}{0.069444in}}{\pgfqpoint{0.000000in}{0.069444in}}%
\pgfpathcurveto{\pgfqpoint{-0.018417in}{0.069444in}}{\pgfqpoint{-0.036082in}{0.062127in}}{\pgfqpoint{-0.049105in}{0.049105in}}%
\pgfpathcurveto{\pgfqpoint{-0.062127in}{0.036082in}}{\pgfqpoint{-0.069444in}{0.018417in}}{\pgfqpoint{-0.069444in}{0.000000in}}%
\pgfpathcurveto{\pgfqpoint{-0.069444in}{-0.018417in}}{\pgfqpoint{-0.062127in}{-0.036082in}}{\pgfqpoint{-0.049105in}{-0.049105in}}%
\pgfpathcurveto{\pgfqpoint{-0.036082in}{-0.062127in}}{\pgfqpoint{-0.018417in}{-0.069444in}}{\pgfqpoint{0.000000in}{-0.069444in}}%
\pgfpathlineto{\pgfqpoint{0.000000in}{-0.069444in}}%
\pgfpathclose%
\pgfusepath{stroke,fill}%
}%
\begin{pgfscope}%
\pgfsys@transformshift{2.513514in}{4.482978in}%
\pgfsys@useobject{currentmarker}{}%
\end{pgfscope}%
\end{pgfscope}%
\begin{pgfscope}%
\pgfpathrectangle{\pgfqpoint{0.100000in}{0.100000in}}{\pgfqpoint{4.700000in}{4.700000in}}%
\pgfusepath{clip}%
\pgfsetbuttcap%
\pgfsetroundjoin%
\definecolor{currentfill}{rgb}{0.831373,0.168627,0.043137}%
\pgfsetfillcolor{currentfill}%
\pgfsetlinewidth{1.003750pt}%
\definecolor{currentstroke}{rgb}{0.831373,0.168627,0.043137}%
\pgfsetstrokecolor{currentstroke}%
\pgfsetdash{}{0pt}%
\pgfsys@defobject{currentmarker}{\pgfqpoint{-0.069444in}{-0.069444in}}{\pgfqpoint{0.069444in}{0.069444in}}{%
\pgfpathmoveto{\pgfqpoint{0.000000in}{-0.069444in}}%
\pgfpathcurveto{\pgfqpoint{0.018417in}{-0.069444in}}{\pgfqpoint{0.036082in}{-0.062127in}}{\pgfqpoint{0.049105in}{-0.049105in}}%
\pgfpathcurveto{\pgfqpoint{0.062127in}{-0.036082in}}{\pgfqpoint{0.069444in}{-0.018417in}}{\pgfqpoint{0.069444in}{0.000000in}}%
\pgfpathcurveto{\pgfqpoint{0.069444in}{0.018417in}}{\pgfqpoint{0.062127in}{0.036082in}}{\pgfqpoint{0.049105in}{0.049105in}}%
\pgfpathcurveto{\pgfqpoint{0.036082in}{0.062127in}}{\pgfqpoint{0.018417in}{0.069444in}}{\pgfqpoint{0.000000in}{0.069444in}}%
\pgfpathcurveto{\pgfqpoint{-0.018417in}{0.069444in}}{\pgfqpoint{-0.036082in}{0.062127in}}{\pgfqpoint{-0.049105in}{0.049105in}}%
\pgfpathcurveto{\pgfqpoint{-0.062127in}{0.036082in}}{\pgfqpoint{-0.069444in}{0.018417in}}{\pgfqpoint{-0.069444in}{0.000000in}}%
\pgfpathcurveto{\pgfqpoint{-0.069444in}{-0.018417in}}{\pgfqpoint{-0.062127in}{-0.036082in}}{\pgfqpoint{-0.049105in}{-0.049105in}}%
\pgfpathcurveto{\pgfqpoint{-0.036082in}{-0.062127in}}{\pgfqpoint{-0.018417in}{-0.069444in}}{\pgfqpoint{0.000000in}{-0.069444in}}%
\pgfpathlineto{\pgfqpoint{0.000000in}{-0.069444in}}%
\pgfpathclose%
\pgfusepath{stroke,fill}%
}%
\begin{pgfscope}%
\pgfsys@transformshift{2.513514in}{1.057316in}%
\pgfsys@useobject{currentmarker}{}%
\end{pgfscope}%
\end{pgfscope}%
\begin{pgfscope}%
\pgfpathrectangle{\pgfqpoint{0.100000in}{0.100000in}}{\pgfqpoint{4.700000in}{4.700000in}}%
\pgfusepath{clip}%
\pgfsetbuttcap%
\pgfsetroundjoin%
\definecolor{currentfill}{rgb}{0.078431,0.317647,0.800000}%
\pgfsetfillcolor{currentfill}%
\pgfsetlinewidth{1.003750pt}%
\definecolor{currentstroke}{rgb}{0.078431,0.317647,0.800000}%
\pgfsetstrokecolor{currentstroke}%
\pgfsetdash{}{0pt}%
\pgfsys@defobject{currentmarker}{\pgfqpoint{-0.069444in}{-0.069444in}}{\pgfqpoint{0.069444in}{0.069444in}}{%
\pgfpathmoveto{\pgfqpoint{0.000000in}{-0.069444in}}%
\pgfpathcurveto{\pgfqpoint{0.018417in}{-0.069444in}}{\pgfqpoint{0.036082in}{-0.062127in}}{\pgfqpoint{0.049105in}{-0.049105in}}%
\pgfpathcurveto{\pgfqpoint{0.062127in}{-0.036082in}}{\pgfqpoint{0.069444in}{-0.018417in}}{\pgfqpoint{0.069444in}{0.000000in}}%
\pgfpathcurveto{\pgfqpoint{0.069444in}{0.018417in}}{\pgfqpoint{0.062127in}{0.036082in}}{\pgfqpoint{0.049105in}{0.049105in}}%
\pgfpathcurveto{\pgfqpoint{0.036082in}{0.062127in}}{\pgfqpoint{0.018417in}{0.069444in}}{\pgfqpoint{0.000000in}{0.069444in}}%
\pgfpathcurveto{\pgfqpoint{-0.018417in}{0.069444in}}{\pgfqpoint{-0.036082in}{0.062127in}}{\pgfqpoint{-0.049105in}{0.049105in}}%
\pgfpathcurveto{\pgfqpoint{-0.062127in}{0.036082in}}{\pgfqpoint{-0.069444in}{0.018417in}}{\pgfqpoint{-0.069444in}{0.000000in}}%
\pgfpathcurveto{\pgfqpoint{-0.069444in}{-0.018417in}}{\pgfqpoint{-0.062127in}{-0.036082in}}{\pgfqpoint{-0.049105in}{-0.049105in}}%
\pgfpathcurveto{\pgfqpoint{-0.036082in}{-0.062127in}}{\pgfqpoint{-0.018417in}{-0.069444in}}{\pgfqpoint{0.000000in}{-0.069444in}}%
\pgfpathlineto{\pgfqpoint{0.000000in}{-0.069444in}}%
\pgfpathclose%
\pgfusepath{stroke,fill}%
}%
\begin{pgfscope}%
\pgfsys@transformshift{4.053314in}{1.827216in}%
\pgfsys@useobject{currentmarker}{}%
\end{pgfscope}%
\end{pgfscope}%
\begin{pgfscope}%
\pgfpathrectangle{\pgfqpoint{0.100000in}{0.100000in}}{\pgfqpoint{4.700000in}{4.700000in}}%
\pgfusepath{clip}%
\pgfsetbuttcap%
\pgfsetroundjoin%
\definecolor{currentfill}{rgb}{0.078431,0.317647,0.800000}%
\pgfsetfillcolor{currentfill}%
\pgfsetlinewidth{1.003750pt}%
\definecolor{currentstroke}{rgb}{0.078431,0.317647,0.800000}%
\pgfsetstrokecolor{currentstroke}%
\pgfsetdash{}{0pt}%
\pgfsys@defobject{currentmarker}{\pgfqpoint{-0.069444in}{-0.069444in}}{\pgfqpoint{0.069444in}{0.069444in}}{%
\pgfpathmoveto{\pgfqpoint{0.000000in}{-0.069444in}}%
\pgfpathcurveto{\pgfqpoint{0.018417in}{-0.069444in}}{\pgfqpoint{0.036082in}{-0.062127in}}{\pgfqpoint{0.049105in}{-0.049105in}}%
\pgfpathcurveto{\pgfqpoint{0.062127in}{-0.036082in}}{\pgfqpoint{0.069444in}{-0.018417in}}{\pgfqpoint{0.069444in}{0.000000in}}%
\pgfpathcurveto{\pgfqpoint{0.069444in}{0.018417in}}{\pgfqpoint{0.062127in}{0.036082in}}{\pgfqpoint{0.049105in}{0.049105in}}%
\pgfpathcurveto{\pgfqpoint{0.036082in}{0.062127in}}{\pgfqpoint{0.018417in}{0.069444in}}{\pgfqpoint{0.000000in}{0.069444in}}%
\pgfpathcurveto{\pgfqpoint{-0.018417in}{0.069444in}}{\pgfqpoint{-0.036082in}{0.062127in}}{\pgfqpoint{-0.049105in}{0.049105in}}%
\pgfpathcurveto{\pgfqpoint{-0.062127in}{0.036082in}}{\pgfqpoint{-0.069444in}{0.018417in}}{\pgfqpoint{-0.069444in}{0.000000in}}%
\pgfpathcurveto{\pgfqpoint{-0.069444in}{-0.018417in}}{\pgfqpoint{-0.062127in}{-0.036082in}}{\pgfqpoint{-0.049105in}{-0.049105in}}%
\pgfpathcurveto{\pgfqpoint{-0.036082in}{-0.062127in}}{\pgfqpoint{-0.018417in}{-0.069444in}}{\pgfqpoint{0.000000in}{-0.069444in}}%
\pgfpathlineto{\pgfqpoint{0.000000in}{-0.069444in}}%
\pgfpathclose%
\pgfusepath{stroke,fill}%
}%
\begin{pgfscope}%
\pgfsys@transformshift{0.973713in}{3.713078in}%
\pgfsys@useobject{currentmarker}{}%
\end{pgfscope}%
\end{pgfscope}%
\begin{pgfscope}%
\pgfpathrectangle{\pgfqpoint{0.100000in}{0.100000in}}{\pgfqpoint{4.700000in}{4.700000in}}%
\pgfusepath{clip}%
\pgfsetrectcap%
\pgfsetroundjoin%
\pgfsetlinewidth{2.007500pt}%
\definecolor{currentstroke}{rgb}{0.000000,0.000000,0.000000}%
\pgfsetstrokecolor{currentstroke}%
\pgfsetdash{}{0pt}%
\pgfpathmoveto{\pgfqpoint{2.513514in}{2.597116in}}%
\pgfpathlineto{\pgfqpoint{2.513514in}{4.482978in}}%
\pgfusepath{stroke}%
\end{pgfscope}%
\begin{pgfscope}%
\pgfpathrectangle{\pgfqpoint{0.100000in}{0.100000in}}{\pgfqpoint{4.700000in}{4.700000in}}%
\pgfusepath{clip}%
\pgfsetrectcap%
\pgfsetroundjoin%
\pgfsetlinewidth{2.007500pt}%
\definecolor{currentstroke}{rgb}{0.000000,0.000000,0.000000}%
\pgfsetstrokecolor{currentstroke}%
\pgfsetdash{}{0pt}%
\pgfpathmoveto{\pgfqpoint{4.053314in}{1.827216in}}%
\pgfpathlineto{\pgfqpoint{4.053314in}{3.713078in}}%
\pgfusepath{stroke}%
\end{pgfscope}%
\begin{pgfscope}%
\pgfpathrectangle{\pgfqpoint{0.100000in}{0.100000in}}{\pgfqpoint{4.700000in}{4.700000in}}%
\pgfusepath{clip}%
\pgfsetrectcap%
\pgfsetroundjoin%
\pgfsetlinewidth{2.007500pt}%
\definecolor{currentstroke}{rgb}{0.000000,0.000000,0.000000}%
\pgfsetstrokecolor{currentstroke}%
\pgfsetdash{}{0pt}%
\pgfpathmoveto{\pgfqpoint{2.513514in}{2.597116in}}%
\pgfpathlineto{\pgfqpoint{4.053314in}{1.827216in}}%
\pgfusepath{stroke}%
\end{pgfscope}%
\begin{pgfscope}%
\pgfpathrectangle{\pgfqpoint{0.100000in}{0.100000in}}{\pgfqpoint{4.700000in}{4.700000in}}%
\pgfusepath{clip}%
\pgfsetrectcap%
\pgfsetroundjoin%
\pgfsetlinewidth{2.007500pt}%
\definecolor{currentstroke}{rgb}{0.000000,0.000000,0.000000}%
\pgfsetstrokecolor{currentstroke}%
\pgfsetdash{}{0pt}%
\pgfpathmoveto{\pgfqpoint{2.513514in}{4.482978in}}%
\pgfpathlineto{\pgfqpoint{4.053314in}{3.713078in}}%
\pgfusepath{stroke}%
\end{pgfscope}%
\begin{pgfscope}%
\pgfpathrectangle{\pgfqpoint{0.100000in}{0.100000in}}{\pgfqpoint{4.700000in}{4.700000in}}%
\pgfusepath{clip}%
\pgfsetrectcap%
\pgfsetroundjoin%
\pgfsetlinewidth{2.007500pt}%
\definecolor{currentstroke}{rgb}{0.000000,0.000000,0.000000}%
\pgfsetstrokecolor{currentstroke}%
\pgfsetdash{}{0pt}%
\pgfpathmoveto{\pgfqpoint{0.973713in}{1.827216in}}%
\pgfpathlineto{\pgfqpoint{0.973713in}{3.713078in}}%
\pgfusepath{stroke}%
\end{pgfscope}%
\begin{pgfscope}%
\pgfpathrectangle{\pgfqpoint{0.100000in}{0.100000in}}{\pgfqpoint{4.700000in}{4.700000in}}%
\pgfusepath{clip}%
\pgfsetrectcap%
\pgfsetroundjoin%
\pgfsetlinewidth{2.007500pt}%
\definecolor{currentstroke}{rgb}{0.000000,0.000000,0.000000}%
\pgfsetstrokecolor{currentstroke}%
\pgfsetdash{}{0pt}%
\pgfpathmoveto{\pgfqpoint{2.513514in}{1.057316in}}%
\pgfpathlineto{\pgfqpoint{2.513514in}{2.943178in}}%
\pgfusepath{stroke}%
\end{pgfscope}%
\begin{pgfscope}%
\pgfpathrectangle{\pgfqpoint{0.100000in}{0.100000in}}{\pgfqpoint{4.700000in}{4.700000in}}%
\pgfusepath{clip}%
\pgfsetrectcap%
\pgfsetroundjoin%
\pgfsetlinewidth{2.007500pt}%
\definecolor{currentstroke}{rgb}{0.000000,0.000000,0.000000}%
\pgfsetstrokecolor{currentstroke}%
\pgfsetdash{}{0pt}%
\pgfpathmoveto{\pgfqpoint{0.973713in}{1.827216in}}%
\pgfpathlineto{\pgfqpoint{2.513514in}{1.057316in}}%
\pgfusepath{stroke}%
\end{pgfscope}%
\begin{pgfscope}%
\pgfpathrectangle{\pgfqpoint{0.100000in}{0.100000in}}{\pgfqpoint{4.700000in}{4.700000in}}%
\pgfusepath{clip}%
\pgfsetrectcap%
\pgfsetroundjoin%
\pgfsetlinewidth{2.007500pt}%
\definecolor{currentstroke}{rgb}{0.000000,0.000000,0.000000}%
\pgfsetstrokecolor{currentstroke}%
\pgfsetdash{}{0pt}%
\pgfpathmoveto{\pgfqpoint{0.973713in}{3.713078in}}%
\pgfpathlineto{\pgfqpoint{2.513514in}{2.943178in}}%
\pgfusepath{stroke}%
\end{pgfscope}%
\begin{pgfscope}%
\pgfpathrectangle{\pgfqpoint{0.100000in}{0.100000in}}{\pgfqpoint{4.700000in}{4.700000in}}%
\pgfusepath{clip}%
\pgfsetrectcap%
\pgfsetroundjoin%
\pgfsetlinewidth{2.007500pt}%
\definecolor{currentstroke}{rgb}{0.000000,0.000000,0.000000}%
\pgfsetstrokecolor{currentstroke}%
\pgfsetdash{}{0pt}%
\pgfpathmoveto{\pgfqpoint{2.513514in}{2.597116in}}%
\pgfpathlineto{\pgfqpoint{0.973713in}{1.827216in}}%
\pgfusepath{stroke}%
\end{pgfscope}%
\begin{pgfscope}%
\pgfpathrectangle{\pgfqpoint{0.100000in}{0.100000in}}{\pgfqpoint{4.700000in}{4.700000in}}%
\pgfusepath{clip}%
\pgfsetrectcap%
\pgfsetroundjoin%
\pgfsetlinewidth{2.007500pt}%
\definecolor{currentstroke}{rgb}{0.000000,0.000000,0.000000}%
\pgfsetstrokecolor{currentstroke}%
\pgfsetdash{}{0pt}%
\pgfpathmoveto{\pgfqpoint{2.513514in}{4.482978in}}%
\pgfpathlineto{\pgfqpoint{0.973713in}{3.713078in}}%
\pgfusepath{stroke}%
\end{pgfscope}%
\begin{pgfscope}%
\pgfpathrectangle{\pgfqpoint{0.100000in}{0.100000in}}{\pgfqpoint{4.700000in}{4.700000in}}%
\pgfusepath{clip}%
\pgfsetrectcap%
\pgfsetroundjoin%
\pgfsetlinewidth{2.007500pt}%
\definecolor{currentstroke}{rgb}{0.000000,0.000000,0.000000}%
\pgfsetstrokecolor{currentstroke}%
\pgfsetdash{}{0pt}%
\pgfpathmoveto{\pgfqpoint{4.053314in}{1.827216in}}%
\pgfpathlineto{\pgfqpoint{2.513514in}{1.057316in}}%
\pgfusepath{stroke}%
\end{pgfscope}%
\begin{pgfscope}%
\pgfpathrectangle{\pgfqpoint{0.100000in}{0.100000in}}{\pgfqpoint{4.700000in}{4.700000in}}%
\pgfusepath{clip}%
\pgfsetrectcap%
\pgfsetroundjoin%
\pgfsetlinewidth{2.007500pt}%
\definecolor{currentstroke}{rgb}{0.000000,0.000000,0.000000}%
\pgfsetstrokecolor{currentstroke}%
\pgfsetdash{}{0pt}%
\pgfpathmoveto{\pgfqpoint{4.053314in}{3.713078in}}%
\pgfpathlineto{\pgfqpoint{2.513514in}{2.943178in}}%
\pgfusepath{stroke}%
\end{pgfscope}%
\begin{pgfscope}%
\definecolor{textcolor}{rgb}{0.000000,0.000000,0.000000}%
\pgfsetstrokecolor{textcolor}%
\pgfsetfillcolor{textcolor}%
\pgftext[x=2.321039in,y=4.194266in,left,base]{\color{textcolor}{\rmfamily\fontsize{20.000000}{24.000000}\selectfont\catcode`\^=\active\def^{\ifmmode\sp\else\^{}\fi}\catcode`\%=\active\def
\end{pgfscope}%
\begin{pgfscope}%
\definecolor{textcolor}{rgb}{0.000000,0.000000,0.000000}%
\pgfsetstrokecolor{textcolor}%
\pgfsetfillcolor{textcolor}%
\pgftext[x=1.358663in,y=3.284870in,left,base]{\color{textcolor}{\rmfamily\fontsize{20.000000}{24.000000}\selectfont\catcode`\^=\active\def^{\ifmmode\sp\else\^{}\fi}\catcode`\%=\active\def
\end{pgfscope}%
\begin{pgfscope}%
\definecolor{textcolor}{rgb}{0.000000,0.000000,0.000000}%
\pgfsetstrokecolor{textcolor}%
\pgfsetfillcolor{textcolor}%
\pgftext[x=1.551138in,y=3.616840in,left,base]{\color{textcolor}{\rmfamily\fontsize{20.000000}{24.000000}\selectfont\catcode`\^=\active\def^{\ifmmode\sp\else\^{}\fi}\catcode`\%=\active\def
\end{pgfscope}%
\begin{pgfscope}%
\definecolor{textcolor}{rgb}{0.000000,0.000000,0.000000}%
\pgfsetstrokecolor{textcolor}%
\pgfsetfillcolor{textcolor}%
\pgftext[x=3.860839in,y=1.966711in,left,base]{\color{textcolor}{\rmfamily\fontsize{20.000000}{24.000000}\selectfont\catcode`\^=\active\def^{\ifmmode\sp\else\^{}\fi}\catcode`\%=\active\def
\end{pgfscope}%
\begin{pgfscope}%
\definecolor{textcolor}{rgb}{0.000000,0.000000,0.000000}%
\pgfsetstrokecolor{textcolor}%
\pgfsetfillcolor{textcolor}%
\pgftext[x=2.321039in,y=1.389286in,left,base]{\color{textcolor}{\rmfamily\fontsize{20.000000}{24.000000}\selectfont\catcode`\^=\active\def^{\ifmmode\sp\else\^{}\fi}\catcode`\%=\active\def
\end{pgfscope}%
\begin{pgfscope}%
\definecolor{textcolor}{rgb}{0.000000,0.000000,0.000000}%
\pgfsetstrokecolor{textcolor}%
\pgfsetfillcolor{textcolor}%
\pgftext[x=2.898464in,y=1.485523in,left,base]{\color{textcolor}{\rmfamily\fontsize{20.000000}{24.000000}\selectfont\catcode`\^=\active\def^{\ifmmode\sp\else\^{}\fi}\catcode`\%=\active\def
\end{pgfscope}%
\end{pgfpicture}%
\makeatother%
\endgroup%